\documentclass{amsart}
\usepackage{amssymb}
\usepackage{booktabs}
\usepackage{mathrsfs}
\usepackage{psfrag}
\usepackage{xr-hyper}
\usepackage[colorlinks]{hyperref}


\title{Besov-ish spaces through  atomic decomposition}

\author[D.  Smania]{Daniel Smania}
\address{Departamento de Matem\'atica \\
  Instituto de Ci\^encias Matem\'aticas e de Computa\c{c}\~ao-Universidade de S\~ao Paulo (ICMC/USP) - S\~ao Carlos \\
               Caixa Postal 668 \\ S\~ao Carlos-SP \\ CEP 13560-970 \\ Brazil.}
\email{smania@icmc.usp.br} 
\urladdr{\url{https://sites.icmc.usp.br/smania/}}
\thanks{We would like the thank  the referee for the useful suggestions.  D.S. was partially supported by CNPq 306622/2019-0, CNPq 307617/2016-5, CNPq Universal 430351/2018-6 and FAPESP Projeto Tem\'atico 2017/06463-3.}

\newtheorem{theorem}{Theorem}[section]
\newtheorem{corollary}[theorem]{Corollary}

\newtheorem{lemma}[theorem]{Lemma}
\newtheorem{proposition}[theorem]{Proposition}

\theoremstyle{definition}
\newtheorem{definition}[theorem]{Definition}
\newtheorem{remark}[theorem]{Remark}

\renewcommand{\theequation}{\thesection.\arabic{equation}}
\newcommand\numberthis{\addtocounter{equation}{1}\tag{\theequation}} 

\newcommand{\eps}{\varepsilon}

\newcommand{\SB}{{\mathcal B}}

\newcommand{\SP}{{\mathcal P}}

\usepackage{constants} 
\newcommand{\secdot}[1]{\arabic{#1}}
\newconstantfamily{c}{
symbol=\lambda,
format=\secdot,
}
\renewconstantfamily{normal}{
symbol=C,
format=\secdot,
}

\newconstantfamily{e}{
symbol=\theta,
format=\arabic,
}

\newconstantfamily{A}{
symbol=A,
format=\arabic,
}

\newconstantfamily{G}{
symbol=G,
format=\arabic,
}

\newcommand{\Cll}[2][normal]{\Cl[#1]{#2}}
\newcommand{\Crr}[1]{\Cr{#1}}
\newcounter{change}

\usepackage[colorinlistoftodos, textwidth=4cm,disable]{todonotes}
\newcommand{\change}[2][]{\refstepcounter{change} \todo[linecolor=blue,backgroundcolor=blue!25,bordercolor=blue,size=tiny,#1]{Revision \thechange - #2}}
\newcommand{\changei}[2][]{\refstepcounter{change}\todo[inline,linecolor=blue,backgroundcolor=blue!25,bordercolor=blue,size=tiny,#1]{Revision \thechange - #2}}

\makeatletter
\providecommand\@dotsep{5}
\renewcommand{\listoftodos}[1][\@todonotes@todolistname]{%
  \@starttoc{tdo}{#1}}
\makeatother

\hypersetup{ colorlinks=true, pdftitle={Besovish spaces through  atomic decomposition}, pdfsubject={  atomic decomposition, Besov space}, pdfauthor={Daniel Smania},pdfkeywords={  atomic decomposition, Besov space, harmonic analysis, wavelets}}

\begin{document}

\begin{abstract} 
We use  the method of atomic decomposition to build  new families  of function spaces, similar to Besov spaces, in measure spaces with grids, a very mild assumption. Besov spaces with low regularity are considered in measure spaces with good grids, and  we obtain results on multipliers and left compositions in this setting. \end{abstract}

\maketitle

\setcounter{tocdepth}{3}
\tableofcontents



\section{Introduction}

\changei{We changed the abstract a little bit.}
Besov spaces $B^s_{p,q}(\mathbb{R}^n)$  were  introduced  by Besov \cite{besov}.  This scale of  spaces has been a favorite over  \change{we replaced  "thought" by  "over"} the years, with thousands  of  references available.  Perhaps two  of its most interesting features is that many earlier classes of function spaces appear \change{we replaced "appears" by "appear"} in this scale, as Sobolev spaces, and also that there are many equivalent ways to define $B^s_{p,q}(\mathbb{R}^n)$, in such way that you \change{we replaced "in such way you" by ïn such way that you"} can pick the more suitable one for your purpose. The reader may see Stein \cite{stein}, Peetre \cite{peetre},  Triebel \cite{tbook}  for an introduction of Besov spaces on $\mathbb{R}^n$. For a historical account on Besov spaces and related topics, see  Triebel \cite{tbook} and the shorter \change{Replace  "more short" by  "shorter"} but useful Jaffard \cite{jaffard}, Yuan,  Sickel and  Yang  \cite{ysy}  and Besov and Kalyabin \cite{besov2}.\\ 

In the last decades there was \change{we replaced "were" by "was"} a huge amount of activity on the generalisation  of  harmonic analysis (see Deng and Han \cite{deng}), and Besov spaces in particular,   to less regular {\it phase spaces}, replacing $\mathbb{R}^n$ by something with a poorer structure. It turns out that for {\it small $s > 0$} and $p,q\geq 1$, a proper definition demands  strikingly weak assumptions.  There is a large body of literature that provides a definition and properties of Besov spaces on {\it homogeneous spaces}, as defined by Coifman and Weiss \cite{cw}. Those are quasi-metric spaces with a doubling measure, which includes in particular Ahlfors regular metric spaces. We refer to the  \change{we replaced "pionner"by "pioneer"}pioneer work of Han and  Sawyer \cite{hs} and   Han,  Lu and Yang  \cite{han2} for Besov spaces  on homogeneous spaces and more recently  Alvarado and   Mitrea \cite{sharp} and Koskela, Yang, and Zhou \cite{koskela2} (in this case for metric measure spaces) and Triebel \cite{fractal}\cite{fractal2} for Besov spaces on fractals. There is a long list of examples of homogeneous spaces in Coifman  and Weiss \cite{examples}. The d-sets as  defined in Triebel \cite{fractal} are also examples of homogeneous spaces. 
\vspace{5mm}

\noindent {\it We give a very elementary (and yet practical\change{we replaced "pratical" by "practical"}) construction of   Besov  spaces $\mathcal{B}^s_{p,q}$, with $p,q\geq 1$ and $0 < s < 1/p$, for  measure spaces endowed with   a grid, that is, a sequence of finite partitions of measurable sets satisfying certain mild properties. }
\vspace{5mm}

This construction is close to the  martingale Besov spaces  as defined by Gu and Taibleson \cite{martingale}, however we  deal with "nonisotropic  splitting" in our grid without   applying the Gu-Taibleson recalibration procedure to our grid, which  simplifies the definition and it allows a broader \change{we replaced "wider" by "broader"} class of examples. \\

The \change{we replaced "main" by "primary"} primary  tool in this work is the concept of {\it atomic decomposition}.  An atomic decomposition represents \change{ we replaced "is the representation of"  by "represents"} each function in a space of functions  as a (infinite) linear combination of fractions of the so-called {\bf atoms}. The advantage of atomic decompositions is \change{we replaced "in" by "is"}  that the  atoms are functions that are  often far more regular than a typical element of  the space.  \change{we replaced "But" by "However"} However a distinctive feature (compared with  Fourier series with either Hilbert basis  or unconditional basis)  is that  such atomic decomposition is in general {\it not} unique.  Nevertheless, in a successful atomic decomposition of a normed space  of functions,\change{ comma added}  we can attribute  a "cost" to each possible representation, and the norm of the space is equivalent to the minimal cost (or infimum) among all representations.  A function represented by a single atom  has norm at most one, so the term "atom" seems to be appropriated. \\

 Coifman \cite{coifman} introduced the atomic decomposition of the  real Hardy space $H^p(\mathbb{R})$ and Latter \cite{latter} found an atomic decomposition for $H^p(\mathbb{R}^n).$ The influential work of   Frazier and Jawerth \cite{fj} gave us an atomic decomposition for the  Besov spaces $B^s_{p,q}(\mathbb{R}^n)$. In the context of homogeneous spaces, we have results by Han,  Lu and Yang  \cite{han2} on the \change{we included "the"} atomic decomposition of Besov spaces by H\"older atoms.  \\
 
 Closer to the spirit  of  this work we have the atomic decomposition of Besov space $B^s_{1,1}([0,1])$, with $s\in (0,1)$,  by de Souza \cite{souzao1}  using special atoms, that we call  {\bf Souza's atoms} (see also De Souza \cite{souzao2} and de Souza, O'Neil and Sampson  \cite{sn}). A Souza's atom $a_J$ on an   interval $J$ is quite simple, consisting of a  function whose support is $J$ and $a_J$ is {\it constant} on $J$. \\
 
 We also refer to the results on the B-spline atomic decomposition of the Besov space  of the unit cube of $\mathbb{R}^n$ in DeVore and  Popov \cite{spline2} (with  Ciesielski \cite{spline1} as a precursor), that in the case $0< s < 1/p$ reduces to an atomic decomposition by Souza's atoms, and the work of  Oswald \cite{oswald0} \cite{oswald} on finite element approximation in bounded polyhedral domains on $\mathbb{R}^n$.\\
 
On the  study  of Besov spaces on $\mathbb{R}^n$ and smooth manifolds,  Souza's atoms may seem to have  setbacks that restrict  its usefulness. They are not smooth, so it is fair to doubt  the effectiveness  of  atomic decomposition by Souza's atoms to obtain  a better understanding  of a  partial differential equation or  to represent  data  faithfully/without artifacts, a constant concern in applications of  smooth wavelets (see Jaffard, Meyer and Ryan \cite{jaffard2}). \\

On the other hand, in the study of   ergodic properties  of piecewise smooth dynamical systems, \change{comma inserted}\change{we deleted "the"} {\it transfer operators} play \change{we replaced "plays" by "play"}  a huge role. Those operators act \change{we replaced "acts" by "act"} on  Lebesgue spaces $L^1(m)$,  but they are more useful if one can show \change{we replaced "it has a" by "they have a"}  they have a  (good) action on more regular function spaces. Unfortunately, \change{we included a comma}   in many cases the transfer operator  does {\it not}  preserve neither smoothness nor even continuity of a function.  Discontinuities   are a fact of life in this setting, and they do not go away as in certain dissipative PDEs, since the positive $1$-eigenvectors of this operator, of utmost \change{repeated word "importance" erased}  importance in its  study, often have discontinuities themselves. The works of  Lasota and Yorke \cite{ly}  and Hofbauer and  Keller  \cite{hk3} are  landmark results in this direction. See  also Baladi \cite{bb} and Broise \cite{broise} for more details.   Atomic decomposition with Souza's atoms, being the simplest possible kind of atom with discontinuities, might come in handy  in these cases. That was  a major motivation for this work. \\

Besides this fact, in  an {\it abstract} homogeneous space higher-order \change{we replaced "higher order" by "higher-order"} smoothness does not seem to be a very useful concept \change{comma removed}  since  we can define $B^s_{p,q}$ just for small values of $s$, so atomic decompositions by  Souza's atoms sounds far more attractive.  \\

Indeed, the simplicity of Souza's atoms allows \change{we replaced "allow" by "allows"}  us to  throw away the necessity of a metric/topological structure on the phase space.  We define  {\bf Besov spaces} on every  measure space with a non-atomic \change{we replaced "non atomic" by "non-atomic"}  finite measure, provided we endowed it  with a {\it good grid}. A {\bf good grid}  is just a sequence of finite partitions of measurable sets satisfying certain mild properties.   We give the  definition of Besov space on measure space with a (good) grid in Part II. \\

In \change{we replaced "On" by "In"} the literature we usually see a known space of functions be described using atomic decomposition. This typically comes after a careful study of such space, and it is often a challenging  task. More rare is the path we follow here. We {\it start} by {\it defining }  the Besov spaces {\it through}  atomic decomposition by  Souza's atoms. This construction of the spaces and the study of its basic properties, as its completeness and \change{ we removed the extra "its"}  compact inclusion in Lebesgue spaces, uses fairly general and simple arguments, and it does not depend on the particular nature of the atoms used, except for very mild conditions on its regularity. In Part I we describe this construction in full generality. 
\vspace{5mm}

\noindent {\it We construct Besov-ish   spaces. There are far more general than Besov spaces.  In particular, \change{we added a comma}  the nature of the atoms may change with its location and scale in the phase space and the grid itself can be  very irregular.} 
\vspace{5mm}

The main result in Part I is Proposition \ref{trans}. Due to \change{we added "to"}  the generality of the setting, its  statement  is \change{ we replaced "hopeless" by " hopelessly"} hopelessly technical in nature, however this very powerful result has a simple meaning. Suppose we  have an atomic decomposition of a function. If we replace each of those atoms by a combination of atoms (possibly of a different class) in such way that we \change{we replaced "in such way we" by "in such way that we"} do not change the location of the support and also  the "mass" of the representation is concentrated \change{we replaced "concentrate" by "concentrated"} pretty much on the same original scale, then we obtain a new atomic decomposition, and the cost of this atomic decomposition can be estimated \change{we replaced "estimate" by "estimated"}  by the cost of the original atomic decomposition. We will use this result many times throughout \change{we replaced "along" by "throughout"}   this work. This is obviously  connected with  {\it almost diagonal operators}  as in   Frazier and Jawerth \cite{fj}\cite{fj2}.\\

In Part II we also offer a detailed analysis of the Besov spaces defined there. Since we define \change{we replaced "defined" by "define"}  it using combinations of Souza's atoms, it is not clear {\it a priori} how rich are those spaces. So 
\vspace{5mm}

\noindent {\it  We give a bunch of alternative characterisations of these Besov spaces.  We  show that using more  flexible classes of atoms (piecewise H\"older atoms, $p$-bounded variation atoms and even Besov atoms with higher regularity), we obtain the same Besov space. This often allows us to easily verify if  a given function belongs to $B^s_{p,q}$. We also have  a mean oscillation characterisation in the spirit of Dorronsoro \cite{mo} and Gu and Taibleson  \cite{martingale}, and we also obtained  another one using  Haar wavelets.}
\vspace{5mm}

  Haar wavelets were introduced by Haar  \cite{haar}  in the  real line.  The elegant construction of unbalanced Haar  wavelets in general measure spaces with a grid by  Girardi and Sweldens \cite{gw}  will play an essential role here. If in general homogeneous spaces the Calder\'on reproducing formula appears to be the bit of  harmonic analysis that survives in it and it allows  to carry out the theory, in finite measure spaces with a good grid (and in particular {\it compact } homogenous spaces) full-blown Haar systems are  alive and well. Recently  a Haar system was used by Kairema, Li, Pereyra and Ward \cite{ka} to study dyadic versions of the Hardy and BMO spaces in homogeneous spaces.  \\
 
 We also provided a few applications in part III. In particular, \change{we added a comma}  we study the boundedness of pointwise multipliers acting in the Besov space. Since it is effortless \change{we replaced "very easy" by "effortless"} to multiply Souza's atoms, the proofs of these results are concise \change{we replaced "very short" by "concise"} and easy to understand, generalising some of the results for Besov spaces in $\mathbb{R}^n$ by  Triebel\cite{multi} and Sickel \cite{sickel}. We also study the boundedness of left composition in  Besov spaces of measure spaces with a grid, similar to some results for $B^s_{p,q}(\mathbb{R}^n)$ in  Bourdaud and Kateb \cite{k1} (see also Bourdaud and Kateb \cite{k0}\cite{kateb1}).\\

It may come as a surprise to the reader  that Besov spaces on  compact homogeneous spaces as defined by  Han,  Lu and Yang  \cite{han2} (and in particular  Gu-Taibleson recalibrated martingale Besov spaces \cite{martingale})  are indeed {\it particular cases} of Besov spaces defined here, provided $0< s< 1/p$ and $s$  is small.  We show this in a forthcoming work \cite{smania-homo}.
\vspace{5mm}

\section{Notation} We will use $C_1, C_2, \dots...$ for positive constants and  $\lambda_1, \lambda_2, \dots$ for positive constants smaller than one.

\begin{table}[h]
  \centering
  \caption{Most important notation/symbols/constants}
  \label{tab:table1}
  \begin{tabular}{cc}
    \toprule
    Symbol & Description\\
    \midrule
    $I$ & phase space\\
     $m$ & finite measure in  the phase space $I$\\
     $a_P, b_P$ & an atom supported on P\\
     $\mathcal{A}$ & a class of atoms \\
     $\mathcal{A}(Q)$ & set of atoms of class $\mathcal{A}$ supported on $Q$ \\
     $\mathcal{A}^{sz}_{s,p}$& class of $(s,p)$-Souza's atoms \\
     $\mathcal{A}^{h}_{s,\beta,p}$& class of $(s,\beta,p)$-H\"older atoms \\
     $\mathcal{A}^{bv}_{s, \beta,p}$& class of $(s,\beta,p)$-bounded variation atoms \\
     $\mathcal{A}^{bs}_{s,\beta,p,q} $ & class of $(s,\beta,p,q)$-Besov's atoms \\
     $\mathcal{P}$ & grid of  $I$\\
    $\mathcal{P}^k$ & family  subsets of $I$ at the $k$-th level of $\mathcal{P}$\\
    $\Crr{menor} \leq \Crr{maior}$ & describes geometry of the grid $\mathcal{P}$\\
    $\mathcal{B}^s_{p,q}(\mathcal{A})$ & (s,p,q)-Banach space defined by the class of atoms $\mathcal{A}$\\
    $\mathcal{B}^s_{p,q}$ or  $\mathcal{B}^s_{p,q}(\mathcal{A}^{sz}_{s,p})$ & (s,p,q)-Banach space defined by Souza's atoms\\
    $P, Q, W$ & elements of the grid $\mathcal{P}$\\
    $L^p$ or $L^p(m)$ & Lesbesgue spaces of  $(I,\mathbb{A}, m)$ \\
    $|\cdot|_p$ & norm in $L^p$, $p \in (0,\infty].$ \\
    $p'$ & $1/p+ 1/p'=1$, with $p\in [1,\infty]$.  \\
    $\rho$ & $\min\{1,p,q\}$.\\
    $\hat{t}$  & $\max\{t,1\}.$\\
    \bottomrule
  \end{tabular}
\end{table}

\vspace{1cm}

\newpage 

\centerline{ \bf I. DIVIDE AND RULE.}
\addcontentsline{toc}{chapter}{\bf I. DIVIDE AND RULE.}
\vspace{1cm}

\fcolorbox{black}{white}{
\begin{minipage}{\textwidth}
\noindent   In Part I. we are going to assume  $s > 0$, $p \in (0,\infty)$ and $q\in (0,\infty].$ \end{minipage} 
}
\ \\ \\

\section{Measure spaces and grids}\label{partition}   Let $I$ be a measure space with a $\sigma$-algebra $\mathbb{A}$ and $m$ be a  measure on $(I,\mathbb{A})$, $m(I)<  \infty$. Given a measurable set $J\subset  I$ denote $|J|=m(J)$.  We denote the Lebesgue spaes of $(I,\mathbb{A},m)$ by $L^p$. A {\bf grid}  is a sequence of finite  families of measurable sets with positive measure  $\mathcal{P}= (\mathcal{P}^k)_{k\in \mathbb{N}}$, so that at least one of these families is non empty and
\begin{itemize}
\item[${\Cll[G]{fin}}$.] Given $Q\in \mathcal{P}^k$, let
$$\Omega_Q^k=\{ P \in \mathcal{P}^k\colon \ P\cap Q\neq \emptyset \}.$$
Then
$$\Cll{mult1}=\sup_k \sup_{Q \in \mathcal{P}^k} \# \Omega_Q^k< \infty.$$
\end{itemize}

  Define  $|\mathcal{P}^k|=\sup \{|Q|\colon Q \in \mathcal{P}^k\}$.  To simplify the notation we also assume that $P\neq Q$ for every $P\in \mathcal{P}^i$ and $Q\in \mathcal{P}^j$ satisfying $i\neq j$. We often abuse notation using $\mathcal{P}$ for both   $(\mathcal{P}^k)_{k\in \mathbb{N}}$ and $\cup_k \mathcal{P}^k$. 
 \begin{remark} 
  There are plenty of examples of spaces with  grids. Perhaps the simplest one is obtained considering  $[0,1)$ with the Lebesgue measure and the {\bf dyadic grid} $\mathcal{D}=(\mathcal{D}^k)_k$ given by
  $$\mathcal{D}^k=\{[i/2^k,(i+1)/2^k), \ 0\leq i < 2^k   \}.$$
  We can also consider the dyadic grid $\mathcal{D}_n=(\mathcal{D}^k_n)_k$ of $[0,1)^n$, endowed with the Lebesgue measure,  given by
   $$\mathcal{D}^k_n=\{J_1\times\cdots \times J_n, \ with \ J_i\in \mathcal{D}^k \}.$$
  and also the corresponding $d$-adic grids replacing $2^k$ by $d^k$ in the above definitions. The above grids are somehow special since they are nested sequence of partitions of the phase space $I$ and all elements on the same level have \change{we replaced "has" by "have"} exactly the same measure. 
  \end{remark}
  
    \begin{remark} Indeed, any  measure space with a finite non-atomic measure can be endowed with a grid  made of a nested sequence of partitions and such that all elements on the same level have \change{we replaced "has" by "have"}  \change{we replaced "exactly" by "precisely"} precisely the same measure \change{we removed a comma}  since any such measure space is measure-theoretically the same that a finite interval with the Lebesgue measure.
  \end{remark}
  
   \begin{remark} If we  consider a (quasi)-metric space $I$ with a finite measure $m$, we would like to construct "nice" grids on $(I,m)$. It turns out that if $(I,m)$ is a homogeneous space one can construct a nested sequence of partitions of the phase space $I$ and all elements on the same level are open subsets and have "essentially" the same measure. This is an easy consequence of   a remarkable \change{we replaced 'famous"by "remarkable"} result by Christ \cite{christ}. See \cite{smania-homo}.
   \end{remark}
  
  \begin{remark} One can constructs grids for smooth compact manifolds and bounded polyhedral domains in $\mathbb{R}^n$ using successive subdivisions of an initial triangulation of the domain (see for instance Oswald \cite{oswald0}\cite{oswald} ).
  \end{remark}


\section{A bag  of tricks.}

 Following closely the notation of Triebel \cite{fractal}, consider the set $\ell_q(\ell_p^{\mathcal{P}})$ of all indexed sequences $$x=(x_{P})_{P\in \mathcal{P}},$$ with $x_{P}\in \mathbb{C}$, satisfying  
$$|x|_{\ell_q(\ell_p^{\mathcal{P}})}= \Big( \sum_k \big(\sum_{P\in \mathcal{P}^k} |x_{P}|^p \big)^{q/p} \Big)^{1/q} < \infty,$$
with the usual modification when  $q=\infty$. Then  $(\ell_q(\ell_p^{\mathcal{P}}),|\cdot|_{\ell_q(\ell_p^{\mathcal{P}})})$ is a complex $\rho$-Banach space  with $\rho= \min \{1,p,q\}$, that is,  $d(x,y)=|x-y|_{\ell_q(\ell_p^{\mathcal{P}})}^\rho$ is a complete metric in $\ell_q(\ell_p^{\mathcal{P}}).$

The following is  a pair of arguments we will use across \change{we replaced "along" by "across"} this paper to estimate norms in $\ell_p$ and $\ell_q(\ell_p^{\mathcal{P}})$. Those are very elementary, and we do not claim any originality. We collect them here to simplify the exposition. The reader can skip this for the cases  $p, q \geq 1$, when the  results bellow reduce  to the familiar H\"older's and Young's inequalities. Their proofs  were mostly based  on  \cite[Proof of Theorem 3.1]{fj}. Recall that for $t \in (0,\infty]$ we defined $\hat{t}=\max \{ 1, t\}$.

\begin{proposition}[H\"older-like trick] \label{holder}Let $t\in (0, \infty)$ and $q\in (0,\infty]$. Let $a=(a_k)_k, b=(b_k)_k,c=(c_k)_k$ nonnegative   sequences  such that for every $k$
$$a_k^{1/\hat{t}} \leq C^{1/\hat{t}} b_k^{1/\hat{t}}c_k^{1/\hat{t}}.$$
Then if $q< \infty$ we have
$$(\sum_k a_k^{1/\hat{t}})^{\hat{t}/t} \leq  C^{1/t} \Cll{co}(t,q,b)\Big(  \sum_k c_k^{q/t}\Big)^{1/q},$$
and if  $q=\infty$ 
$$(\sum_k a_k^{1/\hat{t}})^{\hat{t}/t} \leq  C^{1/t} \Crr{co}(t,q,b)\sup_k c_k^{1/t}.$$
where  
\begin{itemize} 
\item[A.] If $t\geq 1$ and $q\geq 1$ then $\Crr{co}(t,q,b)=(\sum_k b_k^{q'/t})^{1/q'} $ if  $q > 1$, \\ and  $\Crr{co}(t,1,b)=\sup_k b_k^{1/t}$  if  $q=1$.
\item[B.] If $t\geq 1$ and $q\leq  1$ then $\Crr{co}(t,q,b)=\sup_k b_k^{1/t}$.
\item[C.] If $t< 1$ and $q/t\geq 1$ then $\Crr{co}(t,q,b)=(\sum_k b_k^{(q/t)'})^{1/(t(q/t)')} $ if $q< t$, \\ and 
$\Crr{co}(t,q,b)=\sup_k b_k^{1/t}$ if $q=t$. 
\item[D.] If $t< 1$ and $q/t<   1$ then $\Crr{co}(t,q,b)= \sup_k b_k^{1/t} $.
\end{itemize}

\end{proposition}  
\begin{proof}   We have 

\noindent {\it Case A.}  If $t\geq 1$ and $q\geq 1$, by the H\"older inequality for the pair $(q,q')$
$$\sum_{k} a_k^{1/\hat{t}}=\sum_k a_k^{1/t} \leq C^{1/t} (\sum_k b_k^{q'/t})^{1/q'}    \big(  \sum_k c_k^{q/t}\big)^{1/q}.$$

\noindent {\it Case B.} If $t\geq 1$ and $q\leq 1$ then   the triangular inequality for $|\cdot|^q$ implies
\begin{eqnarray*}
(\sum_{k} a_k^{1/\hat{t}})^q&=&(\sum_k a_k^{1/t})^q \leq C^{q/t} \big(   \sum_k b_k^{1/t}c_k^{1/t}\big)^q \\
&\leq& C^{q/t}   \sum_k \big(  b_k^{1/t}c_k^{1/t}\big)^q  \\
&\leq& C^{q/t}   (\sup_k b_k^{q/t})\Big( \sum_kc_k^{q/t}\Big)
\end{eqnarray*}

\noindent {\it Case C.} If $t< 1$ and $q/t \geq 1$ then $\hat{t}/t=1/t$ and by the  H\"older inequality for the pair $(q/t,(q/t)')$
$$\sum_{k} a_k^{1/\hat{t}}\leq C (\sum_k b_k^{(q/t)'})^{1/(q/t)'}    \big(  \sum_k c_k^{q/t}\big)^{t/q},$$

\noindent {\it Case D.} if $t<  1$ and $q/t<  1$ then  using the triangular inequality for $|\cdot|^{q/t}$
\begin{align*} (\sum_k a_k^{1/\hat{t}})^{q/t} &\leq C^{q/t} \big(\sum_k  b_k c_k\big)^{q/t}\\
&\leq C^{q/t} \sum_k  \big(b_k c_k\big)^{q/t}\\
&\leq C^{q/t} (\sup_k b_k^{q/t}) \Big(  \sum_k c_k^{q/t}\Big). \end{align*} 
\end{proof}

\begin{proposition}[Convolution  trick]\label{young} Let $p,q\in (0, \infty)$. Let $a=(a_k)_{k\in \mathbb{Z}}, b=(b_k)_{k\in \mathbb{Z}},c=(c_k)_{k\in \mathbb{Z}}\geq 0$ be such that for every $k$
$$a_k^{1/\hat{p}} \leq C^{1/\hat{p}} \sum_{i\in \mathbb{Z}} b_{k-i}^{1/\hat{p}}c_i^{1/\hat{p}}.$$
Then
$$(\sum_k a_k^{q/p})^{1/q} \leq  C^{1/p}    \Crr{co2}(p,q,b) \Big(  \sum_k c_k^{q/p}\Big)^{1/q},$$
where $\Cll{co2}\geq 1$ satisfies 
\begin{itemize} 
\item[A.] If $p\geq 1$ and $q\geq 1$ then $\Crr{co2}(p,q,b)= \sum_{k\in \mathbb{Z}} b_k^{1/p}$. 
\item[B.] If $p\geq 1$ and $q\leq  1$ then $\Crr{co2}(p,q,b)=(\sum_{k\in \mathbb{Z}} b_k^{q/p})^{1/q}$.
\item[C.] If $p< 1$ and $q/p\geq 1$ then $\Crr{co2}(p,q,b)=(\sum_{k\in \mathbb{Z}} b_k)^{1/p}$.
\item[D.] If $p< 1$ and $q/p<   1$ then $\Crr{co2}(p,q,b)=(\sum_{k\in \mathbb{Z}} b_k^{q/p})^{1/q}$
\end{itemize}

\end{proposition}  
\begin{proof} We have

\noindent {\it Case A.}  If $p\geq 1$ and $q\geq 1$, by the Young's inequality for the pair $(1,q)$ 
$$(\sum_k a_k^{q/p})^{1/q} \leq  C^{1/p}    \big( \sum_k b_k^{1/p}\big) \Big(  \sum_k c_k^{q/p}\Big)^{1/q}.$$

\noindent {\it Case B.} If $p\geq 1$ and $q\leq 1$ then   the triangular inequality for $|\cdot|^q$ and the Young's inequality for the pair $(1,1)$ imply 
\begin{eqnarray*}
\sum_{k} a_k^{q/p}&=& C^{q/p} \sum_k \big(  \sum_{i\in \mathbb{Z}} b_{k-i}^{1/p}c_i^{1/p} \big)^q \\
&\leq&C^{q/p}   \sum_k \sum_{i\in \mathbb{Z}} b_{k-i}^{q/p}c_i^{q/\hat{p}}  \\
&\leq& C^{q/p}   (\sum_k b_k^{q/p})\Big( \sum_k c_k^{q/p}\Big)
\end{eqnarray*}

\noindent {\it Case C.} If $p< 1$ and $q/p \geq 1$ then by the  Young's  inequality for the pair $(1, q/p)$
$$(\sum_{k} a_k^{q/p})^{p/q} \leq C (\sum_k b_k)    \big(  \sum_k c_k^{q/p}\big)^{p/q}.$$

\noindent {\it Case D.} If $p<  1$ and $q/p<  1$ then  using the triangular inequality for $|\cdot|^{q/p}$ and the Young's inequality for the pair $(1,1)$ 
\begin{eqnarray*} \sum_k a_k^{q/p}&\leq& C^{q/p} \sum_k \big(\sum_{i\in \mathbb{Z}}  b_{k-i} c_i\big)^{q/p}\\
&\leq& C^{q/p} \sum_k  \sum_{i\in \mathbb{Z}} b_{k-i}^{q/p}  c_i^{q/p}\\
&\leq& C^{q/p} (\sum_k b_k^{q/p}) \Big(  \sum_k c_k^{q/p}\Big).
\end{eqnarray*} 
\end{proof}

\section{Atoms}

 Let $\mathcal{P}$ be a grid.  Let $p \in [1,\infty), u \in [1,\infty]$,  and $s> 0$. A family of  {\bf atoms } associated with $\mathcal{P}$ of type $(s,p,u)$ is an indexed family $\mathcal{A}$ of pairs   $(\mathcal{B}(Q),\mathcal{A}(Q))_{Q\in \cup_k \mathcal{P}^k}$ where  \\
\begin{itemize}
\item[${\Cll[A]{banach}}$.] $\mathcal{B}(Q)$ is a complex Banach space contained in $L^{pu}$.
\item[${\Cll[A]{suporte}}$.] If $\phi \in \mathcal{B}(Q)$ then $\phi(x)=0$ for every $x \not\in Q$.
\item[${\Cll[A]{convex}}$.]  $\mathcal{A}(Q)$ is a convex subset of $\mathcal{B}(Q)$ such that $\phi \in \mathcal{A}(Q)$ if and only if $\sigma~\phi~\in~\mathcal{A}(Q)$ for every $\sigma \in \mathbb{C}$ satisfying $|\sigma|=1$.
\item[${\Cll[A]{atom}}$.] We have $$|\phi|_{pu} \leq |Q|^{s-\frac{1}{u'p}}$$ for every $\phi \in \mathcal{A}(Q)$. 
\end{itemize}

 We will say that $\phi \in \mathcal{A}(Q)$ is an $\mathcal{A}$-atom of type $(s,p,u)$ supported on $Q$ and that $\mathcal{B}(Q)$ is  the local Banach space on $Q$. Sometimes we also assume \\
\begin{itemize}
\item[${\Cll[A]{compact}}$.]  For every $Q\in \mathcal{P}$ we have that $\mathcal{A}(Q)$ is a compact subset in the strong topology of    $L^p$.\\
\end{itemize}
or
\begin{itemize}
\item[${\Cll[A]{compactw}}$.] We have $p\in [1,\infty)$ and  every $Q\in \mathcal{P}$ the set  $\mathcal{A}(Q)$ is a sequentially compact subset in the {\it weak } topology   of $L^p$.\\
\end{itemize}
or even \\
\begin{itemize}
\item[${\Cll[A]{finite}}$.]  For every $Q\in \mathcal{P}$ we have that $\mathcal{B}(Q)$ is finite dimensional and  $\mathcal{A}(Q)$ contains a neighborhood of $0$ in $\mathcal{B}(Q)$.\\
\end{itemize}

 We provide examples of classes of atoms in Section \ref{secatom}.

\section{Besov-ish spaces} 

 Let $p \in (0,\infty)$, $u \in [1,\infty]$, $q \in (0,\infty]$, $s> 0$,  $\mathcal{P}=(\mathcal{P}^k)_{k\geq 0}$ be a grid  and let $\mathcal{A}$ be a family of atoms of type $(s,p,u)$.  We will also assume that  
\begin{itemize}
\item[${\Cll[G]{sum}}.$] We have
$$\Cll{decaypart} =    \Crr{co}(p,q, (|\mathcal{P}^k|^s)_k) < \infty.$$
and
$$\lim_k |\mathcal{P}^k|=0.$$
\end{itemize}
The {\bf Besov-ish space} $\mathcal{B}^s_{p,q}(I,\mathcal{P}, \mathcal{A})$  is  the set of all complex valued functions $g \in L^p$ that  can  be represented by an absolutely convergent series on $L^p$
\begin{equation} \label{rep} g = \sum_{k=0}^{\infty} \sum_{Q \in \mathcal{P}^k}    s_Q a_Q\end{equation}
where $a_Q$ is in $\mathcal{A}(Q)$, $s_Q \in \mathbb{C}$  and with finite {\bf cost }
\begin{equation} \label{rep2} \big( \sum_{k=0}^{\infty} (\sum_{Q \in \mathcal{P}^k}    |s_Q|^p)^{q/p} \big)^{1/q} < \infty.\end{equation}
Note that the inner sum in (\ref{rep}) is finite. By absolutely convergence in $L^p$ we mean that 
$$\sum_{k=0}^{\infty} \big| \sum_{Q \in \mathcal{P}^k}    s_Q a_Q  \big|_p^{p/\hat{p}} < \infty.$$ 
The series in (\ref{rep}) is called a {\bf $\mathcal{B}^s_{p,q}(I,\mathcal{P}, \mathcal{A})$-representation} of the function $g$. 
Define
$$|g|_{\mathcal{B}^s_{p,q}(I,\mathcal{P}, \mathcal{A})} = \inf  \big( \sum_{k=0}^{\infty} (\sum_{Q \in \mathcal{P}^k}     |s_Q|^p )^{q/p} \big)^{1/q} ,$$
where the infimum runs over all possible representations of $g$ as in (\ref{rep}).  

Quite often along this work, when it is obvious the choice of the measure space $I$ and/or the grid $\mathcal{P}$  we will write either $\mathcal{B}^s_{p,q}(\mathcal{P}, \mathcal{A})$ or even $\mathcal{B}^s_{p,q}( \mathcal{A})$ instead of $\mathcal{B}^s_{p,q}(I,\mathcal{P}, \mathcal{A})$.  Whenever we write just $\mathcal{B}^s_{p,q}$ it means that we choose the Souza's atoms $\mathcal{A}^{sz}_{s,p}$, with parameters $s$ and $p$,  a class of atoms  we  properly define in Section \ref{souzaa}. 



\begin{proposition}\label{lp}   Assume ${\Crr{fin}}$-${\Crr{sum}}$ and ${\Crr{banach}}$-${\Crr{atom}}$. Let $t \in (0,\infty)$ be such that 
\begin{equation}\label{c11} s-\frac{1}{p}+\frac{1}{t} \geq 0, \ p \leq t \leq pu\end{equation} 
and suppose
\begin{equation}\label{decay}  \Cll{kt}=\Crr{mult1}^{1+1/t}  \Crr{co}(t,q, ( |\mathcal{P}^k|^{t(s-\frac{1}{p} +\frac{1}{t})} )_k ) < \infty,\end{equation} 
Then for every coefficients $s_Q$ satisfying (\ref{rep2}) and every $\mathcal{A}$-atoms $a_Q$  on $Q$ the series (\ref{rep}) converges absolutely on $L^t$. Indeed 
\begin{align}\label{flp}
|\sum_{Q \in \mathcal{P}^k}   s_Q a_Q|_t&\leq  \Crr{mult1}^{1+1/t}   |\mathcal{P}^k|^{s-\frac{1}{p} +\frac{1}{t}} \Big(  \sum_{Q \in \mathcal{P}^k} |s_Q|^p   \Big)^{1/p}
\end{align} 
and
\begin{equation}\label{inclulp}  |g|_{t} \leq  \Crr{kt} | \phi|_{\mathcal{B}^s_{p,q}(\mathcal{A})}.\end{equation}
\end{proposition}
\begin{proof}  Firstly note that if $p\leq t \leq  pu$
\begin{align*}  |a_P|^t_t &= \int |a_P(x)|^t  1_P\ dm(x)\\
&\leq  ||a_P|^t|_{\frac{pu}{t}}  |1_P|_{\frac{pu}{pu-t}} \\
&\leq  |a_P|_{pu}^t |P|^{\frac{pu-t}{pu}}\\
&\leq |P|^{ts-\frac{t}{u'p}} |P|^{\frac{pu-t}{pu}} =|P|^{t(s-\frac{1}{p} +\frac{1}{t})}.\end{align*} 
Consequently
\begin{align}
|\sum_{Q \in \mathcal{P}^k}   s_Q a_Q|_t^{t}&\leq   \int |\sum_{Q \in \mathcal{P}^k}   s_Q a_Q|^t \ dm\nonumber \\
&\leq  \sum_{Q \in \mathcal{P}^k}  \int_Q |\sum_{P \in \mathcal{P}^k}   s_P a_P|^t  \ dm \nonumber \\
&\leq   \sum_{Q \in \mathcal{P}^k}  \int_Q |\sum_{P \in \Omega_Q^k}   s_P a_P|^t \ dm \nonumber \\
&\leq  \Crr{mult1}^{t} \sum_{Q \in \mathcal{P}^k}  \int_Q  \sum_{P \in \Omega_Q^k}  |s_P a_P|^t \ dm  \nonumber \\
&\leq  \Crr{mult1}^{t}  \sum_{Q \in \mathcal{P}^k}  \sum_{P \in \Omega_Q^k}  \int   |s_P a_P|^t \ dm \nonumber \\
&\leq \Crr{mult1}^{t}\sum_{Q \in \mathcal{P}^k}  \sum_{P \in \Omega_Q^k} \int  |s_P|^t |a_P|^t_t \ dm \ \nonumber \\
&\leq \Crr{mult1}^{t}   \sum_{Q \in \mathcal{P}^k} \sum_{P \in \Omega_Q^k}   |P|^{t(s-\frac{1}{p} +\frac{1}{t})} |s_P|^t  \nonumber \\
&\leq  \Crr{mult1}^{t+1}  \sum_{P \in \mathcal{P}^k}  |P|^{t(s-\frac{1}{p} +\frac{1}{t})} |s_P|^t   \ \nonumber \\
&\leq \Crr{mult1}^{t+1}  |\mathcal{P}^k|^{t(s-\frac{1}{p} +\frac{1}{t})} \sum_{P \in \mathcal{P}^k} |s_P|^t  \ \nonumber 
\end{align} 
By Proposition \ref{holder} (H\"older-like trick) and $p\leq t$ we have 
\begin{align}
|\sum_k \sum_{Q \in \mathcal{P}^k}   s_Q a_Q|_t &\leq   \Big( \sum_k\big( |\sum_{Q \in \mathcal{P}^k}   s_Q a_Q|_t^{t}\big)^{1/\hat{t}}   \Big)^{\hat{t}/t} \nonumber \\ 
&\leq  \Crr{mult1}^{1+\frac{1}{t}}  \Crr{co}(t,q, ( |\mathcal{P}^k|^{t(s-\frac{1}{p} +\frac{1}{t})} )_k ) \Big(  \sum_k \big( \sum_{P \in \mathcal{P}^k} |s_P|^t  \big)^{q/t}\Big)^{1/q}. \nonumber  \\
&\leq  \Crr{mult1}^{1+\frac{1}{t}}  \Crr{co}(t,q, ( |\mathcal{P}^k|^{t(s-\frac{1}{p} +\frac{1}{t})} )_k ) \Big(  \sum_k \big( \sum_{P \in \mathcal{P}^k} |s_P|^p \big)^{q/p}\Big)^{1/q}. \nonumber 
\end{align} 
\end{proof}
\begin{remark}\label{sharp} Note that due to \change{we added "to"}  ${\Crr{sum}}$ if  $t=p$ then  $\Crr{kt} <\infty$.  Sometimes it is convenient to use sharper estimates than (\ref{flp}) and (\ref{inclulp}) replacing  $|\mathcal{P}^k|$ by the sequence 
$$C^k =\max \{ |Q|\colon \ Q\in \mathcal{P}^k, \ s_Q\neq 0\}.$$
For instance, if $s_Q=0$ for every $Q\in \mathcal{P}^k$ with $k\leq N$, then we can replace $ \Crr{co}(t,q, ( |\mathcal{P}^k|^{t(s-\frac{1}{p} +\frac{1}{t})} )_k ) $ by  $\Crr{co}(t,q, ( |\mathcal{P}^k|^{t(s-\frac{1}{p} +\frac{1}{t})}1_{(N,\infty)}(k) )_k ) $ in (\ref{decay}). Here $1_{(N,\infty)}$ denotes  the indicator function of the set $(N,\infty)$.
\end{remark}

\begin{proposition}  Assume ${\Crr{fin}}$-${\Crr{sum}}$ and ${\Crr{banach}}$-${\Crr{atom}}$. Then  $\mathcal{B}^s_{p,q}(\mathcal{A})$ is a complex  linear space and $|\cdot|_{\mathcal{B}^s_{p,q}(\mathcal{A})}$ is a $\rho$-norm, with $\rho=\min\{1,p,q\}$. Moreover the linear inclusion
$$\imath\colon  (\mathcal{B}^s_{p,q}(\mathcal{A}),|\cdot|_{\mathcal{B}^s_{p,q}(\mathcal{A})})\rightarrow (L^p,|\cdot|_p)$$
is continuous.  \end{proposition} 
\begin{proof} Le $f , g  \in \mathcal{B}^s_{p,q}(\mathcal{A})$. Then there are $\mathcal{B}^s_{p,q}(\mathcal{A})$-representations 
$$ f = \sum_{k=0}^{\infty} \sum_{Q \in \mathcal{P}^k}    s_Q' a_Q'   \  and \ g = \sum_{k=0}^{\infty} \sum_{Q \in \mathcal{P}^k}    s_Q a_Q.$$
Let $sgn \ 0 =0$ and $sgn \ z =z/|z|$ otherwise. Of course
\begin{equation} \label{quo}  \sum_{Q \in \mathcal{P}^k}    c_Q b_Q = \sum_{Q \in \mathcal{P}^k}    s_Q' a_Q'  + \sum_{Q \in \mathcal{P}^k}    s_Q a_Q,\end{equation} 
where\footnote{We don't need to worry so much if  $|s_Q| + |s_Q'|=0$, since in this case  $c_Q=0$ and we can choose $b_Q$ to be an arbitrary atom (for instance $a_Q$) in such way that (\ref{quo}) holds.  For this reason we are not going to explicitly deal with similar situations ( that is going to appear quite often) along the paper.} 
$$b_Q =   \frac{|s_Q'|}{|s_Q| + |s_Q'|} sgn(s_Q') a_Q'     +  \frac{|s_Q|}{|s_Q| + |s_Q'|}  sgn(s_Q') a_Q,$$
and
$$c_Q = |s_Q'| + |s_Q|.$$
Note that $sgn(s_Q') a_Q', sgn(s_Q) a_Q$ are atoms due to \change{we added "to"}  ${\Crr{convex}}$. So  by ${\Crr{convex}}$  we have that $b_Q$ is also an atom, since it  is a convex combination of atoms. In particular 
$$\sum_k  \sum_{Q \in \mathcal{P}^k}    c_Q b_Q $$
converges absolutely in $L^p$ to $f+g$. It remains to prove that this is a $\mathcal{B}^s_{p,q}(\mathcal{A})$-representation of $f+g$. Indeed
$$\big( \sum_{k=0}^{\infty} (\sum_{Q \in \mathcal{P}^k}    |c_Q|^p)^{q/p} \big)^{\rho/q} \leq \big( \sum_{k=0}^{\infty} (\sum_{Q \in \mathcal{P}^k}    |s_Q'|^p)^{q/p} \big)^{\rho/q} + \big( \sum_{k=0}^{\infty} (\sum_{Q \in \mathcal{P}^k}    |s_Q|^p)^{q/p} \big)^{\rho/q}.$$
Taking the infimum over all possible $\mathcal{B}^s_{p,q}(\mathcal{A})$-representations of $f$ and $g$ we obtain
$$|f+g|_{\mathcal{B}^s_{p,q}(\mathcal{A})}^\rho \leq  |f|_{\mathcal{B}^s_{p,q}(\mathcal{A})}^\rho  + |g|_{\mathcal{B}^s_{p,q}(\mathcal{A})}^\rho.$$
The identity $ |\gamma f|_{\mathcal{B}^s_{p,q}(\mathcal{A})}=  |\gamma| |f|_{\mathcal{B}^s_{p,q}(\mathcal{A})}$ is obvious.  By Proposition \ref{lp} we have that  if $|f|_{\mathcal{B}^s_{p,q}(\mathcal{A})}=0$ then $|f|_p=0$, so $f=0$, so $|\cdot|_{\mathcal{B}^s_{p,q}(\mathcal{A})}$ is a $\rho$-norm, moreover  (\ref{inclulp}) tell us that $\imath$ is continuous. 
\end{proof}
\begin{proposition}\label{compa2} Assume ${\Crr{fin}}$-${\Crr{sum}}$ and ${\Crr{banach}}$-${\Crr{atom}}$. Suppose that $g_n$ are functions in $\mathcal{B}^s_{p,q}(\mathcal{A})$ with  $\mathcal{B}^s_{p,q}(\mathcal{A})$-representations 
$$g_n=\sum_{k=0}^{\infty}\sum_{Q\in \SP^k}s_Q^na_Q^n,$$
where $a_Q^n$ is a $\mathcal{A}$-atom supported on $Q$, satisfying 
\begin{itemize}
\item[i.] There is $C$ such that for every $n$ 
\begin{equation}\label{estt}  (\sum_{k=0}^{\infty}(\sum_{Q\in \SP^k}|s_Q^n|^p)^{q/p})^{1/q}\leq C.\end{equation} 
\item[ii.] For every $Q\in \mathcal{P}$ we have that $s_Q=\lim_n s_Q^n$ exists. 
\item[iii.] For every $Q\in \mathcal{P}$ there is  $a_Q\in \mathcal{A}(Q)$ such that 
\begin{enumerate}
\item  either the sequence $a_Q^n$ converges to $a_Q$  in the strong topology  of $L^p$, or
\item we have  $p\in [1,\infty)$ and $a_Q^n$ weakly converges to  $a_Q$.
\end{enumerate}
\end{itemize}
then $g_n$ either strongly or weakly  converges in $L^p$, respectively,  to $g \in \mathcal{B}^s_{p,q}(\mathcal{A})$,  where  $g$ has  the   $\mathcal{B}^s_{p,q}(\mathcal{A})$-representation
\begin{equation}\label{repp} g=\sum_{k=0}^{\infty}\sum_{Q\in \SP^k}s_Q a_Q\end{equation} 
that satisfies 
\begin{equation}\label{estt2} (\sum_{k=0}^{\infty}(\sum_{Q\in \SP^k}|s_Q^n|^p)^{q/p})^{1/q}\leq C\end{equation} 
\end{proposition}
\begin{proof}By (\ref{estt}) it follows that (\ref{estt2}) holds and that (\ref{repp}) is indeed a $\mathcal{B}^s_{p,q}(\mathcal{A})$-representation of a function $g$. It remains to prove that $g_n$ converges to $g$ in $L^p$ in the topology under consideration. Given $\epsilon > 0$,  fix  $N$ large enough such that 
$$\Crr{mult1}^{1+1/p}  \Crr{co}(p,q, ( |\mathcal{P}^k|^{ps}1_{[N,\infty)}(k) )_k )   (2C^\rho+1)^{1/\rho}< (\epsilon/2)^{\hat{p}/p}.$$
We can write
$$g_n-g= \sum_{k\leq N} \sum_{Q\in \SP^k} (s_Q^na_Q^n - s_Qa_Q) + \sum_{k> N} \sum_{Q\in \SP^k}  c_Q^n b_Q^n,$$
where
$$b_Q^n =   \frac{|s_Q^n|}{|s_Q^n| + |s_Q|} sgn(s_Q^n) a_Q^n     +  \frac{|s_Q|}{|s_Q^n| + |s_Q|}  sgn(-s_Q) a_Q,$$
is an atom in $\mathcal{A}(Q)$, and 
$$c_Q^n = |s_Q^n| + |s_Q|.$$
Note that  the series in the r.h.s. converges absolutely  in $L^p$. Of course
$$ (\sum_{k=0}^{\infty}(\sum_{Q\in \SP^k}|c_Q^n|^p)^{q/p})^{1/q}\leq (2C^\rho+1)^{1/\rho},$$
So by (\ref{inclulp})  in Proposition \ref{lp} (see also Remark \ref{sharp}) we have 
$$
|\sum_{k> N} \sum_{Q \in \mathcal{P}^k}   c_Q^n b_Q^n|_p\leq \Crr{mult1}^{1+1/p}  \Crr{co}(p,q, ( |\mathcal{P}^k|^{ps}1_{[N,\infty)}(k) )_k )   (2C^\rho+1)^{1/\rho}  < (\epsilon/2)^{\hat{p}/p}. 
$$
In the case $ii.1$, note that if $n$  is large enough  then 
$$| \sum_{k\leq N} \sum_{Q\in \SP^k} (s_Q^na_Q^n - s_Qa_Q)|_p^{p/\hat{p}}< \epsilon/2,$$
and consequently $|g-g_n|_p^{p/\hat{p}} < \epsilon$.  So $g_n$ strongly converges to $g$. 

\noindent In the case $ii.2$, given $\phi\in (L^p)^\star$, with $p\geq 1$ we have that for $n$  large enough
$$|\phi(\sum_{k\leq N} \sum_{Q\in \SP^k} (s_Q^na_Q^n - s_Qa_Q))|\leq |\phi|_{(L^p)^\star} \epsilon/2,$$
and of course
$$
|\phi(\sum_{k> N} \sum_{Q \in \mathcal{P}^k}   c_Q^n b_Q^n|_p)| \leq  |\phi|_{(L^p)^\star}  \epsilon/2,
$$
so $g_n$ weakly converges to $g$ in $L^p$. 
\end{proof}

\begin{corollary}\label{compa1}  Assume ${\Crr{fin}}$-${\Crr{sum}}$ and ${\Crr{banach}}$-${\Crr{atom}}$, and 
\begin{enumerate}
\item either ${\Crr{compact}}$ or
\item we have $p\geq 1$ and  ${\Crr{compactw}}$.
\end{enumerate}
Then
\begin{itemize}
\item[i.] Let $g_n \in \mathcal{B}^s_{p,q}(\mathcal{A})$ be such that $|g_n|_{\mathcal{B}^s_{p,q}(\mathcal{A})}\leq C$ for every $n$. Then there is a subsequence that converges either strongly or weakly in $L^p$, respectively,  to some $g \in \mathcal{B}^s_{p,q}(\mathcal{A})$ with $|g|_{\mathcal{B}^s_{p,q}(\mathcal{A})}\leq~C$.
\item[ii.] In both cases $(\mathcal{B}^s_{p,q}(\mathcal{A}), |\cdot|_{\mathcal{B}^s_{p,q}(\mathcal{A})})$ is a complex $\rho$-Banach space, with $\rho=\min\{1,p,q\}$, 
\item[iii] If ${\Crr{compact}}$ holds then  the inclusion 
$$\imath \colon (\mathcal{B}^s_{p,q}(\mathcal{A}), |\cdot|_{\mathcal{B}^s_{p,q}(\mathcal{A})}) \rightarrow (L^p, |\cdot|_p)$$
is a compact linear inclusion. 
\end{itemize} 
\end{corollary} 
\begin{proof}[ Proof of i.] There are $\mathcal{B}^s_{p,q}(\mathcal{A})$-representations
$$g_n=\sum_{k=0}^{\infty}\sum_{Q\in \SP^k}s_Q^na_Q^n,$$
where $a_Q^n$ is a $\mathcal{A}$-atom supported on $Q$ and
\begin{equation} \label{sed0}  (\sum_{k=0}^{\infty}(\sum_{Q\in \SP^k}|s_Q^n|^p)^{q/p})^{1/q}\leq C+\eps_n,\end{equation}
and  $1\geq \eps_n\to 0$.
In particular, $|s_Q^n|\leq C+1$. Since the set $\cup_k \mathcal{P}^k$ is countable,  by the Cantor  diagonal argument, taking a subsequence we can assume  that $s_Q^n\to_n s_Q$ for some $s_Q \in \mathbb{C}$.  Due to \change{we added "to"}  ${\Crr{compact}}$ (${\Crr{compactw}}$) and the Cantor  diagonal argument, we can suppose that  $a_Q^n$ strongly (weakly) converges  in $L^p$ to  some $a_Q \in \mathcal{A}(Q)$. We set 
$$g=\sum_{k=0}^{\infty}\sum_{Q\in \SP^k}s_Qa_Q.$$
By Proposition \ref{compa2}  we conclude that  $g\in \mathcal{B}^s_{p,q}(\mathcal{A})$ with  $|g|_{\mathcal{B}^s_{p,q}(\mathcal{A})}\leq C$, and that $g_n$ converges to $g$ in $L^p$.
\end{proof}
\begin{proof}[Proof of ii.] Let $g_n$ be a Cauchy sequence on $\mathcal{B}^s_{p,q}(\mathcal{A})$. By Proposition \ref{lp} we have that $g_n$ is also a Cauchy sequence in $L^p$. Let $g$ be its limit in $L^p$. By Corollary   \ref{compa1}.i  have that $g \in \mathcal{B}^s_{p,q}(\mathcal{A})$. Note that for large $m$ and $n$  
$$|g_n-g_m|_{\mathcal{B}^s_{p,q}(\mathcal{A})}\leq \epsilon,$$
and $g_n - g_m$ converges to $g_n-g$ in $L^p$, so again by Corollary  \ref{compa1}.i  we have that 
$$|g_n-g|_{\mathcal{B}^s_{p,q}(\mathcal{A})}\leq \epsilon,$$
so $g_n$ converges to $g$ in $\mathcal{B}^s_{p,q}(\mathcal{A})$.
\end{proof}

\begin{proof}[Proof of iii.]  It follows from i. \end{proof}

The proof of the following result \change{we replaced "resukt" by "result"}  is quite similar. 
\begin{corollary}\label{compa12}  Assume ${\Crr{fin}}$-${\Crr{sum}}$ and ${\Crr{banach}}$-${\Crr{atom}}$, and 
\begin{enumerate}
\item either ${\Crr{compact}}$ or
\item we have $p\geq 1$ and  ${\Crr{compactw}}$.
\end{enumerate}
Then for every $f\in \mathcal{B}^s_{p,q}(\mathcal{A})$ there is a $\mathcal{B}^s_{p,q}(\mathcal{A})$-representation
$$f=\sum_k \sum_{P\in \mathcal{P}^k}   c_P a_P$$
such that 
$$|f|_{\mathcal{B}^s_{p,q}(\mathcal{A})}=\Big( \sum_k \big(\sum_{P\in \mathcal{P}^k}  |c_P|^p\big)^{q/p} \Big)^{1/q}.$$
\end{corollary}

We refer to Edmunds and  Triebel \cite{et} for more information on compact linear transformations between quasi-Banach spaces.

\begin{corollary} Assume ${\Crr{fin}}$-${\Crr{sum}}$ and ${\Crr{banach}}$-${\Crr{atom}}$. If  for every $Q\in \mathcal{P}$ we have that $\mathcal{B}(Q)$ is finite-dimensional and $\mathcal{A}(Q)$ is a closed subset of $\mathcal{B}(Q)$    then 
 $(\mathcal{B}^s_{p,q}(\mathcal{A}), |\cdot|_{\mathcal{B}^s_{p,q}(\mathcal{A})})$ is a $\rho$-Banach space, with $\rho=\min\{1,p,q\}$.
\end{corollary}
\begin{proof} Since all norms are equivalent in $\mathcal{B}(Q)$ we have that ${\Crr{atom}}$ implies that $\mathcal{A}(Q)$ is a closed and bounded subset of $\mathcal{B}(Q)$, so it is compact. By Corollary \ref{compa1}.ii it follows that $(\mathcal{B}^s_{p,q}(\mathcal{A}), |\cdot|_{\mathcal{B}^s_{p,q}(\mathcal{A})})$ is a $\rho$-Banach space.
\end{proof}

\section{Scales of spaces} 

\label{escala} Note that a family of atoms $\mathcal{A}$ of type $(s,p,u)$ induces an one-parameter scale 
$$\tilde{s}\rightarrow \mathcal{A}_{\tilde{s},p},$$
where $\mathcal{A}_{\tilde{s},p}$ is the family of atoms of type $(\tilde{s},p,u)$ defined by 
$$\mathcal{A}_{\tilde{s},p}(Q)= \{ |Q|^{\tilde{s}-s}a_Q\colon \ a_Q\in \mathcal{A}\}.$$
Moreover  a family of atoms $\mathcal{A}$ of type $(s,p,\infty)$ induces a two-parameter scale
$$(\tilde{s},\tilde{p})\rightarrow \mathcal{A}_{\tilde{s},\tilde{p}},$$
where $\mathcal{A}_{\tilde{s},\tilde{p}}$ is the family of atoms of type $(\tilde{s},\tilde{p},\infty)$ defined by 
$$\mathcal{A}_{\tilde{p},\tilde{s}}(Q)= \{ |Q|^{\tilde{s}-s+1/p-1/\tilde{p} }a_Q\colon \ a_Q\in \mathcal{A}\}.$$

\begin{proposition} \label{compa} Assume ${\Crr{fin}}$-${\Crr{sum}}$. Suppose that the $(s,p,\infty)$-atoms $\mathcal{A}$ satisfy ${\Crr{banach}}$-${\Crr{atom}}$. Let $0\leq s < \tilde{s} $ and $q,\tilde{q} \in [1,\infty]$.  Suppose
$$\Big(  \sum_{k}  |\mathcal{P}^k|^{q(\tilde{s}-s)} \Big)^{1/q}  < \infty.$$
Then 
\begin{itemize}
\item[A.] We have $\mathcal{B}^{\tilde{s}}_{p,\tilde{q}}(\mathcal{A}_{\tilde{s},p})\subset \SB^s_{p,q}(\mathcal{A}_{s,p})$ and the inclusion is a continuous linear map.
\item[B.] Suppose that also satisfies ${\Crr{compact}}$.  Let $g_n \in \mathcal{B}^{\tilde{s}}_{p,\tilde{q}}(\mathcal{A}_{\tilde{s},p})$ be such that $|g_n|_{\mathcal{B}^{\tilde{s}}_{p,\tilde{q}}(\mathcal{A}_{\tilde{s},p})}\leq C$ for every $n$. Then there is a subsequence that converges in $\SB^s_{p,q}(\mathcal{A}_{s,p})$ to some $g \in \mathcal{B}^{\tilde{s}}_{p,\tilde{q}}(\mathcal{A}_{s,p})$ with $|g|_{\mathcal{B}^{\tilde{s}}_{p,\tilde{q}}(\mathcal{A}_{\tilde{s},p})}\leq C$. 
\item[C.] Suppose that also satisfies ${\Crr{finite}}$. The inclusion $\imath\colon \mathcal{B}^{\tilde{s}}_{p,\tilde{q}}(\mathcal{A}_{\tilde{s},p})\mapsto \SB^s_{p,q}(\mathcal{A}_{s,p})$ is a compact linear map. 
\end{itemize}
\end{proposition}
\begin{proof} Consider a $\mathcal{B}^{\tilde{s}}_{p,\tilde{q}}(\mathcal{A}_{\tilde{s},p})$-representation 
$$f=\sum_{k=0}^{\infty}\sum_{Q\in \SP^k}s_Qa_Q,$$
Since $a_Q$ is an $\mathcal{A}_{\tilde{s},p}$-atom, we have that $b_Q=a_Q|Q|^{s-\tilde{s}}$ is an $\mathcal{A}_{s,p}$-atom. In particular, we can write
$$f=\sum_{k=0}^{\infty}\sum_{Q\in \SP^k}s_Q|Q|^{\tilde{s}-s}b_Q.$$
If  $k\geq k_0$ then 
\begin{eqnarray} 
(\sum_{Q\in \SP^k}|s_Q|^p |Q|^{p(\tilde{s}-s)}\Big)^{1/p} &\leq&  |\mathcal{P}^k|^{\tilde{s}-s} (\sum_{Q\in \SP^k}|s_Q|^p )^{1/p} \nonumber \\ 
&\leq&  |\mathcal{P}^k|^{\tilde{s}-s}   \Big(\sum_{k\geq k_0 } (\sum_{Q\in \SP^k}|s_Q|^p )^{\tilde{q}/p}\Big)^{1/\tilde{q}},\nonumber
\end{eqnarray}
so 
\begin{eqnarray} 
&& \Big( \sum_{k\geq k_0} (\sum_{Q\in \SP^k}|s_Q|^p |Q|^{p(\tilde{s}-s)}\Big)^{q/p} \Big)^{1/q} \nonumber \\  &\leq& \Big(  \sum_{k\geq k_0}  |\mathcal{P}^k|^{q(\tilde{s}-s)} \Big)^{1/q}   \Big(\sum_{k\geq k_0 } (\sum_{Q\in \SP^k}|s_Q|^p )^{\tilde{q}/p}\Big)^{1/\tilde{q}}.\label{uyuy}
\end{eqnarray}

\noindent {\it Proof of A.} In particular, \change{we added a comma}  taking $k_0=0$ we conclude that $\mathcal{B}^{\tilde{s}}_{p,\tilde{q}}(\mathcal{A}_{\tilde{s},p})\subset \SB^s_{p,q}(\mathcal{A}_{s,p})$ and
$$|f|_{\SB^s_{p,q}(\mathcal{A}_{p,s})}\leq   \Big(  \sum_{k}  |\mathcal{P}^k|^{q(\tilde{s}-s)} \Big)^{1/q}  | f|_{\mathcal{B}^{\tilde{s}}_{p,\tilde{q}}(\mathcal{A}_{p,\tilde{s}})}.$$

\noindent {\it Proof of B.}  By definition, there exist  $s^n_Q\in \mathbb{C}$, such that
$$g_n=\sum_{k=0}^{\infty}\sum_{Q\in \SP^k}s_Q^na_Q^n,$$
where $a_Q^n$ is a  $\mathcal{A}_{\tilde{s},p}$-atom supported on $Q$ and
\begin{equation} \label{sed}  (\sum_{k=0}^{\infty}(\sum_{Q\in \SP^k}|s_Q^n|^p)^{\tilde{q}/p})^{1/\tilde{q}}\leq C+\eps_n,\end{equation}
where $\eps_n\to 0$.
In particular, $|s_Q^n|\leq C+\eps_n$. Since the set $\cup_k \mathcal{P}^k$ is countable,
  by the Cantor  diagonal argument, taking a subsequence we can assume that 
  $s_Q^n\to s_Q$ and   (due to \change{we added "to"}  ${\Crr{compact}}$) that $a_Q^n$ 
converges in $\mathcal{B}(Q)$ and $L^p$ to some $a_Q\in \mathcal{A}_{\tilde{s},p}$.  By Lemma \ref{compa2} the sequence $g_n$ converge in $L^p$ to a function $g$ such that $|g|_{\mathcal{B}^{\tilde{s}}_{p,\tilde{q}}(\mathcal{A}_{\tilde{s},p})}\leq C$ and with $\mathcal{B}^{\tilde{s}}_{p,\tilde{q}}(\mathcal{A}_{s,p})$-representation
$$g=\sum_{k=0}^{\infty}\sum_{Q\in \SP^k}s_Q a_Q.$$
It remains to show that the convergence indeed occurs in the topology of  $\SB^s_{p,q}(\mathcal{A}_{s,p})$. For  every $k_0\geq 0$  and $\delta > 0$ we can write 
$$g_n-g=  \sum_{k< k_0} \sum_{Q\in \SP^k} \delta d_Q^{n} + \sum_{k\geq k_0} \sum_{Q\in \SP^k}  |Q|^{\tilde{s}-s} c_Q^{n} b_Q^{n},$$
where 
$$d_Q^{n}= \frac{1}{\delta}(s_Q^na_Q^n - s_Q a_Q),$$
and with $b_Q^{n}\in \mathcal{A}_{s,p}$ given by 
$$b_Q^{n} =   \frac{ |Q|^{s-\tilde{s}} |s_Q^n|}{|s_Q^n| + |s_Q|} sgn(s_Q^{n}) a_Q^{n}     +  \frac{ |Q|^{s-\tilde{s}} |s_Q|}{|s_Q^n| + |s_Q|}  sgn(-s_Q) a_Q,$$
and
$$c_Q^{n} = |s_Q^n| + |s_Q|.$$
Note that $b_Q^{n}  \in \mathcal{A}_{s,p}$.  Given $\epsilon > 0$, choose $k_0$ such that 
 $b_Q^{n}  \in \mathcal{A}_{s,p}$.  Given $\epsilon > 0$, choose $k_0$ such that 
$$ \Big(  \sum_{k\geq k_0}  |\mathcal{P}^k|^{q(\tilde{s}-s)} \Big)^{1/q}   (2C+1) \leq (\epsilon/2)^{1/\rho}.$$
By  (\ref{uyuy}) and (\ref{sed})  for each  $n$ large enough we have
$$
 \Big( \sum_{k\geq k_0} (\sum_{Q\in \SP^k}|c_Q^{n}|^p |Q|^{p(\tilde{s}-s)}\Big)^{q/p} \Big)^{1/q} \leq \Big(  \sum_{k\geq k_0}  |\mathcal{P}^k|^{q(\tilde{s}-s)} \Big)^{1/q}   (2C+1) < (\epsilon/2)^{1/\rho}.
$$
In particular
$$\big|  \sum_{k\geq k_0} \sum_{Q\in \SP^k}  |Q|^{\tilde{s}-s} c_Q^{n} b_Q^{n}\big|_{\SB^s_{p,q}(\mathcal{A}_{s,p})}^\rho < \epsilon/2.$$
Choose $\delta > $ such that 
$$
 \Big( \sum_{k<  k_0} (\sum_{Q\in \SP^k}\delta^p\Big)^{q/p} \Big)^{\rho/q}<   \epsilon/2.
$$

Due to \change{we added "to"}  ${\Crr{finite}}$  there is $\eta > 0$ such that for every $Q\in \mathcal{P}^k$, with $k< k_0$, if $h \in \mathcal{B}(Q)$ satisfies  $|h|_{\mathcal{B}(Q)}\leq \eta$ then $h \in  \mathcal{A}_{p,s}(Q)$. Since $\lim_n s_Q^na_Q^n=s_Q a_Q$ in $\mathcal{B}(Q)$ we conclude that for $n$ large enough we have 
$$    d_Q^{n}=\frac{1}{\delta}(s_Q^na_Q^n - s_Q a_Q) \in \mathcal{A}_{s,p}(Q)$$
for every  $Q\in \mathcal{P}^k$, with $k< k_0$. In particular
$$\big|\sum_{k< k_0} \sum_{Q\in \SP^k} \delta d_Q^{n}\big|_{\SB^s_{p,q}(\mathcal{A}_{s,p})}^\rho<  \epsilon/2.$$
We conclude that
$$| g_n-g|_{\SB^s_{p,q}(\mathcal{A}_{s,p})}^\rho<  \epsilon,$$
for $n$ large enough, so the sequence $g_n$  converges (due to \change{we added "to"}   Corollary \ref{compa1})  to $g$ in the topology of $\SB^s_{p,q}(\mathcal{A}_{s,p})$.

\end{proof}

\section{Transmutation of atoms} 
It turns out that sometimes a Besov-ish space can be obtained using different classes of atoms. The key result in Part I is the following

\begin{figure}
\includegraphics[scale=0.6]{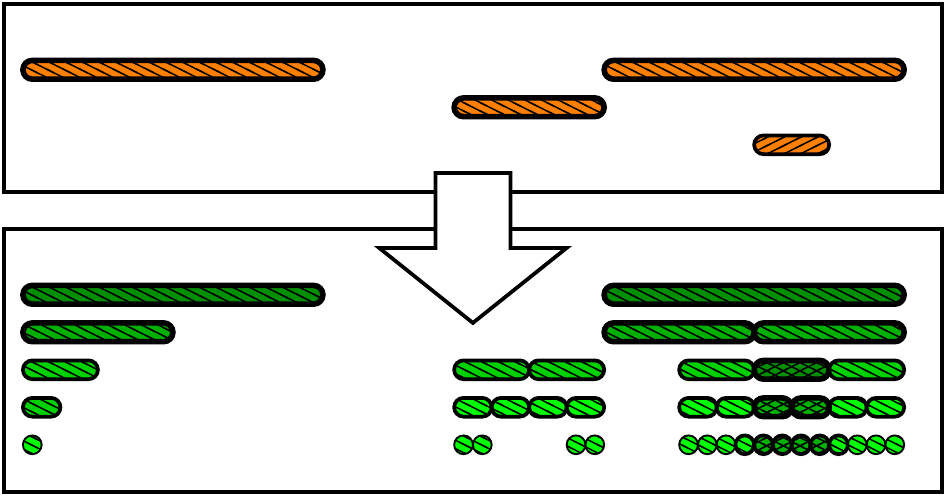}
\caption{ In the transmutation of atoms we replace  atoms whose supports are represented in orange (supports with distinct scales are drawn in distinct lines for a better understanding) by a  linear combination of atoms  (whose supports are represented  in green), in such way that  we  do not change too much  the location of the support and also  most of the  "mass" of the representation concentrates  on the same original scale. }
\end{figure}

\begin{proposition}[Transmutation of atoms] \label{trans} Assume  \\ 

\begin{itemize}
\item[{\bf I.}] Let  $\mathcal{A}_2$    be  a class of $(s,p,u_2)$-atoms  for  a grid  $\mathcal{W}$, satisfying  ${\Crr{banach}}$-${\Crr{atom}}$ and ${\Crr{fin}}$-${\Crr{sum}}$. Let $\mathcal{G}$ be also a  grid satisfying ${\Crr{fin}}$-${\Crr{sum}}$.
\item[{\bf II.}] Let $k_i\in \mathbb{N}$ for $i\geq 0$ be a sequence such that there is  $\alpha > 0 $ and $A, B\in \mathbb{R}$ satisfying
$$ \alpha i +A    \leq   k_i\leq \alpha i +B $$
for every $i$.
\item[{\bf III.}] There is $\lambda \in(0,1)$ such that the following holds. For every $Q\in \mathcal{G}$ and $P\in \mathcal{W}$ satisfying $P\subset Q$ there  are atoms $b_{P,Q} \in \mathcal{A}_2(P)$ and  corresponding $s_{P,Q}\in \mathbb{C}$ such that 
$$h_Q=\sum_k  \sum_{P\in \mathcal{W}^k, P\subset Q} s_{P,Q} b_{P,Q}.$$
is a  $\mathcal{B}^s_{p,q}(\mathcal{A}_2)$-representation of a function $h_Q$, with $s_{P,Q}=0$ for every $Q\in \mathcal{G}^i$,  $P \in \mathcal{W}^k$ with $k < k_i$  and moreover
\begin{equation}\label{ad}  \sum_{P\in \mathcal{W}^k, P \subset Q} |s_{P,Q}|^p   \leq \Cll{rf}  \lambda^{k-k_i}.\end{equation}
for every $k\geq k_i$.  \\ \\
\end{itemize}
Let
$$\mathcal{H}^k= \bigcup_{Q\in \mathcal{G}} \{ P \subset Q\colon \ P\in \mathcal{W}^k \ and  \ s_{P,Q}\neq 0      \}. $$
Then  \\ \\
\begin{itemize}
\item[{\bf A.}] For every coefficients $(c_Q)_{Q\in \mathcal{G}}$ such that 
$$\big( \sum_i  \big( \sum_{Q\in \mathcal{G}^i}  |c_{Q}|^p  \big)^{q/p}\big)^{1/q}<\infty$$
we have that the sequence
\begin{equation}\label{seq33} N\mapsto \sum_{i\leq N} \sum_{Q\in \mathcal{G}^i }c_Q h_Q\end{equation}
converges in $L^p$ to a function in $\mathcal{B}^s_{p,q}(\mathcal{A}_2)$ that has a $\mathcal{B}^s_{p,q}(\mathcal{A}_2)$-representation 
\begin{equation}\label{cs} \sum_k \sum_{P\in \mathcal{H}^k}     m_P d_P\end{equation}
where $m_P \geq 0$ for every $P$ and 
\begin{eqnarray} \label{igual1} &&\big(\sum_k \big( \sum_{\substack{ P \in \mathcal{W}^k\\ P\in \mathcal{H}}} |m_{P}|^p \big)^{q/p}\big)^{1/q}  \nonumber \\ 
&\leq& \Crr{mult1}  \Crr{rf}^{1/p}\lambda^{-\frac{B}{p}} \Crr{co2}(p,q,b) \Crr{m1}^{1/q} \big( \sum_i  \big( \sum_{Q\in \mathcal{G}^i}  |c_{Q}|^p  \big)^{q/p}\big)^{1/q}
\end{eqnarray}
Here $\Cll{m1}=\max \{\ell \in \mathbb{N}, \ell < \alpha\}+1$ and $b=(b_n)_{n\in \mathbb{Z}}$ is defined  by 
$$b_n = \begin{cases}    \lambda^{ \alpha n} &   \text{ if } n>   \frac{A}{\alpha} - 1,\\
0  &   \text{ if } n\leq    \frac{A}{\alpha} - 1,
 \end{cases}
 $$

\item[{\bf B.}] Suppose that the   assumptions of A. hold and that  $s_{P,Q}$ are non negative real numbers and $b_{P,Q} > 0$  on $P$ for every   $P, Q$.   Then $m_P\neq 0$ and $d_P \neq  0$ on $P$ imply that  $P\subset  supp \ h_Q$ for some $Q\in \mathcal{W}^k$ satisfying $c_Q \neq 0$ and $s_{P,Q}  >0$.  If we additionally  assume that $c_Q \geq 0$ for every $Q$ then $m_P\neq 0$ also   implies $d_P >  0$ on $P$.
\item[{\bf C.}] Let $\mathcal{A}_1$     be a   class of $(s,p,u_1)$-atoms for  the  grid  $\mathcal{G}$ satisfying  ${\Crr{banach}}$-${\Crr{atom}}$. Suppose that there is $\lambda < 1$ such that   for every atom $a_Q\in \mathcal{A}_1(Q)$ we can find $s_{P,Q}$ and $b_{P,Q}$ in III. such that  $h_Q=a_Q$. Then  $$\mathcal{B}^s_{p,q}(\mathcal{A}_1)\subset \mathcal{B}^s_{p,q}(\mathcal{A}_2)$$ and this inclusion is continuous. Indeed
$$|\phi|_{\mathcal{B}^s_{p,q}(\mathcal{A}_2)} \leq  \Crr{mult1}  \Crr{rf}^{1/p}\lambda^{-\frac{B}{p}} \Crr{co2}(p,q,b) \Crr{m1}^{1/q} |\phi|_{\mathcal{B}^s_{p,q}(\mathcal{A}_1)}$$
for every  $\phi \in \mathcal{B}^s_{p,q}(\mathcal{A}_1).$ 
\end{itemize}
\end{proposition}

\changei{In the caption of the figure 1 we replaced "are concentrated" by "concentrates"}
\begin{proof} 
For every $P\in \mathcal{H}^k$, with $k\in \mathbb{N}$,  and $N\in \mathbb{N}\cup \{\infty\}$ define
$$m_{P,N}=\sum_{i\leq N} \sum_{\substack{Q\in \mathcal{G}^i\\ P \subset Q}} |c_Q s_{P,Q}|=\sum_{\substack{ k_i\leq k \\ i\leq N}} \sum_{ \substack{Q\in \mathcal{G}^i\\ P \subset Q}} |c_Q s_{P,Q}|$$
Due to \change{we added "to"}  $\Crr{sum}$ this sum has a finite number of terms. If this sum has  zero terms define $m_{P,N}=0$ and let $d_{P,N}$ be  the zero function. Otherwise define 
\begin{equation}\label{from} d_{P,N}= \frac{1}{m_{P,N}}\sum_{i\leq N} \sum_{\substack{Q\in \mathcal{G}^i\\ P \subset Q}} c_Q s_{P,Q} b_{P,Q}.\end{equation} 
We have that $d_{P,N}$  is an $\mathcal{A}_2(P)$-atom.  
\vspace{5mm} 

\noindent {\bf Claim I.} {\it We claim that for $N\in \mathbb{N}$}
 $$\sum_{k}   \sum_{P\in \mathcal{H}^k} m_{P,N}  d_{P,N} = \sum_{i\leq N} \sum_{Q\in \mathcal{G}^i}c_Q h_Q.   $$
Note that if $Q\in \mathcal{G}^i$ then due to \change{we added "to"}  (\ref{flp}), with $t=p$  and (\ref{ad})  we have
$$\sum_k  \big| \sum_{\substack{P\in \mathcal{H}^k\\ P\subset Q}} s_{P,Q} b_{P,Q}\big|_p^{p/\hat{p}} < \infty.$$
Consequently we can do the following manipulation in $L^p$
\begin{eqnarray*}
 \sum_{i\leq N} \sum_{Q\in \mathcal{G}^i} c_Q h_Q&=& \sum_{i\leq N} \sum_{Q\in \mathcal{G}^i}  \sum_k  \sum_{\substack{ P\in \mathcal{H}^k\\  P\subset Q}} s_{P,Q} b_{P,Q}\\
 &=& \sum_k \sum_{ P\in \mathcal{H}^k}    \sum_{i\leq N} \sum_{Q\in \mathcal{G}^i}   \sum_{P\subset Q} s_{P,Q} b_{P,Q}\\
 &=& \sum_k \sum_{ P\in \mathcal{H}^k}   m_{P,N}  d_{P,N}.
 \end{eqnarray*}      
This concludes the proof of Claim I. 
\vspace{5mm}

\noindent {\bf Claim II.} {\it For every $N\in \mathbb{N}\cup \{\infty\}$ we claim that 
\begin{equation}\label{a222}  \sum_k \sum_{P\in \mathcal{H}^k} m_{P,N}  d_{P,N}\end{equation}
is a $\mathcal{B}^s_{p,q}(\mathcal{A}_2)$-representation and }
\begin{align}&   \big(\sum_k \big( \sum_{ P\in \mathcal{H}^k} |m_{P,N}|^p \big)^{q/p}\big)^{1/q} \nonumber \\
\label{ess} &\leq \Crr{mult1}  \Crr{rf}^{1/p}\lambda^{-\frac{B}{p}} \Crr{co2}(p,q,b) \Crr{m1}^{1/q}\big( \sum_i  \big( \sum_{Q\in \mathcal{G}^i}  |c_{Q}|^p  \big)^{q/p}\big)^{1/q}.
\end{align}
Indeed 
\begin{align*} 
\Big( \sum_{P \in \mathcal{H}^k} |m_{P,N}|^p \Big)^{1/\hat{p}}&= \Big(  \sum_{P \in \mathcal{W}^k} \big( \sum_{\substack{ k_i\leq k \\ i\leq N}} \sum_{ \substack{Q\in \mathcal{G}^i\\ P \subset Q}} |c_Q s_{P,Q}|\big)^p \Big)^{1/\hat{p}} \\
&\leq  \sum_{\substack{ k_i\leq k \\ i\leq N}}  \big( \sum_{P \in \mathcal{W}^k} \big( \sum_{ \substack{Q\in \mathcal{G}^i\\ P \subset Q}} |c_Q s_{P,Q}|\big)^p \big)^{1/\hat{p}} \\
&\leq \Crr{mult1}^{p/\hat{p}} \sum_{\substack{ k_i\leq k \\ i\leq N}}  \big( \sum_{P \in \mathcal{W}^k} \sum_{ \substack{Q\in \mathcal{G}^i\\ P \subset Q}} |c_Q|^p |s_{P,Q}|^p \big)^{1/\hat{p}} \\
&\leq \Crr{mult1}^{p/\hat{p}}  \sum_{\substack{ k_i\leq k \\ i\leq N}}  \big(\sum_{Q\in \mathcal{G}^i}   |c_Q|^p  \sum_{\substack{P \in \mathcal{W}^k\\ P \subset Q}}  |s_{P,Q}|^p \big)^{1/\hat{p}} \\
&\leq  \Crr{mult1}^{p/\hat{p}}   \Crr{rf}^{1/\hat{p}}  \sum_{\substack{k_i\leq k \\ i\leq N}}  \lambda^{(k-k_i)/\hat{p}}   \big( \sum_{Q\in \mathcal{G}^i}  |c_{Q}|^p  \big)^{1/\hat{p}} \\
&\leq  \Crr{mult1}^{p/\hat{p}}   \Crr{rf}^{1/\hat{p}}  \sum_{\substack{\alpha i + A\leq k}}  \lambda^{(k-\alpha i -B)/\hat{p}}   \big( \sum_{Q\in \mathcal{G}^i}  |c_{Q}|^p  \big)^{1/\hat{p}} . \numberthis \label{for}
\end{align*}
If $\alpha=1$ and $A=B=0$ then this is a convolution, so we can use Proposition \ref{young} (the convolution trick)  and it easily follows (\ref{igual1}). In the general case, consider
$$u_k=\sum_{P \in \mathcal{H}^k} |m_{P,N}|^p  \text{  and }  c_i=\sum_{Q\in \mathcal{G}^i}  |c_{Q}|^p$$
Every $k\in \mathbb{N}$ can be  written in an {\it unique  way} as $k=\alpha j_k+ \ell_k+r_k$, with $j_k \in \mathbb{N}$, $\ell_k \in \mathbb{N}$, $\ell_k+r_k < \alpha$ and $r_k \in [0,1)$.   Fix  $\ell \in [0,\alpha)\cap\mathbb{N}$ and $j\in \mathbb{N}$. Then there is at most one  $k'\in \mathbb{N}$ such that $\ell_k=\ell$ and $j_k=j$. Indeed, if $k'= \alpha j + \ell+r'$ and $k''= \alpha j + \ell+r''$, with $r',r''\in [0,1)$ and $\ell+r'$ and $\ell+r''$ smaller than $\alpha$, then $k'-k''=r'-r''\in (-1,1)$, so $k'=k''$ and $r'=r''$. If such $k'$ exists, denote $k(\ell,j)=k'$ and $r(\ell,j)= k(\ell,j)-\alpha j -\ell$ and $a_{\ell,j}=u_{k(\ell,j)}$. Otherwise let $a_{\ell,j}=0$. Then (\ref{for}) implies 
\begin{align*}
a_{\ell,j}^{1/\hat{p}}&\leq \Crr{mult1}^{p/\hat{p}}   \Crr{rf}^{1/\hat{p}}  \sum_{\substack{\alpha i + A\leq  \alpha j+ \ell+r(\ell,j)}}  \lambda^{( \alpha j+ \ell+r(\ell,j)-\alpha i -B)/\hat{p}} c_i^{1/\hat{p}}  \\
&\leq \Crr{mult1}^{p/\hat{p}}   \Crr{rf}^{1/\hat{p}} \lambda^{-B/\hat{p}}  \sum_{\substack{ i \leq   j+ \frac{-A+\ell+r(\ell,j)}{\alpha} }  }  \lambda^{ \alpha (j- i)/\hat{p}} c_i^{1/\hat{p}}\\
&\leq  \Crr{mult1}^{p/\hat{p}}   \Crr{rf}^{1/\hat{p}}   \lambda^{-B/\hat{p}}   \sum_{\substack{ i <    j+ \frac{-A}{\alpha} +1}}  \lambda^{ \alpha (j- i)/\hat{p}} c_i^{1/\hat{p}}\\
&\leq \Crr{mult1}^{p/\hat{p}}   \Crr{rf}^{1/\hat{p}}  \lambda^{-B/\hat{p}}   \sum_{i \in \mathbb{Z} }  b_{j-i}^{1/\hat{p}} c_i^{1/\hat{p}}
\end{align*}
Here $b_n= \lambda^{ \alpha n} $ if  $n> A/\alpha - 1$, and $b_n=0$ otherwise. Fixing $\ell \in \mathbb{N}, \ \ell< \alpha$, Proposition \ref{young} (the convolution trick) gives us 
\begin{eqnarray*} K_\ell= \big(\sum_{\substack{k\in \mathbb{N} \\ \ell_k=\ell}}  u_{k}^{q/p}\big)^{1/q} &=&\big(\sum_{j} a_{\ell,j}^{q/p}\big)^{1/q}\\
&\leq&  \Crr{mult1}  \Crr{rf}^{1/p}\lambda^{-\frac{B}{p}} \Crr{co2}(p,q,b) \big( \sum_i  \big( \sum_{Q\in \mathcal{G}^i}  |c_{Q}|^p  \big)^{q/p}\big)^{1/q}.
\end{eqnarray*} 
and
\begin{align} \label{bbbb} \big(\sum_k \big( \sum_{ P\in \mathcal{H}^k} |m_{P,N}|^p \big)^{q/p}\big)^{1/q}&=  \big(\sum_k u_k^{q/p}\big)^{1/q}  \nonumber \\
&=  \big(\sum_{0\leq \ell < \alpha} \sum_{\substack{k\in \mathbb{N} \\ \ell_k=\ell}}  u_{k}^{q/p}\big)^{1/q}=  \big(\sum_{0\leq \ell < \alpha}K_\ell^q \big)^{1/q}\nonumber \\
 &\leq \Crr{mult1}  \Crr{rf}^{1/p}\lambda^{-\frac{B}{p}} \Crr{co2}(p,q,b) \Crr{m1}^{1/q}  \big( \sum_i  \big( \sum_{Q\in \mathcal{G}^i}  |c_{Q}|^p  \big)^{q/p}\big)^{1/q}.
\end{align}
This implies in particular that the sum  in (\ref{a222}) is a $\mathcal{B}^s_{p,q}(\mathcal{A}_2)$-representation. This proves Claim II.  
 \vspace{5mm}
  
  \noindent {\bf Claim III.} {\it We have that in the strong topology of  $L^p$
$$\lim_{N\rightarrow \infty}  \sum_k \sum_{P\in \mathcal{H}^k} m_{P,N}  d_{P,N} =\sum_k \sum_{P\in \mathcal{H}^k} m_{P,\infty}  d_{P,\infty}.   $$}
For each $P\in \mathcal{H}$ the sequence
\begin{equation}\label{seqdf} N\mapsto m_{P,N}\end{equation} 
is eventually constant, therefore convergent. The same happens with 
\begin{equation}\label{seqdf2} N\mapsto d_{P,N}.\end{equation}

\noindent Estimate (\ref{ess}) and Proposition \ref{compa2} imply that  that (\ref{seq33}) converges in $L^p$ to a function    with $\mathcal{B}^s_{p,q}(\mathcal{A}_2)$-representation  (\ref{a222}) with $N=\infty$. This concludes the proof of Claim III. 

Then Claim I, II, and III imply A. taking $m_p=m_{P,\infty}$ and $d_P=d_{P,\infty}.$ We have that  C. is an immediate consequence of A.  Note that (\ref{from}) and A. give B.
\end{proof}
\vspace{1cm}

\section{Good grids} \label{goodgrids} A $(\Cll[c]{menor},\Cll[c]{maior})$-{\bf good  grid} , with $0< \Crr{menor}<  \Crr{maior} <  1$, is a grid  $\mathcal{P}= (\mathcal{P}^k)_{k\in \mathbb{N}}$  with the following properties:
\begin{itemize}
\item[${\Cll[G]{g2}}.$] We have $\mathcal{P}^0=\{I\}$.
\item[${\Cll[G]{g3}}.$] We have $I=\cup_{Q \in \mathcal{P}^k} Q$  (up to a set of zero $m$-measure).
\item[${\Cll[G]{g4}}.$] The elements of the  family $\{  Q \}_{Q\in \mathcal{P}^k} $ are pairwise disjoint. 
\item[${\Cll[G]{g5}}.$] For every $Q\in \mathcal{P}^k$ and $k > 0$ there exists $P \in \mathcal{P}^{k-1}$ such that $Q\subset P$.  
\item[${\Cll[G]{g6}}.$] We have
$$\Crr{menor} \leq \frac{|Q|}{|P|} \leq \Crr{maior}$$
for every $Q\subset P$ satisfying  $Q\in \mathcal{P}^{k+1}$ and $P\in \mathcal{P}^{k}$ for some $k\geq 0$. 
\item[${\Cll[G]{g7}}.$]  The family $\cup_k \mathcal{P}^k$ generate the  $\sigma$-algebra $\mathbb{A}$. 
\end{itemize}

\section{Induced  spaces} 
Consider a Besov-ish space $\mathcal{B}^s_{p,q}(I,\mathcal{P}, \mathcal{A})$, where $\mathcal{P}$ is a  good grid. Given $Q\in \mathcal{P}^{k_0}$, we can  consider the sequence of finite families of subsets $\mathcal{P}_Q=(\mathcal{P}_Q^i)_{i\geq 0}$ of $Q$  given by 
$$\mathcal{P}^i_Q=\{ P\in \mathcal{P}^{k_0+i}, P\subset Q\}.$$
Let $\mathcal{A}_Q$ be the restriction of the indexed family $\mathcal{A}$ of pairs   $(\mathcal{B}(P),\mathcal{A}(P))_{P\in \mathcal{P}}$ to indices  belonging to $\mathcal{P}_Q$. Then we can consider the {\bf induced} Besov-ish space $\mathcal{B}^s_{p,q}(Q,\mathcal{P}_Q, \mathcal{A}_Q)$. Of course the inclusion
$$i\colon \mathcal{B}^s_{p,q}(Q,\mathcal{P}_Q, \mathcal{A}_Q) \rightarrow \mathcal{B}^s_{p,q}(I,\mathcal{P}, \mathcal{A})$$
is well-defined \change{we replaced "well defined"by "well-defined"} and  it is a weak contraction, that is
$$|f|_{\mathcal{B}^s_{p,q}(I,\mathcal{P}, \mathcal{A})}\leq |f|_{\mathcal{B}^s_{p,q}(Q,\mathcal{P}_Q, \mathcal{A}_Q)}.$$
Under the degree of generality we are considering here, the {\bf restriction } transformation
$$r\colon \mathcal{B}^s_{p,q}(I,\mathcal{P}, \mathcal{A}) \rightarrow L^p$$
given by $r(f)= f\cdot 1_Q$, is a bounded linear transformation, however it is easy to find examples of Besov-ish spaces where $f 1_Q \not\in \mathcal{B}^s_{p,q}(Q,\mathcal{P}_Q, \mathcal{A}_Q).$

\section{Examples of classes of atoms.}\label{secatom}

There are many classes of atoms one may consider. We list here just  a few of them.

\subsection{Souza's  atoms}\label{souzaa}

Let $Q\in \mathcal{P}$. A {\bf $(s,p)$-Souza's atom}  supported  on $Q$ is a function $a\colon I \rightarrow \mathbb{C}$ such that $a(x)=0$ for every $x \not\in Q$ and $a$ is constant on $Q$, with 
$$|a|_\infty\leq |Q|^{s-1/p}.$$

The set of Souza's atoms supported  on $Q$ will be denoted by $\mathcal{A}^{sz}_{s,p}(Q).$ A {\bf canonical Souza's atom} on $Q$ is the Souza's atom such that  $a(x)= |Q|^{s-1/p}$ for every $x \in Q$.  Souza's atoms are $(s,p,\infty)$-type atoms.




 \subsection{H\"older atoms} \label{hatoms} Suppose that $I$ is a quasi-metric space with a quasi-distance $d(\cdot,\cdot)$, such that every $Q\in \mathcal{P}$ is a bounded set and there is $\Cll[c]{hh},\Cll[c]{hhh} \in (0,1)$ such that
$$\Crr{hh} \leq \frac{diam \ P}{diam \ Q} \leq \Crr{hhh}$$
for every $P\subset Q$ with $P \in \mathcal{P}^{k+1}$ and $Q \in \mathcal{P}^{k}$.  Additionally assume  that there is $\Cll{kki}\geq 0$ and $D\geq 0$ such that 
$$\frac{1}{\Crr{kki}^{D}}|Q| \leq (diam \ Q)^D \leq \Crr{kki}^{D} |Q|.$$

Let 
$$0< s <  \frac{1}{p}, \ s < \beta.$$   For every $Q\in \mathcal{P}$, Let  $C^\alpha(Q)$ be the Banach space of all functions  $\phi$ such that $\phi(x)=0$ for $x \not\in Q$, and   
$$|\phi|_{C^\alpha(Q)}= |\phi|_\infty + \sup_{\substack{x,y \in Q\\ x\neq y}} \frac{|\phi(x)-\phi(y)|}{d(x,y)^\alpha} < \infty.$$
Let $\mathcal{A}^{h}_{s,\beta,p}(Q) \subset C^{\beta D}(Q)$ be the convex subset of all functions $\phi$ satisfying 
$$\sup_{\substack{x,y \in Q\\ x\neq y}} \frac{|\phi(x)-\phi(y)|}{d(x,y)^{D \beta}} \leq  |Q|^{s-1/p-\beta} \   and \ |\phi|_{\infty} \leq |Q|^{s-1/p}.$$
We say that $\mathcal{A}^{h}_{s,\beta,p}(Q)$ is the   set of {\bf $(s,\beta,p)$-H\"older atoms } supported on $Q$. Of course $\mathcal{A}^{h}_{s,\beta,p}$-atoms are  $(s,p,\infty)$-type atoms and $\mathcal{A}^{sz}_{s,p}(Q)\subset \mathcal{A}^{h}_{s,\beta,p}(Q)$.

\subsection{Bounded variation atoms}  Now suppose that  $I$ is an interval of $\mathbb{R}$ with  length $1$, $m$ is the Lebesgue measure on it and the partitions in the  grid $\mathcal{P}$ are partitions by intervals.  
Let $Q$ be an interval and $s \leq   \beta$, $p\in [1,\infty)$.  A $(s,\beta,p)$-bounded variation  atom on $Q$ is a function $a\colon \mathbb{R} \rightarrow \mathbb{C}$ such that $a(x)=0$ for every $x \not\in Q$,
$$|a|_\infty\leq |Q|^{s-1/p}.$$ 
and
$$var_{1/\beta}(a,Q) \leq |Q|^{s-1/p}.$$ 
Here $var_{1/\beta}(\cdot,Q)$ is the pseudo-norm 
$$var_{1/\beta}(a,Q) = \sup (\sum_i |a(x_{i+1})-a(x_i)|^{1/\beta})^{\beta},$$
where the sup runs over all possible sequences $x_1 < x_1 < \dots < x_n$, with $x_i$ in the  interior of $Q$.
We will denote the set of  bounded variation  atoms on  $Q$ as $\mathcal{A}^{bv}_{s, p,\beta}(Q)$. Bounded variation atoms are also $(s,p,\infty)$-type atoms.
\vspace{1cm}

\newpage

\centerline{ \bf II. SPACES DEFINED BY  SOUZA'S ATOMS.}
\addcontentsline{toc}{chapter}{\bf II. SPACES DEFINED BY  SOUZA'S ATOMS.}
\vspace{1cm}

\fcolorbox{black}{white}{
\begin{minipage}{\textwidth}
\noindent   In Part II. we  suppose that $s> 0$, $p\in [1,\infty)$ and $q\in [1,\infty]$. \end{minipage} 
}
\ \\ \\

\section{Besov spaces in a measure space with a good grid}

 We will study the Besov-ish spaces  $\mathcal{B}^s_{p,q}(\mathcal{P},\mathcal{A}^{sz}_{s,p})$ associated with the measure space with a good grid $(I, \mathcal{P},m)$. We denote  $\mathcal{B}^s_{p,q}=\mathcal{B}^s_{p,q}(\mathcal{P},\mathcal{A}^{sz}_{s,p})$. Note that $\mathcal{A}^{sz}_{s,p}$ satisfies ${\Crr{banach}}$-${\Crr{finite}}$. Note that by Proposition \ref{lp} there is $\beta > 1$ such that $\mathcal{B}^s_{p,q} \subset L^\beta$.
 
If  $p\in [1,\infty)$, $q\in [1,\infty]$ and  $0< s< \frac{1}{p}$ we wil say that  $\mathcal{B}^s_{p,q}$ is a   {\bf Besov space}.

\section{Positive cone}

We say that $f$ is $\mathcal{B}^s_{p,q}$-positive if there is a $\mathcal{B}^s_{p,q}$-representation
$$f= \sum_k \sum_{P\in \mathcal{P}^k} c_Pa_P$$
where $c_P\geq 0$ and $a_P$ is the standard $(s,p)$-Souza's atom supported on $P$. The set of all $\mathcal{B}^s_{p,q}$-positive functions is a convex cone in $\mathcal{B}^s_{p,q}$, denoted $\mathcal{B}^{s+}_{p,q}$. We can define a ``norm" on $\mathcal{B}^{s+}_{p,q}$ as
$$|f|_{\mathcal{B}^{s+}_{p,q}}=\inf   \Big( \sum_k \big( \sum_{P\in \mathcal{P}^k}  c_P^p \big)^{q/p} \Big)^{1/q},$$
where the infimum runs over all possible $\mathcal{B}^s_{p,q}$-positive representations of $f$. Of course for every $f,g \in \mathcal{B}^{s+}_{p,q}$ and $\alpha \geq 0$ we have 
$$|\alpha f|_{\mathcal{B}^{s+}_{p,q}}=\alpha  | f|_{\mathcal{B}^{s+}_{p,q}}, \ |f+g|_{\mathcal{B}^{s+}_{p,q}}\leq  |f|_{\mathcal{B}^{s+}_{p,q}}+  |g|_{\mathcal{B}^{s+}_{p,q}}, \   | f|_{\mathcal{B}^{s}_{p,q}} \leq | f|_{\mathcal{B}^{s+}_{p,q}}.$$
Moreover if $f\in \mathcal{B}^s_{p,q}$   is a real-valued function then one can find $f_+,f_-\in \mathcal{B}^{s+}_{p,q}$ such that $f=f_+-f_-$ and $$|f_+|_{\mathcal{B}^{s+}_{p,q}} \leq |f|_{\mathcal{B}^{s}_{p,q}} \ and \  |f_-|_{\mathcal{B}^{s+}_{p,q}} \leq |f|_{\mathcal{B}^{s}_{p,q}}.$$

An obvious but important observation is

\begin{proposition} if $f \in \mathcal{B}^{s+}_{p,q}$ then its support    
$$supp \ f = \{ x\in I\colon \ f(x)\neq 0\}.$$
is (up to a set of zero measure) is a countable union of elements of $\mathcal{P}$. 
\end{proposition}

\section{Unbalanced Haar wavelets} \label{haar} Let $\mathcal{P}=(\mathcal{P}^k)_k$ be a good grid. For every $Q \in \mathcal{P}^k$ let $\Omega_Q=\{ P_1^Q,\dots,P_{n_Q}^Q\}$, $n_Q\geq 2$,  be  the family of  elements  $\mathcal{P}^{k+1}$ such that $P_i^Q \subset Q$ for every $i$, and ordered in some arbitrary way.  The elements of  $\Omega_Q$ will be called   {\bf children} of $Q$.  Note that every $Q\in \mathcal{P}$ has at least two children. We will use one of the method described (Type I tree with the logarithmic subtrees construction) in  Girardi and Sweldens \cite{gw} to construct an unconditional basis of $L^{\beta}$, for every  $1< \beta <\infty$. 

Let   $\mathcal{H}_Q$ be the family of pairs $(S_1,S_2)$, with $S_i \subset \Omega_Q$ and $S_1\cap S_2=\emptyset$, defined  as
$$\mathcal{H}_Q=\cup_{j\in \mathbb{N}} \mathcal{H}_{Q,j},$$
where $\mathcal{H}_{Q,j}$ are constructed recursively  in the  following way. Let $\mathcal{H}_{Q,0}=\{ (A,B)\}$, where $A=\{ P_1^Q,\dots,P_{[n_Q/2]}^Q   \}$ and $B=\{ P_{[n_Q/2]+1}^Q,\dots,P_{n_Q}^Q   \}$. Here $[x]$ denotes the integer part of $x\geq 0$.   Suppose that we have defined  $\mathcal{H}_{Q,j}$. For each element   $(S_1,S_2)\in \mathcal{H}_{Q,j}$, {\it fix}  an ordering $S_1=\{ R_1^1, \dots, R_{n_1}^1  \}$ and $S_2=\{ R_1^2, \dots, R_{n_2}^2  \}$. For each  $i=1,2$ such that  $n_i\geq 2$, define $T_i^1=\{ R_1^i,\dots,R_{[n_i/2]}^i  \}$ and $T_i^2=\{ R_{[n_i/2]+1}^i,\dots,R_{n_i}^i   \}$ and  add  $(T_i^1,T_i^2)$ to  $\mathcal{H}_{Q,j+1}$. This defines $\mathcal{H}_{Q,j+1}.$

Note that since $\mathcal{P}$ is a good grid we have $\mathcal{H}_{Q,j}=\emptyset$ for large $j$ and indeed 
$$\sup_{Q\in \mathcal{P}} \#\mathcal{H}_Q < \infty.$$

Define $\mathcal{H}=\cup_{Q\in \mathcal{P}} \mathcal{H}_Q$.  For every $S=(S_1,S_2) \in\mathcal{H}_Q$ define

$$\phi_{S} =     \frac{1}{m_{(S_1,S_2)}}\Big( \frac{\sum_{P \in S_1} 1_{P}}{\sum_{P \in S_1} |P| }   -  \frac{\sum_{R \in S_2} 1_{R}}{\sum_{R \in S_2} |R|} \Big)$$
where 
$$m_{(S_1,S_2)}=  \big(  \frac{1}{\sum_{P \in S_1} |P|}  + \frac{1}{\sum_{R \in S_2} |R|}   \big)^{1/2} .$$
Note that 
$$\int_Q \phi_S  \ dm=0.$$
Since $1\leq \#S_i\leq 1/\Crr{menor}$ we have
$$\Crr{menor}|Q|\leq \sum_{P \in S_i} |P|\leq  \frac{\Crr{maior}}{\Crr{menor}} |Q|,$$
so
$$\Big(\frac{2\Crr{menor}}{\Crr{maior}}\Big)^{1/2}\frac{1}{|Q|^{1/2}} \leq m_{(S_1,S_2)}\leq \Big(\frac{2}{\Crr{menor}}\Big)^{1/2}\frac{1}{|Q|^{1/2}}$$
Consequently 
\begin{align}\label{estphi} \frac{\Crr{c1}}{|Q|^{1/2}}  &\leq \frac{1}{m_{(S_1,S_2)}}    \min \{  \frac{1}{\sum_{P \in S_1} |P| },   \frac{1}{\sum_{R \in S_2} |R| }    \} \nonumber \\
&\leq |\phi_{S}(x)|\leq    \frac{1}{m_{(S_1,S_2)}}    \max \{  \frac{1}{\sum_{P \in S_1} |P| },   \frac{1}{\sum_{R \in S_2} |R| }    \}\nonumber \\
&\leq \frac{\Crr{c2}}{|Q|^{1/2}} .\end{align}
for every $x \in \cup_{P\in S_1\cup S_2} P.$  Here 
$$\Cll{c1}= \frac{\Crr{menor}^{3/2}}{\sqrt{2}\Crr{maior}}  \ and \ \Cll{c2}=  \frac{\Crr{maior}^{1/2}}{\sqrt{2}\Crr{menor}^{3/2}}+1. $$

\begin{figure}
\includegraphics[scale=0.6]{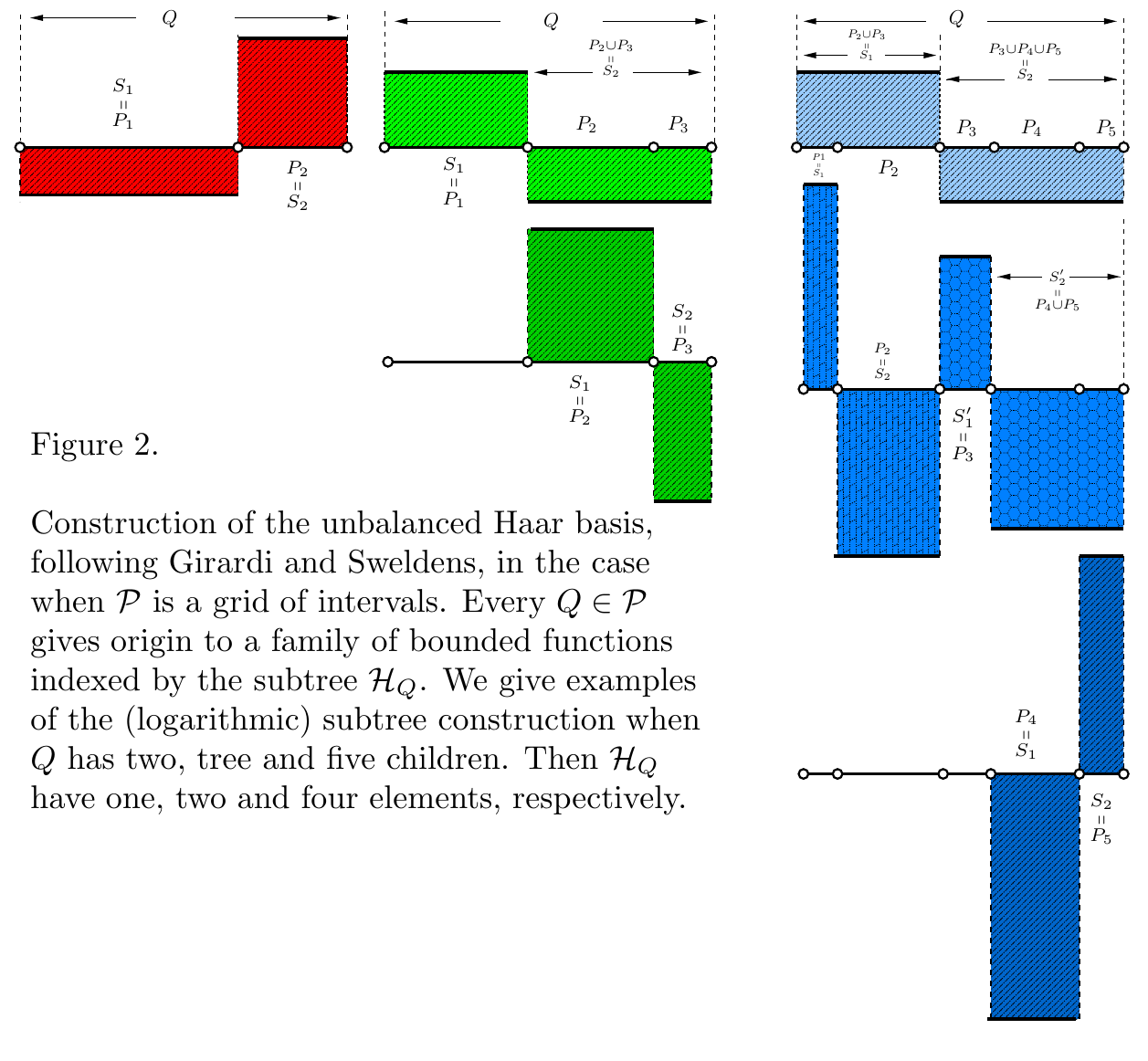}
\end{figure}
Let 
$$\hat{\mathcal{H}}= \{ I\}\cup \mathcal{H}$$ and define
$$\phi_I= \frac{1_I}{|I|^{1/2}}.$$
Then  by  Girardi and Sweldens \cite{gw}  we have that 
$$\{ \phi_S\}_{S\in \hat{\mathcal{H}}}$$
is an unconditional basis of $L^\beta$ for every $\beta > 1$.   
 
  \section{Alternative characterizations I: Messing with norms.}\label{alt1}
 We are going to describe three norms that are equivalent to $|\cdot|_{\mathcal{B}^s_{p,q}}$. Their advantage is that they are far more concrete, in the sense that we do not need to consider arbitrary atomic decompositions to define them. 

 \subsection{Haar  representation}  
For every $f\in L^\beta$, $\beta > 1$,  the series 
\begin{equation}\label{fs22}  f =    \sum_{S\in \hat{\mathcal{H}}} d_S^f \phi_S \end{equation} 
is converges unconditionally  in $L^\beta$, where
 $d_S^f = \int f \phi_S \ dm.$ 
 We will call the r.h.s. of (\ref{fs22}) the {\bf Haar representation} of $f$. 
 Define
 $$ N_{haar}(f)=|I|^{1/p-s-1/2} |d_I^f|+ \Big( \sum_{k}  \big(  \sum_{\substack{ _{Q\in \mathcal{P}^{k}}}}  |Q|^{_{1-sp -\frac{p}{2}}}   \sum_{_{S\in \mathcal{H}_Q}}   |d_S^f|^p\big)^{q/p} \Big)^{1/q},$$

\subsection{Standard atomic representation} Note that 
$$k_I^fa_I= d_I^f \phi_I,$$
where  $k_I= |I|^{1/p-s-1/2}d_I^f$ and $a_I$ is the canonical Souza's atom on $I$.  Let  $S\in \mathcal{H}$. Then $S\in \mathcal{H}_Q$, with $S=(S_1,S_2)$ and some $Q\in \mathcal{P}^k$, with $k\geq 0$.  It is easy to see that for every $P\in S_1\cup S_2$ the function 
 $$a_{S,P}= \frac{ |Q|^{1/2}}{\Crr{c2}} |P|^{s-1/p}  \phi_S 1_{P}$$
is a Souza's atom on $P$. Choose 
 $$c_{S,P}^f=   \Crr{c2} |Q|^{-1/2}  |P|^{1/p-s}  d_S^f. $$
 Note that
\begin{equation} \label{estcsp} |c_{S,P}^f|\leq   \Crr{c2} \max\{ \Crr{maior}^{1/p-s}, \Crr{menor}^{1/p-s}\}   |Q|^{1/p-s-1/2}  |d_S^f|.\end{equation} 
For every child $P$ of $Q \in \mathcal{P}^k$, $k\geq 0$,  define 
 $$\tilde{a}_P^f=   \frac{1}{\tilde{k}_P^f}  \sum_{S=(S_1,S_2)\in \mathcal{H}_Q}   \sum_{P\in S_1\cup S_2} c_{S,P}^f a_{S,P},$$
 where  
\begin{equation}\label{hh} \tilde{k}_P^f =   \sum_{S=(S_1,S_2)\in \mathcal{H}_Q}   \sum_{P\in S_1\cup S_2} |c_{S,P}^f|.\end{equation}
The (finite) number of terms on this sum depends only on the geometry of $\mathcal{P}$. 
Then $\tilde{a}_P^f$ is a Souza's  atom on $P$ and
$$ \sum_{S\in \mathcal{H}_Q}   d_S^f \phi_S =    \sum_{\substack{ _{P \in \mathcal{P}^{k+1}}\\_{P\subset Q}}}  \tilde{k}_P^f \tilde{a}_P^f.$$
 Let $a_P$ be  the canonical $(s,p)$-Souza's atom on $P$ and choose $x_P\in P$.  Denote
$$k_P^f= \frac{ \tilde{a}_P^f(x_P)}{|P|^{s-1/p}} \tilde{k}_P^f= \frac{1}{|P|^{s-1/p}}\sum_{S\in \mathcal{H}_Q}   d_S^f \phi_S(x_P)=\frac{1}{|P|^{s-1/p}} \int   f \sum_{S\in \mathcal{H}_Q} \phi_S(x_P)\phi_S \ dm.$$
In particular, \change{we added a comma}  for  every $P$ 
 $$f\mapsto k^f_P$$
 extends to a bounded linear functional in $L^1$. We have $|k_P^f|\leq \tilde{k}_P^f$ and 
\begin{align} f &= k_I^f a_I  + \sum_k  \sum_{Q\in \mathcal{P}^{k}}\sum_{\substack{ _{P \in \mathcal{P}^{k+1}}\\_{P\subset Q}}} k_P^f a_P \nonumber \\
 \label{sumf}   &= \sum_i  \sum_{Q\in \mathcal{P}^{i}} k_Q^f a_Q.
\end{align}
 where this series converges unconditionally in $L^\beta$. Here $a_P$ is the canonical Souza's atom. We will call the r.h.s. of (\ref{sumf}) the {\bf standard atomic  representation} of $f$. Let
 $$  N_{st}(f)=|k_I^f|+ \Big(  \sum_{k\geq 1} \big( \sum_{Q\in \mathcal{P}^{k}}  |k_Q^f|^p    \big)^{q/p}\Big)^{1/q}.$$
 
 \subsection{Mean oscillation}
 Define for $p\in [1,\infty)$
$$osc_p(f,Q)= \inf_{c \in \mathbb{C}} (\int_Q |f(x)-c|^p \ dm(x) )^{1/p},$$
and 
$$osc_\infty(f,Q)=\inf_{c \in \mathbb{C}} |f-c|_{L^\infty(Q)}.$$
Denote for every $p\in [1,\infty)$ and $q \in [1,\infty]$
\begin{equation} \label{oc} osc^s_{p,q}(f)= \Big(  \sum_k \big(  \sum_{Q\in \mathcal{P}^k} |Q|^{-sp} osc_p(f,Q)^p  \big)^{q/p} \Big)^{1/q},\end{equation} 
with the obvious adaptation for $q=\infty$.  Let $$ N_{osc}(f)=|I|^{-s}|f|_p+osc^s_{p,q}(f).$$

 \subsection{These norms are equivalent} We have 
 
 \begin{theorem} \label{alte} Suppose $s> 0$, $p\in [1,\infty)$ and $q\in [1,\infty]$. Each one of the norms  $|f|_{\mathcal{B}^s_{p,q}}$, $N_{st}(f)$, $N_{haar}(f)$, $N_{osc}(f)$ is finite if and only if  $f\in \mathcal{B}^s_{p,q}$. Furthermore these norms are equivalent on $\mathcal{B}^s_{p,q}$. Indeed
\begin{equation}\label{in1}  |f|_{\mathcal{B}^s_{p,q}}\leq N_{st}(f),\end{equation}
\begin{equation}\label{in2}N_{st}(f)\leq \Cll{e} N_{haar}(f),\end{equation} 
\begin{equation}\label{in3}N_{haar}(f)\leq \Crr{c2} N_{osc}(f),\end{equation}
\begin{equation}\label{in4}N_{osc}(f)\leq \Cll{no}  |f|_{\mathcal{B}^s_{p,q}}.\end{equation}
where
$$\Crr{e}=1+\Crr{c2} \max\{ \Crr{maior}^{1/p-s}, \Crr{menor}^{1/p-s}\}  \Crr{menor}^{-2-1/p},$$
 $$\Crr{no}=\Crr{mult1}^{1+1/p}  \Crr{co}(t,q, ( |\mathcal{P}^k|^{sp} )_k )|I|^{-s}+\frac{1}{1-\Crr{maior}^s}.$$
 \end{theorem}
 \begin{proof} The inequality (\ref{in1}) is obvious. To simplify the notation we write $d_S, k_P$ instead of $d_S^f, k_P^f$.\ \\ \\
 \noindent {\bf Proof of (\ref{in2}).} The  number of terms in the r.h.s. of (\ref{hh}) depends only on the geometry of $\mathcal{P}$. Indeed 
 \begin{equation}\label{numero}  \sup_{Q \in \mathcal{P}} \sum_{S=(S_1,S_2) \in \mathcal{H}_Q} \#(S_1 \cup S_2) \leq \frac{1}{\Crr{menor}^2}.\end{equation} 
  Consider the standard atomic representation of $f$ given by  (\ref{sumf}). Note that  by (\ref{estcsp})
\begin{align*} 
\sum_{Q\in \mathcal{P}^{k}} \sum_{\substack{ _{P \in \mathcal{P}^{k+1}}\\ _{P\subset Q}}} |k_P|^p &\leq   \sum_{Q\in \mathcal{P}^{k}} \Big(\sum_{_{S\in \mathcal{H}_Q}}   \sum_{\substack{_{P\in S_1\cup S_2}\\_{S=(S_1,S_2)}}} |c_{S,P}| \Big)^p  \\
&\leq  \Crr{c2}^p \max\{ \Crr{maior}^{1-sp}, \Crr{menor}^{1-sp}\}  \sum_{Q\in \mathcal{P}^{k}}  |Q|^{_{1-sp -\frac{p}{2}}}   \Big( \sum_{_{S\in \mathcal{H}_Q}}   \sum_{\substack{_{P\in S_1\cup S_2}\\_{S=(S_1,S_2)}}}   |d_S|\Big)^p   \\
&\leq   \Crr{c2}^p  \max\{ \Crr{maior}^{1-sp}, \Crr{menor}^{1-sp}\}\Crr{menor}^{-2p} \sum_{Q\in \mathcal{P}^{k}}  |Q|^{_{1-sp -\frac{p}{2}}}   \sum_{_{S\in \mathcal{H}_Q}}   \sum_{\substack{_{P\in S_1\cup S_2}\\_{S=(S_1,S_2)}}}   |d_S|^p   \\
&\leq   \Crr{c2}^p  \max\{ \Crr{maior}^{1-sp}, \Crr{menor}^{1-sp}\} \Crr{menor}^{-2p-1} \sum_{Q\in \mathcal{P}^{k}}  |Q|^{_{1-sp -\frac{p}{2}}}   \sum_{_{S\in \mathcal{H}_Q}}   |d_S|^p.
\end{align*}
for every $k$. Consequently
\begin{align*} 
& |k_I|+ \Big( \sum_{k\geq 1} \big( \sum_{P \in \mathcal{P}^{k}} |k_P|^p\big)^{q/p}  \Big)^{1/q} \leq \\
&\leq   |I|^{_{\frac{1}{p}-s-1/2}} |d_I|  + \Crr{c2} \max\{ \Crr{maior}^{_{\frac{1}{p}-s}}, \Crr{menor}^{_{\frac{1}{p}-s}}\}\Crr{menor}^{_{-2-\frac{1}{p}}}  \Big( \sum_k  \big(  \sum_{Q\in \mathcal{P}^{k}}  |Q|^{_{1-sp -\frac{p}{2}}}   \sum_{_{S\in \mathcal{H}_Q}}   |d_S|^p\big)^{q/p} \Big)^{1/q}.
\end{align*}
This completes the proof of (\ref{in2}). \\ \\
\noindent {\bf Proof of (\ref{in3}).} Note that
$$|d_I|\leq  \int_I |f| |\phi_I|\ dm\leq |f|_p |I|^{1/2-1/p}.$$Given $\epsilon > 0$ and  $Q \in \mathcal{P}$, choose $c_Q\in Q$ such that 
$$\Big( \int_Q |f - c_Q|^p \ dm \Big)^{1/p} \leq (1+\epsilon)  osc_p(f,Q).$$
 Since $\phi_S$ has zero mean on $Q$ for every $S\in \mathcal{H}_Q$ we have 
\begin{align*} 
&\big(\sum_{Q\in \mathcal{P}^{k}}|Q|^{1-sp -\frac{p}{2}}  \sum_{S\in \mathcal{H}_Q}  |d_S|^p\big)^{1/p} \\
&\leq\big(\sum_{Q\in \mathcal{P}^{k}}  |Q|^{_{1-sp -\frac{p}{2}}}  \sum_{S\in \mathcal{H}_Q}  |\int f \phi_S \ dm |^p \big)^{1/p} \\
&\leq\big(\sum_{Q\in \mathcal{P}^{k}} |Q|^{_{1-sp -\frac{p}{2}}}  \sum_{S\in \mathcal{H}_Q}  |\int_Q f \phi_S -c_Q \phi_S\ dm |^p \big)^{1/p} \\
&\leq\big(\sum_{Q\in \mathcal{P}^{k}}  |Q|^{_{1-sp -\frac{p}{2}}}  \sum_{S\in \mathcal{H}_Q} (\int_Q |f - c_Q| |\phi_S| \ dm )^p \big)^{1/p} \\
&\leq \Crr{c2}\big(\sum_{\substack{_{Q\in \mathcal{P}^{k}}\\ _{Q\subset J}}}  |Q|^{_{1-sp -\frac{p}{2}}}   \sum_{S\in \mathcal{H}_Q} ((1+\epsilon) osc_p(f,Q) |Q|^{1/p'-1/2} )^p \big)^{1/p} \\
&\leq \Crr{c2}(1+\epsilon) \big(\sum_{Q\in \mathcal{P}^{k}} |Q|^{_{1-sp +p/p'-p}}   \sum_{S\in \mathcal{H}_Q} osc_p(f,Q)^p  \big)^{1/p} \\
&\leq \Crr{c2} (1+\epsilon)\big(\sum_{Q\in \mathcal{P}^{k}} |Q|^{_{-sp}}  \sum_{S\in \mathcal{H}_Q}  osc_p(f,Q)^p  \big)^{1/p}.
\end{align*}
Since $\epsilon$ is arbitrary, this concludes the proof of (\ref{in3}). \\ \\
\noindent {\bf Proof of (\ref{in4}).} Finally note that  if $f \in \mathcal{B}^s_{p,q}$ and $\epsilon > 0$ then there is a $\mathcal{B}^s_{p,q}$-representation of $f$ 
$$f= \sum_{P \in \mathcal{P}} k_P a_P.$$
 such that 
 $$ \big( \sum_{k=0}^{\infty} (\sum_{Q \in \mathcal{P}^k}    |k_Q|^p)^{q/p} \big)^{1/q} \leq (1+\epsilon)  |f|_{\mathcal{B}^s_{p,q}}.$$
For each $J \in \mathcal{P}^{k_0}$, choose  $x_J \in J$. Then 
\begin{align}  & \Big( \sum_{J\in \mathcal{P}^{k_0} } |J|^{-sp }  osc_p(f,J)^p  \Big)^{1/p} \nonumber \\
&\leq   \Big(  \sum_{J\in \mathcal{P}^{k_0} }  |J|^{-sp} \int_J |f(x) -    \sum_{\substack{_{Q \in \mathcal{P}}\\ _{J \subset Q}}} k_Q a_Q(x_J)|^p \ dm   \Big)^{1/p} \nonumber \\
&\leq  \Big( \sum_{J\in \mathcal{P}^{k_0} }  |J|^{-sp} \int |\sum_{k> k_0}  \sum_{\substack{_{R \in \mathcal{P}^k}\\_{R  \subset J}}} k_R a_R |^p \ dm  \Big)^{1/p}  \nonumber \\
&\leq \Big( \int |  \sum_{k>  k_0}\sum_{J\in \mathcal{P}^{k_0} }    |J|^{-s}  \sum_{\substack{_{R \in \mathcal{P}^k}\\_{R  \subset J}}} k_R a_R |^p \ dm  \Big)^{1/p}  \nonumber \\
&\leq  \sum_{k>  k_0} \Big( \int | \sum_{J\in \mathcal{P}^{k_0} }    |J|^{-s}  \sum_{\substack{_{R \in \mathcal{P}^k}\\_{R  \subset J}}} k_R a_R |^p \ dm  \Big)^{1/p}  \nonumber \\
&\leq  \sum_{k> k_0} \Big(   \sum_{J\in \mathcal{P}^{k_0} }  |J|^{-sp}\int   |    \sum_{\substack{_{R \in \mathcal{P}^k}\\_{R  \subset J}}} k_R a_R |^p \ dm  \Big)^{1/p}  \nonumber \\
&\leq  \sum_{k> k_0} \Big(   \sum_{J\in \mathcal{P}^{k_0} } \Big( \frac{\sup \{ |R|\colon R \in \mathcal{P}^k, R  \subset J \}}{ |J|}\Big)^{sp}  \sum_{\substack{_{R \in \mathcal{P}^k}\\_{R  \subset J}}} |k_R|^p  \Big)^{1/p}  \nonumber \\
&\leq  \sum_{k> k_0} \Big(  \Crr{maior}^{sp(k-k_0)}   \sum_{J\in \mathcal{P}^{k_0} }  \sum_{\substack{_{R \in \mathcal{P}^k}\\_{R  \subset J}}} |k_R|^p  \Big)^{1/p}  \nonumber \\
&\leq  \sum_{k> k_0}  \Crr{maior}^{s(k-k_0)}  \big( \sum_{R \in \mathcal{P}^k} |k_R|^p  \big)^{1/p}.  \nonumber 
\end{align}
This is  a convolution, so 
\begin{align*} \Big( \sum_{k_0} \Big( \sum_{J\in \mathcal{P}^{k_0} } |J|^{-sp }  osc_p(f,J)^p \Big)^{q/p} \Big)^{1/q} &\leq \frac{1+\epsilon}{1-\Crr{maior}^s} |f|_{\mathcal{B}^s_{p,q}}.\end{align*}
and since $\epsilon > 0$ is arbitrary,  by Proposition \ref{lp} we obtain 
\begin{align*} &|I|^{-s}|f|_p + \Big( \sum_{k_0} \Big( \sum_{J\in \mathcal{P}^{k_0} } |J|^{-sp }  osc_p(f,J)^p  \Big)^{q/p} \Big)^{1/q} \\
&\leq \Big(\Crr{mult1}^{1+1/p}  \Crr{co}(t,q, ( |\mathcal{P}^k|^{sp} )_k )|I|^{-s}+\frac{1}{1-\Crr{maior}^s}\Big) |f|_{\mathcal{B}^s_{p,q}}.\end{align*}
This proves (\ref{in4}).
 \end{proof}
 
 The following is an important consequence of this section. 
 \begin{corollary} \label{fou} For each $P\in \mathcal{P}$ there exists a linear functional in $L^1$ 
 $$f  \mapsto k_P^f$$
 with the following property. The so-called  standard $\mathcal{B}^s_{p,q}$-representation of $f\in \mathcal{B}^s_{p,q}$ given by 
$$ f = \sum_k \sum_{P \in \mathcal{P}^{k}} k_P^f a_P,$$
satisfies 
$$\Big(\sum_k  \big( \sum_{P\in \mathcal{P}^k}  |k_P^f|^p\big)^{q/p}\Big)^{1/p}\leq \Crr{c2}  \Crr{e}\Crr{no} |f|_{\mathcal{B}^s_{p,q}}.$$

 \end{corollary}

  \section{Alternative characterizations II: Messing with atoms.}
 Here we move to  alternative descriptions of $\mathcal{B}^s_{p,q}$, \change{we added a comma} which are quite different from \change{we replaced preposition "of" by "from"} those in  Section \ref{alt1}. Instead of choosing  a definitive representation of elements of $\mathcal{B}^s_{p,q}$, \change{we added a comma} we indeed give atomic decompositions of $\mathcal{B}^s_{p,q}$ using far more general classes of atoms.
 
 \subsection{Using Besov's atoms}

The advantage of Besov's atoms is that it is a wide and general class of atoms, that includes even unbounded functions. They can  be considered in {\it every} measure space   endowed with a good grid as in Section \ref{partition}. Moreover  in appropriate settings it contains  H\"older and bounded variations atoms, which will be quite useful in the get other characterizations of $\mathcal{B}^s_{p,q}(\mathcal{P},\mathcal{A}^{sz}_{s,p})$. The atomic decompositions of  $B^s_{p,q}(\mathbb{R}^n)$ by Besov's atoms  were considered  by Triebel \cite{ns} in the case $s>0$, $p=q\in [1,\infty]$ and 
by Schneider and Vyb\'\i ral \cite{corjan} in the case $s>0$, $p,q\in [1,\infty]$. They are refered there as "non-smooth atomic decompositions".

Let $s < \beta$ and $\tilde{q}\in [1,\infty]$. A  {\bf $(s,\beta,p,\tilde{q})$-Besov atom}  on the interval $Q$ is a function $a \in \mathcal{B}^\beta_{p,\tilde{q}}(Q, \mathcal{P}_Q, \mathcal{A}^{sz}_{\beta,p})$ such that  $a(x)=0$ for $x \not\in Q$,
and
$$|a|_{\mathcal{B}^\beta_{p,\tilde{q}}(Q, \mathcal{P}_Q, \mathcal{A}^{sz}_{\beta,p})}\leq \frac{1}{\Crr{ba}} |Q|^{s-\beta}.$$
where 
$$\Cll{ba}=\Crr{mult1}^{1+1/p}  \Big( \sum_k \Crr{maior}^{k\beta\tilde{q}}    \Big)^{1/\tilde{q}}\geq 1.$$
The family of $(s,\beta,p,\tilde{q})$-Besov atoms supported on $Q$ will be denoted by  $\mathcal{A}^{bv}_{s,\beta,p,\tilde{q}}(Q)$.  Naturally $\mathcal{A}^{sz}_{s,p}(Q)\subset \mathcal{A}^{bs}_{s,\beta,p,\tilde{q}}(Q)$. By Proposition \ref{lp}  we have
\begin{eqnarray}
|a|_{p}&\leq&  \frac{\Crr{mult1}^{1+1/p}}{\Crr{ba}}  \Big( \sum_k |\mathcal{P}^k_Q|^{\beta\tilde{q}}    \Big)^{1/\tilde{q}} |a|_{\mathcal{B}^\beta_{p,\tilde{q}}(Q, \mathcal{P}_Q, \mathcal{A}^{sz}_{\beta,p})}\\
&\leq&  |Q|^{s-\beta} |Q|^{\beta} =  |Q|^s=  |Q|^{s-\frac{1}{\infty}}. 
\end{eqnarray} 
so a $(s,\beta,p,\tilde{q})$-Besov atom is an atom of  type $(s,p,1)$. \\ \\

The following result says there are many ways to define $\mathcal{B}^s_{p,q}$ using various classes of atoms. 

\begin{proposition}[Souza's atoms and  Besov's  atoms]\label{besova}Let $\mathcal{P}$ be a good grid. Let $\mathcal{A}$ be a class of $(s,p,u)$-atoms , with $u\geq 1$, such that  for some $s< \beta$, $\tilde{q}\in [1,\infty]$, and $\Crr{56}$, $\Crr{566} \geq 0$ we have that  for every $Q\in \mathcal{P}$ 
$$\frac{1}{\Cll{56}} \mathcal{A}^{sz}_{s,p}(Q) \subset \mathcal{A}(Q)\subset  \Cll{566} \mathcal{A}^{bs}_{s,\beta,p,\tilde{q}}(Q).$$
Then 
 $$\mathcal{B}^s_{p,q}(\mathcal{P},\mathcal{A}^{sz}_{s,p})=\mathcal{B}^s_{p,q}(\mathcal{P},\mathcal{A})=\mathcal{B}^s_{p,q}(\mathcal{P},\mathcal{A}^{bs}_{s,\beta,p,\tilde{q}}).$$ 
 Moreover
$$|f|_{\mathcal{B}^s_{p,q}(\mathcal{A})} \leq \Crr{56} |f|_{\mathcal{B}^s_{p,q}(\mathcal{A}^{sz}_{s,p})} \  and \  |f|_{\mathcal{B}^s_{p,q}(\mathcal{A}^{sz}_{s,p})}  \leq \frac{\Crr{566}}{1-\Crr{maior}^{\beta-s}}|f|_{\mathcal{B}^s_{p,q}(\mathcal{A})}.$$
\end{proposition} 
\begin{proof} The  first inequality is obvious. To prove the second inequality, recall that due to \change{we added "to"}  Proposition \ref{trans} it is enough to show  the following claim
\vspace{5mm}

\noindent {\it Claim. Let  $b_J$ be a  $(s,\beta,p,\tilde{q})$-Besov atom on $J\in \mathcal{P}^j$.  Then for every $P\subset J$ with $P\in \mathcal{P}$ there is  $m_P \in \mathbb{C}$ such that 
$$b_J= \sum_{P\subset J} m_P a_P,$$
where $a_P$ is the canonical $(s,p)$-Souza's atom on $P$ and 
$$  \Big( \sum_{P\in \mathcal{P}^k, P \subset J} |m_P|^p\Big)^{1/p} \leq   \frac{2}{\Crr{ba}}  \Crr{maior}^{(k-j)(\beta-s)}.$$}
\vspace{5mm}

\noindent Indeed, since
$$|b_J|_{\mathcal{B}^\beta_{p,\tilde{q}}(J,\mathcal{P},\mathcal{A}^s_{\beta,p})}\leq \frac{1}{\Crr{ba}} |J|^{s-\beta},$$
there exists a $\mathcal{B}^\beta_{p,\tilde{q}}(J,\mathcal{P},\mathcal{A}^s_{\beta,p})$-representation
$$b_J=\sum_{P\in \mathcal{P}, P\subset J} c_P d_P,$$
where $d_P$ is the canonical $(\beta,p)$-Souza's atom on $P$ and 
$$\big( \sum_{P\in \mathcal{P}^k, P \subset J} |c_P|^p\big)^{1/p} \leq \Big(\sum_i \big( \sum_{P\in \mathcal{P}^i, P\subset J} |c_P|^p  \big)^{\tilde{q}/p}    \Big)^{1/\tilde{q}}\leq \frac{2}{\Crr{ba}}   |J|^{s-\beta}.$$
Then 
$$a_P = |P|^{s-\beta}d_P$$
is a $(s,p)$-Souza's atom and 
$$b_J=\sum_{P\in \mathcal{P}, P\subset J} m_P a_P,$$
with $m_P = c_P |P|^{\beta-s}$ and 
\begin{align} \big( \sum_{P\in \mathcal{P}^k, P \subset J} |m_P|^p\big)^{1/p} &=  \big( \sum_{P\in \mathcal{P}^k, P \subset J} \Big(\frac{|P|}{|J|}\Big)^{p(\beta-s)} |J|^{p(\beta-s)}|c_P|^p\big)^{1/p}  \nonumber \\
&\leq \Crr{maior}^{(k-j)(\beta-s)}  |J|^{\beta-s}  \big(  \sum_{P\in \mathcal{P}^k, P \subset J} |c_P|^p\big)^{1/p}  \nonumber    \\
&\leq  \frac{2}{\Crr{ba}}  \Crr{maior}^{(k-j)(\beta-s)}. \nonumber
\end{align}
\end{proof}


 \subsection{Using H\"older atoms} \label{hatoms2} Suppose that $I$ is a quasi-metric space with a quasi-distance $d(\cdot,\cdot)$ and a good grid satisfying the assumptions in Section \ref{hatoms}. 

\begin{proposition} \label{hold} Suppose
$$0< s  < \beta < \tilde{\beta},$$
$p\in [1,\infty)$ and  $\tilde{q}\in [1,\infty].$  For every $Q\in \mathcal{P}$ we have 
\begin{equation}\label{final1} \Cll{at3} \mathcal{A}^{h}_{s,\tilde{\beta},p}(Q)\subset \mathcal{A}^{bs}_{s,\beta,p,\tilde{q}}(Q),\end{equation}  for some $\Crr{at3} > 0$. Moreover
\begin{equation}\label{final2} \Cll{at4} \mathcal{A}^{h}_{s,\beta,p}(Q)\subset \mathcal{A}^{bs}_{s,\beta,p,\infty}(Q),\end{equation}  for some $\Crr{at4} > 0.$ In particular
\begin{equation}\label{inclusion1}\mathcal{B}^s_{p,q}=\mathcal{B}^s_{p,q}(\mathcal{A}^{h}_{s,\tilde{\beta},p})=\mathcal{B}^s_{p,q}(\mathcal{A}^{h}_{s,\beta,p}(Q))\end{equation}
and the corresponding norms are equivalent.
\end{proposition} 
\begin{proof} Let $\phi \in \mathcal{A}^{h}_{s,\tilde{\beta},p}(Q)$.  Then $\phi$ has a continuous extension to $\overline{Q}$.  So firstly we  assume that    $\phi\geq 0$ has a continuous extension to $\overline{Q}$. Define 
$$c_Q= \min \phi(Q) |Q|^{1/p-\beta}.$$
and for every  $P\subset W\subset Q \in \mathcal{P}^j$ with $P \in \mathcal{P}^{k+1}$ and $W \in \mathcal{P}^{k}$ define
\begin{equation} \label{cv} c_P= (\inf \phi (P) -\inf \phi (W))|P|^{1/p-\beta}\end{equation}
Of course in this case $c_P\geq 0$ and
\begin{align*} |c_P| &\leq (\inf \phi (P) -\inf \phi (W))|P|^{1/p-\beta} \\
&\leq  |Q|^{s-1/p-\tilde{\beta}}  (diam \ W )^{\tilde{\beta}D} |P|^{1/p-\beta}\\
& \leq \frac{\Crr{kki}^{D\tilde{\beta}}}{\Crr{hh}^{\tilde{\beta}}}\Big( \frac{|P|}{|Q|}\Big)^{1/p} |Q|^{s-\tilde{\beta}} |P|^{\tilde{\beta}-\beta}\\
& \leq \frac{\Crr{kki}^{D\tilde{\beta}}}{\Crr{hh}^{\tilde{\beta}}} \Cll{super}  \Big( \frac{|P|}{|Q|}\Big)^{1/p} |Q|^{s-\beta} \Big(\frac{|P|}{|Q|}\Big)^{\tilde{\beta}-\beta}.
\end{align*}
Here $\Crr{super}= \sup_{Q\in \mathcal{P}}|Q|^{\tilde{\beta}-\beta}.$
Consequently 
\begin{align}\label{adapt} \sum_{\substack{P\in \mathcal{P}^{k+1}\\ P\subset Q}} |c_P|^p   & \leq  |Q|^{p(s-\beta)}   \Crr{maior}^{(k+1-j)p(\tilde{\beta}-\beta)} \sum_{\substack{P\in \mathcal{P}^{k+1}\\ P\subset Q}} \frac{ \Crr{super}^p\Crr{kki}^{D\tilde{\beta}p}}{\Crr{hh}^{p \beta}} \frac{|P|}{|Q|} \nonumber \\
&\leq |Q|^{p(s-\beta)} \frac{ \Crr{super}^p \Crr{kki}^{D\tilde{\beta}p}}{\Crr{hh}^{\beta p}} \Crr{maior}^{(k+1-j)p(\tilde{\beta}-\beta)}.\end{align}
so
\begin{align*}\Big(  \sum_k \big( \sum_{\substack{P\in \mathcal{P}^{k+1}\\ P\subset Q}} |c_P|^p\big)^{\tilde{q}/p} \Big)^{1/{\tilde{q}}}  &\leq  \frac{ \Crr{super}\Crr{kki}^{D\tilde{\beta}}}{\Crr{hh}^\beta} \Big(\frac{1}{1-\Crr{maior}^{\tilde{q} (\tilde{\beta}-\beta) }}\Big)^{1/\tilde{q}}  |Q|^{s-\beta}.\end{align*}
This implies that 
$$\tilde{\phi}=\sum_{\substack{ P\in \mathcal{P}\\ P\subset Q}} c_P a_P, $$
where $a_P$ is the canonical $(\beta,p)$-atom on $P$,  is a $\mathcal{B}^\beta_{p,\tilde{q}}(Q,\mathcal{P}_Q\mathcal{A}^{sz}_{\beta,p})$-representation of  a function $\tilde{\phi}$.  From (\ref{cv}) it follows that
$$\sum_{k\leq N} \sum_{P\in \mathcal{P}^k} c_P a_P(x)= \min \phi(W)$$
for every $x \in W\in \mathcal{P}^{N}.$ In particular
$$\lim_{N\rightarrow \infty} \sum_{k\leq N} \sum_{P\in \mathcal{P}^k} c_P a_P(x) =\phi(x).$$
for almost every $x$, so $\tilde{\phi}=\phi$. So
\begin{equation}\label{inbv} |\Crr{at3}\phi|_{\mathcal{B}^\beta_{p,\tilde{q}}(Q,\mathcal{P}_Q,\mathcal{A}^{sz}_{\beta,p})}\leq  \frac{1}{2} |Q|^{s-\beta}\end{equation} 
where
$$\Crr{at3}= \Big( 2\frac{\Crr{super} \Crr{kki}^{D\tilde{\beta}}}{\Crr{hh}^\beta} \Big(\frac{1}{1-\Crr{maior}^{\tilde{q}(\tilde{\beta}-\beta)}}\Big)^{1/\tilde{q}}\Big)^{-1}.$$

In the general case, note that $\phi_+(x)=\max\{ \phi,0\}, \phi_-(x)=\min\{-\phi,0\} \in \mathcal{A}^{h}_{s,\tilde{\beta},p}(Q)$ and $\phi=\phi_+-\phi_-$. Applying (\ref{inbv}) to $\phi_+$ and $\phi_-$ we obtain (\ref{final1}).

The second inclusion (\ref{final2}) in the proposition can be obtained taking $\tilde{\beta}=\beta$ in (\ref{adapt}) and a few modifications in the above argument.  By Proposition \ref{besova} we have (\ref{inclusion1}).
\end{proof}

\subsection{Using bounded variation atoms}  Now suppose that  $I$ is an interval of $\mathbb{R}$ of length $1$, $m$ is the Lebesgue measure on it and the partitions in $\mathcal{P}$ are partitions by intervals. 

\begin{proposition}  If 
$$0< s\leq \beta \leq \frac{1}{p}$$
then  $$\Crr{at1} \mathcal{A}^{bv}_{s, \beta,p}(Q)\subset \mathcal{A}^{bs}_{s,\beta,p,\infty}(Q).$$
for every $Q\in \mathcal{P}.$
If $$0< s\leq \beta < \tilde{\beta} \leq   \frac{1}{p}$$           
then  $$\Crr{at2} \mathcal{A}^{bv}_{s, \tilde{\beta},p}(Q)\subset \mathcal{A}^{bs}_{s,\beta,p,q}(Q)$$
for every $q \in [1,\infty]$. In particular
$$\mathcal{B}^s_{p,q}= \mathcal{B}^s_{p,q}(\mathcal{A}^{bv}_{s, \beta,p})=\mathcal{B}^s_{p,q}(\mathcal{A}^{bv}_{s, \tilde{\beta},p}).$$
and the corresponding norms are equivalents.
\end{proposition} 
\begin{proof} Suppose
$$0< s \leq  \beta  \leq \tilde{\beta}\leq 1$$
Let  $Q\in \mathcal{P}^j$ and $a_Q\in \mathcal{A}^{bv}_{s, \tilde{\beta},p}(Q)$. We have
$$|a_Q|_p\leq (   |Q||Q|^{sp-1}   )^{1/p}=|Q|^s\leq \Cll{supe} |Q|^{s-\beta},$$
where $\Crr{supe}= \sup_{Q\in \mathcal{P}} |Q|^\beta.$
Note that  for every $k\geq j$
\begin{align*}
&\sum_{W\in \mathcal{P}^k} |W|^{-\beta p} osc_p(a_Q,W)^p  \\
 &\leq  
\sum_{W \subset Q, W\in \mathcal{P}^k} |W|^{1-\beta p} osc_\infty(a_Q,W)^p  \\
&\leq \big( \sum_{W \subset Q, W\in \mathcal{P}^k}  |W|^{\frac{1-\beta p}{1-\tilde{\beta} p}} \big)^{1-\tilde{\beta} p}  \big( \sum_{W \subset Q, W\in \mathcal{P}^k}   osc_\infty(a_Q,W)^{1/\tilde{\beta}}   \big)^{\tilde{\beta} p} \\ 
&\leq   \Big( \max_{W \subset Q, W\in \mathcal{P}^k} |W|^{\frac{(\tilde{\beta}-\beta) p}{1-\tilde{\beta} p} } \Big)^{1-\tilde{\beta} p}  \big( \sum_{W \subset Q, W\in \mathcal{P}^k}  |W| \big)^{1-\tilde{\beta} p}   (var_{1/\tilde{\beta}}(a_Q,Q))^{p}  \\
&\leq \Crr{maior}^{(k-j)(\tilde{\beta}-\beta)p}|Q|^{(\tilde{\beta}-\beta)p}   |Q|^{1-\tilde{\beta} p} |Q|^{sp-1}\\
&\leq \Crr{maior}^{(k-j)(\tilde{\beta}-\beta)p}  |Q|^{(s-\beta) p}.
\end{align*}
Note that in the case  $\tilde{\beta}p=1$ the argument above needs a simple modification. For $k < j$ let $W_Q\in \mathcal{P}^k$ be such that $Q\subset W_Q$. Then 
\begin{align*}
&\sum_{W\in \mathcal{P}^k} |W|^{-\beta p} osc_p(a_Q,W)^p \leq  |W_Q|^{-\beta p} |Q|  |Q|^{sp -1}   \\
&\leq  |W_Q|^{-\beta p} |Q|^{sp}  =  \Big( \frac{|Q|}{|W_Q|}\Big)^{\beta p}  |Q|^{(s-\beta)p}\leq \Crr{maior}^{(j-k)\beta p}  |Q|^{(s-\beta)p}.
\end{align*}
By Theorem \ref{alte} we have that if $\beta=\tilde{\beta}$ then  $\Cll{at1} \mathcal{A}^{bv}_{s, \beta,p}(Q)\subset \mathcal{A}^{bs}_{s,\beta,p,\infty}(Q)$ for some $\Crr{at1} > 0$, and if $\beta < \tilde{\beta}$ we have   that for every $q \in [1,\infty)$ 
\begin{align*}
&|I|^{-s}|a_Q|_p+\Big( \sum_{k} \big( \sum_{W\in \mathcal{P}^k} |W|^{-sp} osc_p(a_Q,W)^p \big)^{q/p}\Big)^{1/q} \\
&\leq \Big(  \Crr{supe} |I|^{-s}+\big( \frac{1}{1-\lambda^{q(\tilde{\beta}-\beta)}}    +  \frac{1}{1-\Crr{maior}^{q\beta}}  \big)^{1/q} \Big)   |Q|^{s-\beta} .\\
\end{align*}
so $\Cll{at2} \mathcal{A}^{bv}_{s, \tilde{\beta},p}(Q)\subset \mathcal{A}^{bs}_{s,\beta,p,q}(Q)$ for some $\Crr{at2} > 0$.
\end{proof}

\section{Dirac's approximations} 
We will use the Haar basis and notation defined by Section \ref{haar}. For every $x_0 \in I$ and $k_0 \in \mathbb{N}$  define the finite family 
$$\mathcal{S}_{x_0}^{k_0} =   \{ S \colon \ S=(S_1,S_2) \in \mathcal{H}(Q),\ with \ Q \in \mathcal{P}^k, \ k < k_0, \ x_0 \in \cup_{a=1,2} \cup_{P \in S_{a}} P   \}$$
Let $N(x_0,k_0)=\# \mathcal{S}_{x_0}^{k_0}$. Then we can enumerate the elements $$S^1, S^2, \dots, S^{N(x_0,k_0)}$$ of $\mathcal{S}_{x_0}^{k_0}$ 
such that $S^i=(S^i_1,S^i_2)$ satisfies
$$x_0 \in \cup_{P \in S^i_{a_i}} P $$
for some $a_i \in \{1,2\}$ and
$$\cup_{P \in S^{i+1}_1\cup S^{i+1}_2} P     \subset \cup_{Q \in S^{i}_1\cup S^{i}_2} Q$$ 
for every $i$. 
Let 
$$\psi_0    =   \frac{1_{I}}{|I|}$$
and define for $i >0$ 
\begin{eqnarray*}
\psi_{i}&=& (-1)^{a_{i}+1}\frac{\sum_{R \in S^{i}_{3-a_{i}}}|R|}{\sum_{Q \in S^{i-1}_{a_{i-1}}} |Q|} \big( \frac{ \sum_{P \in S^{i}_1}1_{P}}{\sum_{P \in S^{i}_1}|P|} -  \frac{ \sum_{R \in S^{i}_2}1_{R}}{\sum_{R \in S^{i}_2}|R|} \big)  \\
&=&   (-1)^{a_{i}+1} \frac{\sum_{R \in S^{i}_{3-a_{i}}}|R|}{\sum_{Q \in S^{i-1}_{a_{i-1}}} |Q|} m_{S^i} \phi_i.
\end{eqnarray*}
One can prove by induction on $j$ that for $j > 0$
$$\sum_{i=0}^j  \psi_i = \frac{\sum_{P \in S^{j}_{a_j}}1_{P} }{\sum_{P \in S^{j}_{a_j}}|P|}.$$
and in particular 
$$\sum_{i=0}^{N(x_0,k_0)}  \psi_i (x_0)= \frac{1_{Q_{k_0}} }{|Q_{k_0}|},$$
where $x_0 \in Q_{k_0} \in \mathcal{P}^{k_0}.$ In other words 
\begin{equation}\label{part}  \frac{1}{ |I|^{1/2} }\phi_{I} + \sum_{i=1}^{N(x_0,k_0)}  (-1)^{a_i+1}  \Big( \frac{\sum_{R \in S^{i}_{3-a_i}}|R|}{\sum_{Q \in S^{i-1}_{a_{i-1}}} |Q|}  \Big)  m_{S^i}    \phi_{S^i}=\frac{1_{Q_{k_0}} }{|Q_{k_0}|},\end{equation}  
Note that 
\begin{align}
\Big( \frac{\sum_{R \in S^{i}_{3-a_i}}|R|}{\sum_{Q \in S^{i-1}_{a_{i-1}}} |Q|}  \Big) m_{S^i}&= \Big( \frac{\sum_{R \in S^{i}_{3-a_i}}|R|}{\sum_{Q \in S^{i-1}_{a_{i-1}}} |Q|}  \Big) \Big(    \frac{\sum_{R \in S_2^i} |R| + \sum_{P \in S_1^i} |P|  }{(\sum_{R \in S_2^i} |R| ) (\sum_{P \in S_1^i} |P|)}   \Big)^{1/2} \nonumber \\
\label{igual} &= \Big( \frac{\sum_{R \in S^{i}_{3-a_i}}|R|}{\sum_{P \in S^{i}_{a_i}} |P|}  \Big)^{1/2}  \frac{1}{(\sum_{Q \in S^{i-1}_{a_{i-1}}} |Q|)^{1/2}}  
\end{align}
Multiplying (\ref{part}) by $f$ and integrating it  term by term, and using (\ref{igual}) we obtain 
$$\frac{d_{I}}{|I|^{1/2}} + \sum_{i=1}^{N(x_0,k_0)}  (-1)^{a_i+1}  \Big( \frac{\sum_{R \in S^{i}_{3-a_i}}|R|}{\sum_{P \in S^{i}_{a_i}} |P|}  \Big)^{1/2}  \frac{1}{(\sum_{Q \in S^{i-1}_{a_{i-1}}} |Q|)^{1/2}}    d _{S^i} =\int f \cdot \frac{1_{Q_{k_0}}}{|Q_{k_0}|} \ dm.$$

\noindent If  $f \in L^\beta$, with $\beta > 1$,   it  can be written as 
$$f = \sum_{S\in \hat{\mathcal{H}}(\mathcal{P})} d_S\phi_S$$
with $d_S= \int f \phi_S \ dm$, where this series converges  unconditionally  on $L^\beta$. Let
\begin{equation}\label{k0} f_{k_0} = d_{I} \phi_{I}+ \sum_{k < k_0} \sum_{Q\in \mathcal{P}^k}\sum_{S\in \mathcal{H}(Q)} d_S\phi_S.\end{equation}
Then
\begin{align} f_{k_0}(x_0)  &= d_{I} \phi_{I} +\sum_{i=1}^{N(x_0,k_0)} d_{S^i} \phi_{S^i}(x_0).\nonumber \\
&=\frac{d_{I}}{|I|^{1/2}}+  \sum_{i=1}^{N(x_0,k_0)} (-1)^{a_i+1}d_{S^i}  \frac{1}{m_{S^i}}  \frac{1}{\sum_{P \in S_{a_i}^i} |P| }  \nonumber \\
&=\frac{d_{I}}{|I|^{1/2}}+\sum_{i=1}^{N(x_0,k_0)} (-1)^{a_i+1} d_{S^i}  \Big(    \frac{(\sum_{R \in S_2^i} |R| ) (\sum_{P \in S_1^i} |P|)}{\sum_{R \in S_2^i} |R| + \sum_{P \in S_1^i} |P|  }   \Big)^{1/2} \frac{1}{\sum_{P \in S_{a_i}^i} |P| }  \nonumber  \\
&=\frac{d_{I}}{|I|^{1/2}} +   \sum_{i=1}^{N(x_0,k_0)} (-1)^{a_i+1} d_{S^i}   \Big( \frac{\sum_{R \in S^{i}_{3-a_i}}|R|}{\sum_{P \in S^{i}_{a_i}} |P|}  \Big)^{1/2}  \frac{1}{(\sum_{Q \in S^{i-1}_{a_{i-1}}} |Q|)^{1/2}}   \nonumber \\
&= \int f \cdot \frac{1_{Q_{k_0}}}{|Q_{k_0}|} \ dm. \nonumber 
\end{align}
Let 
\begin{equation}\label{reep} f =  \sum_k  \sum_{P\in \mathcal{P}^{k}} k_P a_P,\end{equation}
be the series given by (\ref{sumf}). Note that
\begin{equation} f_{k_0} = \sum_{k\leq k_0}  \sum_{P\in \mathcal{P}^{k}} k_P a_P,\end{equation}

Consequently 

\begin{proposition}[Dirac's Approximations]\label{boup} Let $f \in L^\beta$, with $\beta > 1$. Let
$$ f = \sum_k  \sum_{P\in \mathcal{P}^{k}} k_P a_P.$$
\begin{itemize}
\item[A.] If  this representation is either as in (\ref{sumf}) or $k_P\geq 0$ for every $P$ then we have For every $Q \in \mathcal{P}$ 
$$\Big| \sum_{J\in \mathcal{P}, Q \subset J} k_J  |J|^{s-1/p} \Big| \leq |f1_Q|_\infty.$$
\item[B.]  In the case of the representation  (\ref{sumf}) we also have 
$$\sum_{J\in \mathcal{P}, Q \subset J} k_J  |J|^{s-1/p} =\int f \cdot \frac{1_{Q}}{|Q|} \ dm.$$
\end{itemize}
\end{proposition} 
\begin{proof} We have that $A.$ is obvious if $k_P\geq 0$ for every $P$. In the other case note that 
for  every $x_0 \in Q\in \mathcal{P}^{k_0}$ we have
 $$f_{k_0}(x_0)=\sum_{J\in \mathcal{P}, Q\subset J} k_J  |J|^{s-1/p}=  \int f \cdot \frac{1_{Q}}{|Q|} \ dm.$$
so A. and  B. follows.
\end{proof}

\newpage

\centerline{ \bf III. APPLICATIONS.}
\addcontentsline{toc}{chapter}{\bf  III. APPLICATIONS.}
\vspace{1cm}

\fcolorbox{black}{white}{
\begin{minipage}{\textwidth}
\noindent   In Part III. we  suppose that $0< s<1/p$, $p\in [1,\infty)$ and $q\in [1,\infty]$. \end{minipage} 
}
\ \\ \\

\section{Pointwise multipliers acting on $\mathcal{B}^s_{p,q}$}

Here we will apply the previous sections to study pointwise multipliers of $\mathcal{B}^s_{p,q}$. To be more precise, Let $g\colon I \rightarrow \mathbb{C}$ be a measurable function. We say that $g$ is  a pointwise multiplier acting on $\mathcal{B}^s_{p,q}$ if  the transformation
$$G(f)=gf$$
defines  a bounded operator in $\mathcal{B}^s_{p,q}$. We denote the set of pointwise multipliers by $M(\mathcal{B}^s_{p,q})$. We can consider the norm on $M(\mathcal{B}^s_{p,q})$ given by
$$|g|_{M(\mathcal{B}^s_{p,q})}=\sup\{|gf|_{\mathcal{B}^s_{p,q}}\ s.t. \ |f|_{\mathcal{B}^s_{p,q}}\leq 1    \}.$$
Of course a necessary condition for  a function to be a multiplier is that 

$$\mathcal{B}^s_{p,q,selfs}=\{ g\in \mathcal{B}^s_{p,q}\colon \sup_{a_Q \in \mathcal{A}^{sz}_{s,p}} |ga_Q|_{\mathcal{B}^s_{p,q}} < \infty\}$$ 
Denote
$$|g|_{\mathcal{B}^s_{p,q,selfs}}=\sup_{a_Q \in \mathcal{A}^{sz}_{s,p}} |ga_Q|_{\mathcal{B}^s_{p,q}}.$$      
The linear space $\mathcal{B}^s_{p,q,selfs}$ endowed with $|\cdot|_{\mathcal{B}^s_{p,q,selfs}}$is a normed space introduced by  Triebel \cite{ns}. We have 
$$|g|_{\mathcal{B}^s_{p,q,selfs}}\leq |g|_{M(\mathcal{B}^s_{p,q})}$$ 
In the following three propositions we see that  many results of  Triebel \cite{ns}   and  Schneider and Vyb\'\i{ral}  \cite{corjan} for  Besov spaces in $\mathbb{R}^n$ can be easily moved to our setting. The simplest  case  occurs when $p=q=1$.

\begin{proposition} We have that $M(\mathcal{B}^s_{1,1})=\mathcal{B}^s_{1,1,selfs}$.
\end{proposition} 
\begin{proof} Let $g \in \mathcal{B}^s_{1,1,selfs}.$ Given $f\in  \mathcal{B}^s_{1,1}$ and  $\epsilon > 0$  one can find a $\mathcal{B}^s_{1,1}$-representation 
$$f=\sum_{k}\sum_{Q\in \mathcal{P}^k} c_Qa_Q,$$
where $a_Q$ is a $(s,1)$-Souza's atom and 
$$\sum_k \sum_{Q\in \mathcal{P}^k} |c_Q| < (1+\epsilon) |f|_{\mathcal{B}^s_{1,1}}$$
so 
$$\sum_k \sum_{Q\in \mathcal{P}^k} |c_Q ga_Q|_{\mathcal{B}^s_{1,1}} < (1+ \epsilon) |g|_{\mathcal{B}^s_{1,1,selfs}}|f|_{\mathcal{B}^s_{1,1}}.$$ 
and consequently 
$$|gf|_{\mathcal{B}^s_{1,1}}=|\sum_{k}\sum_{Q\in \mathcal{P}^k} c_Qga_Q|_{\mathcal{B}^s_{1,1}}\leq (1+ \epsilon)  |g|_{\mathcal{B}^s_{1,1,selfs}}|f|_{\mathcal{B}^s_{1,1}}.$$
Since $\epsilon$ is arbitrary we get
$$|gf|_{\mathcal{B}^s_{1,1}}\leq  |g|_{\mathcal{B}^s_{1,1,selfs}} |f|_{\mathcal{B}^s_{1,1}}.$$
\end{proof}

\begin{lemma}\label{incint3} Let $W\in \mathcal{P}$. The restriction application
$$r\colon \mathcal{B}^s_{p,q}(I, \mathcal{P}, \mathcal{A}^{sz}_{s,p})\rightarrow  \mathcal{B}^s_{p,q}(W, \mathcal{P}_W, \mathcal{A}^{sz}_{s,p})$$
given by $r(f)=1_Wf$ is continuous. Indeed there is $\Cll{incint} \geq 1$, that does not depend on $W$, such that 
\begin{enumerate}
\item[A.] For every $f \in \mathcal{B}^s_{p,q}$ we have 
\begin{equation}\label{incint2} |1_Wf|_{\mathcal{B}^s_{p,q}(W, \mathcal{P}_W, \mathcal{A}^{sz}_{s,p})}\leq \Crr{incint} |f|_{\mathcal{B}^s_{p,q}}. \end{equation}
In particular $1_W\in M(\mathcal{B}^s_{p,q})$.
\item[B.] For every  $\mathcal{B}^{s+}_{p,q}$-representation 
$$f=\sum_{k}\sum_{Q\in \mathcal{P}^k} c_Qa_Q$$
 one can find a $\mathcal{B}^{s+}_{p,q}(W, \mathcal{P}_W, \mathcal{A}^{sz}_{s,p})$-representation 
$$1_Wf= \sum_{k}\sum_{Q\in \mathcal{P}^k} d_Qa_Q$$
such that 
\begin{equation}\label{incint2b}
\Big( \sum_k \big(\sum_{Q\in \mathcal{P}^k}  (d_Q)^p\big)^{q/p} \Big)^{1/q} \leq \Crr{incint} \Big( \sum_k \big(\sum_{Q\in \mathcal{P}^k}  (c_Q)^p\big)^{q/p} \Big)^{1/q} . 
\end{equation} 
Moreover  $d_Q\neq 0$ implies $Q\subset \ supp \ f$.
\end{enumerate}
\end{lemma} 
\begin{proof} Let $Q\in \mathcal{P}.$ Denote by $a_Q$ the canonical $(s,p)$-Souza's atom supported on $Q$.  If $W\cap Q=\emptyset$ we can write $1_Wa_Q= 0a_Q.$ If $Q\subset W$ then $1_Wa_Q= 1a_Q.$ If $W\subset Q$ then
$$1_Wa_Q= \Big(\frac{|W|}{|Q|}\Big)^{1/p-s}a_W,$$
where
$$ \Big(\frac{|W|}{|Q|}\Big)^{1/p-s}\leq \Crr{maior}^{(k_0(W)-k_0(Q))(1/p-s)}$$
In every case we can write
$$h_Q=1_Wa_Q = \sum_k \sum_{\substack{P\in \mathcal{P}^k\\ P\subset Q}}    s_{P,Q} a_P,$$
with 
$$ \sum_{\substack{P\in \mathcal{P}^k\\ P\subset Q}}    |s_{P,Q}| ^p\leq \Crr{maior}^{(k-k_0(Q))(1-sp)}$$
By Proposition \ref{trans}.A and \ref{trans}.B  there is $\Crr{incint}$ such that  A. and  B.  hold. 
\end{proof}

\begin{proposition} We have that $ \mathcal{B}^s_{p,q,selfs} \subset L^\infty$ and this inclusion is continuous.
\end{proposition}  
\begin{proof} Let $g\in \mathcal{B}^s_{p,q,selfs}$. Then    $ga_Q\in \mathcal{B}^s_{p,q}$ and by Lemma \ref{incint3} we have
$$|g1_Q|_{ \mathcal{B}^s_{p,q}(Q, \mathcal{P}_Q, \mathcal{A}^{sz}_{s,p})}\leq \Crr{incint}|g1_Q|_{\mathcal{B}^s_{p,q}} \leq \Crr{incint} |g|_{\mathcal{B}^s_{p,q,selfs}}.$$

By Proposition \ref{lp}  (taking t = $p$) we have  $(s,p)$-Souza's atom $a_Q$ we have 
$$|ga_Q|_p\leq   \Crr{mult1}^{1+1/p} \Crr{co}(p,q, ( |\mathcal{P}^k_Q|^{sp})_k ) \Crr{incint}  |g|_{\mathcal{B}^s_{p,q,selfs}} < \Cll{nova} |Q|^s  |g|_{\mathcal{B}^s_{p,q,selfs}}.$$
for some constant $\Crr{nova}$. In other words
$$\Big( |Q|^{sp-1} \int_Q |g|^p \ dm \Big)^{1/p}  \leq \Crr{nova} |Q|^s  |g|_{\mathcal{B}^s_{p,q,selfs}},$$
so
$$\frac{1}{|Q|}  \int_Q |g|^p \ dm \leq   \Crr{nova}^p   |g|_{\mathcal{B}^s_{p,q,selfs}}^p, $$
for every $Q\in \mathcal{P}$.  Due to the fact that $\cup_k\mathcal{P}^k$ generates the $\sigma$-algebra $\mathbb{A}$, by  L\'evy's Upward Theorem (see Williams \cite{martin})  for almost every $x\in I$ the following holds.  If $x\in Q_k\in \mathcal{P}^k$ then 
$$\lim_k \frac{1}{|Q_k|}  \int_{Q_k} |g|^p \ dm = |g(x)|^p.$$
So 
$$|g|_\infty\leq  \Crr{nova}   |g|_{\mathcal{B}^s_{p,q,selfs}}.$$
\end{proof} 

\begin{figure}
\includegraphics[scale=0.5]{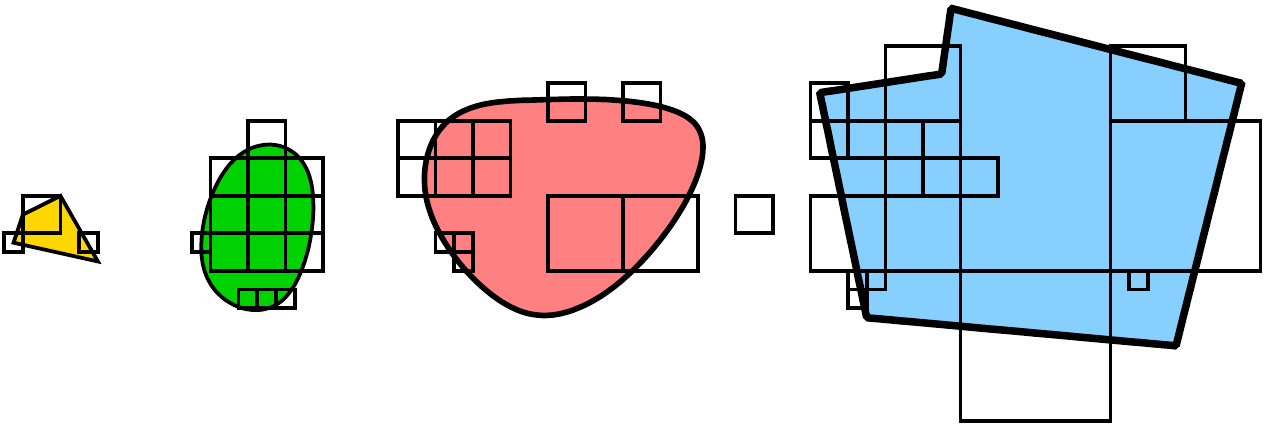}
\caption{ Non-Archimedean  behaviour. Illustration for Proposition \ref{sepa}. The filled regions are the supports of the functions $g_i$. The squares are the elements of the grid on which there are atoms contributing to  the representation of $f$. Every square intercepts at most one support, so each atom ``sees'' only the function whose support  is nearby.}
\end{figure}

\changei{In the caption of Figure 2, we replaced ". the" by ". The" (missing capitalization) and we replaced "contributing for" by "contribution to"}

\subsection{Non-Archimedean  behaviour in $\mathcal{B}^\beta_{p,\tilde{q},selfs}$} If we have a sequence $g_i \in M(\mathcal{B}^s_{p,q})$ we can get the naive  estimate 
$$|(\sum_i g_i)f|_{\mathcal{B}^s_{p,q}}\leq  \big( \sum_i |g_i|_{M(\mathcal{B}^s_{p,q})}\big)|f|_{\mathcal{B}^s_{p,q}}.$$

Nevertheless\change{we replaced "But it is remarkable that" by "Nevertheless, remarkably,"}, remarkably, sometimes one can get a far better estimate.  To state the result we need to define

$$|g|_{\mathcal{B}^{s,t}_{p,q,selfs}}=\sup_{\substack{ a_Q \in \mathcal{A}^{sz}_{s,p}\\ Q\in \mathcal{P}^j \\ j\geq t}} |ga_Q|_{\mathcal{B}^s_{p,q}}.$$
It is easy to see that this norm is equivalent to $|\cdot|_{\mathcal{B}^{s}_{p,q,selfs}}$.

\begin{proposition} \label{sepa} Let $\beta > s$. There is $\Cll{gen}$ with the following property. Let  $g_i \in  \mathcal{B}^{\beta,t_i}_{p,\tilde{q},selfs}$, with $i\in \Lambda \subset \mathbb{N}$,  and $t_i \in \mathbb{N}$.

Consider a function $f$ with  a $\mathcal{B}^s_{p,q}$-representation 
$$ f=\sum_k \sum_{Q\in \mathcal{P}^k} c_Q a_Q$$
satisfying 
\begin{itemize}
\item[A.] We have 
$$\sup_{\substack{ Q\in \mathcal{P} \\ c_Q\neq 0}}  \  \sum_{_{Q\cap supp \ g_i \neq \emptyset}} |g_i|_{ \mathcal{B}^{\beta,t_i}_{p,\tilde{q},selfs}} \leq N.$$
\item[B.]  If $Q\in \mathcal{P}^k$ satisfies $c_Q\neq 0$ and $Q\cap supp \ g_i \neq \emptyset$ then $k\geq t_i$. 
\end{itemize}
Then we can find a $\mathcal{B}^s_{p,q}$-representation
\begin{equation} \label{rre} (\sum_i g_i)f= \sum_k \sum_{P\in \mathcal{P}^k} d_Q a_Q\end{equation} 
such that 
\begin{align}\label{reep33} &\Big(\sum_k \big( \sum_{P\in \mathcal{P}^k} |d_Q|^p \big)^{q/p} \Big)^{1/q} \nonumber \\
&\leq \Crr{gen} N \Big(\sum_k \big( \sum_{P\in \mathcal{P}^k} |c_Q|^p \big)^{q/p} \Big)^{1/q} 
\end{align} 
\end{proposition} 
\begin{proof} It is enough to prove the result for the case when  $\Lambda$ is finite. Let $Q\in \mathcal{P}^{k_0}$ with $c_Q\neq 0$.  There is   $\{i_1,
\dots, i_j\}\subset \Lambda$,  such that   
\begin{equation}\label{aa1} (\sum_i g_i ) a_Q=  \sum_{\ell\leq j} g_{i_\ell} a_Q.\end{equation} 
and $Q\cap supp \ g_{i_\ell} \neq \emptyset$ for every $\ell$. In particular $k_0\geq \max_\ell t_{i_\ell}.$
By Lemma \ref{incint3} for each $\ell\leq j$ we can find a $\mathcal{B}^\beta_{p,q} $-representation 
\begin{equation}\label{hjk} g_{i_\ell}a_Q= \sum_k \sum_{\substack{ P\in \mathcal{P}^k \\ P\subset Q}} \tilde{s}_{P,Q}^\ell b_{P}\end{equation}
such that $b_P$ is the canonical $(\beta,p)$-Souza's atom supported on $P$ and 
$$\Big( \sum_k \big(  \sum_{\substack{ P\in \mathcal{P}^k \\ P\subset Q}}  |\tilde{s}_{P,Q}^\ell|^p   \big)^{\tilde{q}/p}  \Big)^{1/\tilde{q}}\leq 2\Crr{incint}  |g_{i_\ell}|_{\mathcal{B}^{\beta,t_{i_\ell}}_{p,\tilde{q},selfs}} .$$
Since $b_P= |P|^{\beta -s}  a_P$, where $a_P$ is the canonical $(s,p)$-Souza's atoms supported on $P$, we can write
$$g_{i_\ell}a_Q= \sum_k \sum_{\substack{ P\in \mathcal{P}^k \\ P\subset Q}} s_{P,Q}^\ell a_{P},$$
with $s_{P,Q}^\ell=\tilde{s}_{P,Q}^\ell|P|^{\beta -s}$ satisfying 
\begin{align*}
 &\big( \sum_{\substack{ P\in \mathcal{P}^k \\ P\subset Q}}  |s_{P,Q}^\ell|^p \big)^{1/p} \\
&\leq  2\Crr{incint}  |g_{i_\ell}|_{\mathcal{B}^{\beta,t_{i_\ell}}_{p,\tilde{q},selfs}}  \Crr{maior}^{(\beta-s)(k-k_0)} \sup_{Q\in \mathcal{P}}|Q|^{\beta-s},
\end{align*} 
so we can write
$$(\sum_i g_i ) a_Q=  \sum_k \sum_{\substack{ P\in \mathcal{P}^k \\ P\subset Q}} s_{P,Q} a_{P},$$
with 
$$s_{P,Q}=\sum_\ell s_{P,Q}^\ell$$
 satisfying 
$$\big( \sum_{\substack{ P\in \mathcal{P}^k \\ P\subset Q}}  |s_{P,Q}|^p \big)^{1/p} \leq  2 N \Crr{incint} \Crr{maior}^{(\beta-s)(k-k_0)} \sup_{Q\in \mathcal{P}}|Q|^{\beta-s}.$$
By Proposition \ref{trans}.A we can find a $\mathcal{B}^s_{p,q}$-representation (\ref{rre}) satisfying (\ref{reep33}). 
\end{proof}

\begin{remark}\label{posrem} If $g$ is    $\mathcal{B}^\beta_{p,q}$-positive we can define
$$|g|_{\mathcal{B}^{s+,t}_{p,q,selfs}}=\sup_{\substack{ a_Q \in \mathcal{A}^{sz}_{s,p}\\ Q\in \mathcal{P}^j \\ j\geq t}} |ga_Q|_{\mathcal{B}^{s+}_{p,q}}.$$
 If we assume additionally that $g_i$ are $\mathcal{B}^{\beta,t_i}_{p,q}$-positive, Proposition \ref{sepa} remains true if we replace all the instances  of $|\cdot|_{\mathcal{B}^{s,t}_{p,q,selfs}}$ by $|\cdot|_{\mathcal{B}^{s+,t}_{p,q,selfs}}$  in its statement. Moreover  by Proposition \ref{trans}.B and Lemma \ref{incint3}. B we can conclude that
 \begin{itemize}
 \item[i.]  if $c_Q\geq 0$ for every $Q$ then $d_Q\geq 0$ for every $Q$,
 \item[ii.] If $Q$ is such that $d_Q\neq 0$ then $Q\subset supp \ g_i$, for some $i\in \Lambda$. 
\end{itemize}
\end{remark} 

\begin{corollary}\label{23er} For every $\beta > s$ and $\tilde{q}\in [1,\infty]$  we have  $\mathcal{B}^\beta_{p,\tilde{q},selfs} \subset M(\mathcal{B}^s_{p,q})$. Moreover this inclusion is continuous. 
\end{corollary}

\subsection{Strongly regular domains }\label{srd} We may wonder on which conditions the characteristic function of a set  $\Omega$ is a pointwise multiplier in $\mathcal{B}^{s}_{p,q}$.  
\begin{definition} A measurable set $\Omega\subset I$ is  a {\bf $(\alpha,\Cll{rp2},\Cll{k1} )$-strongly regular domain}  if  for every $Q\in \mathcal{P}^j$, with $j\geq \Crr{k1}$,  there is  family $\mathcal{F}^k(Q\cap \Omega) \subset \mathcal{P}^k$  such that 
\begin{itemize}
\item[i.] We have $Q\cap \Omega = \cup_{k} \cup_{P\in \mathcal{F}^k(Q\cap \Omega)} P$.
\item[ii.] If $P,W \in \cup_{k} \mathcal{F}^k(Q\cap \Omega)$ and $P\neq W$ then $P\cap W=\emptyset$. 
\item[iii.] We have
\begin{equation} \sum_{P\in \mathcal{F}^k(Q\cap \Omega)} |P|^{\alpha}\leq \Crr{rp2}  |Q|^{\alpha}.\end{equation} 
\end{itemize}
\end{definition}

The following result can be associated  with results in Triebel \cite{ns} for $B^s_{p,p}(\mathbb{R}^n)$, especially \change{we replaced "specially" by "especially"} when we consider the setting of Besov spaces in compact homogenous spaces. See Section \ref{srd} for details.See also Schneider and Vyb\'\i ral \cite{corjan}.

\begin{proposition} \label{pos2} If $\Omega$ is a $(1-\beta p,\Crr{rp2},\Crr{k1} )$-strongly regular domain then 
\begin{equation}\label{fimm} |1_\Omega|_{\mathcal{B}^{\beta+,{\Crr{k1}}}_{p,\infty,selfs}}\leq   \Crr{rp2}^{1/p}.\end{equation}
\end{proposition} 
\begin{proof} Given $Q\in \mathcal{P}^j$, with $j\geq \Crr{k1}$  we can write
\begin{equation} \label{poi} 1_\Omega a_Q =    \sum_{k}   \sum_{P\in \mathcal{F}^k(Q\cap \Omega)} \Big(\frac{|P|}{|Q|} \Big)^{1/p-\beta}   a_P.\end{equation} 
where $a_P$ is a $(\beta,p)$-atom.  Note that 
$$  \big(   \sum_{P\in \mathcal{F}^k(Q\cap \Omega)} \Big(\frac{|P|}{|Q|} \Big)^{1-\beta p} \big)^{1/p} \leq \Crr{rp2}^{1/p},$$
so (\ref{fimm}) holds.
\end{proof}

\begin{proposition}[Pointwise Multipliers I] \label{pm1} There is $\Cll{gen2}$ with the following property. Suppose that  $\Omega_i$ are  $(1-\beta p,K_i,t_i)$-strongly regular domains, $i\in \Lambda \subset \mathbb{N}$, and $\Theta_i>  0$ for every $i\in \Lambda_i$.
Consider a function $f$ with  a $\mathcal{B}^s_{p,q}$-representation 
$$ f=\sum_k \sum_{Q\in \mathcal{P}^k} c_Q a_Q$$
satisfying  
\begin{itemize}
\item[A.] We have 
$$\sup_{\substack{ Q\in \mathcal{P} \\ c_Q\neq 0}}  \  \sum_{_{Q\cap \Omega_i \neq \emptyset}} \Theta_i K_i^{1/p} \leq N.$$
\item[B.]  If $Q\in \mathcal{P}^k$ satisfies $c_Q\neq 0$ and $Q\cap \Omega_i \neq \emptyset$ then $k\geq t_i$. 
\end{itemize}
  Then we can find a $\mathcal{B}^s_{p,q}$-representation
\begin{equation} \label{rre2} (\sum_i \Theta_i 1_{\Omega_i})f= \sum_k \sum_{P\in \mathcal{P}^k} d_Q a_Q\end{equation} 
such that 
\begin{align}\label{reep2} &\Big(\sum_k \big( \sum_{P\in \mathcal{P}^k} |d_Q|^p \big)^{q/p} \Big)^{1/q} \nonumber \\
&\leq \Crr{gen2} N\Big(\sum_k \big( \sum_{P\in \mathcal{P}^k} |c_Q|^p \big)^{q/p} \Big)^{1/q}.
\end{align} 
Moreover 
\begin{itemize}
\item[i.] If $Q$ satisfies  $d_Q\neq 0$ then $Q\subset \Omega_i$ for some $i \in \Lambda$. 
\item[ii.] If $c_Q\geq 0$ for every $Q$ then $d_Q\geq 0$ for every $Q$. 
\end{itemize}
\end{proposition} 
\begin{proof} It follows from Proposition \ref{sepa}, Proposition \ref{pos2} and Remark \ref{posrem}. 
\end{proof}


\subsection{Functions on $\mathcal{B}^{1/p}_{p,\infty}\cap L^\infty$  } 

We want\change{"we replaced "We would like to give" by  "We want to give"} to give explicit examples of multipliers in $\mathcal{B}^s_{p,q}$. One should compare the following result with the study by Triebel\cite{multi}  of the regularity of  the multiplication on Besov spaces. See also Maz'ya  and  Shaposhnikova \cite{sob} for more information   on multipliers in classical Besov spaces.

\begin{proposition}[Pointwise multipliers II] \label{mult} Let $g\in  \mathcal{B}^{1/p}_{p,\infty}\cap L^\infty$. Then the multiplier operator
$$G\colon \mathcal{B}^s_{p,q} \rightarrow \mathcal{B}^s_{p,q}$$
defined by $G(f)=gf$ is a well-defined \change{we replaced "well defined" by "well-defined"} and bounded operator acting on $(\mathcal{B}^s_{p,q},|~\cdot~|_{\mathcal{B}^s_{p,q}})$.
Indeed
$$|G|_{\mathcal{B}^s_{p,q}}\leq \Crr{e}\Crr{no}  \frac{|g|_{\mathcal{B}^{1/p}_{p,\infty}} }{1-\Crr{maior}^{1/p-s}} +|g|_\infty,$$
where $\Crr{e}=\Crr{e}(1/p,p,\infty)$ and $\Crr{no}=\Crr{no}(1/p,p,\infty)$ are as in Corollary \ref{fou}.
\end{proposition} 
\begin{remark} \label{pos3} We can get a similar result replacing $\mathcal{B}^{a}_{p, b}$ by $\mathcal{B}^{a+}_{p, b}$ everywhere. 
\end{remark} 
\begin{proof} Let  $a_Q=|Q|^{s-1/p}1_Q$ be  the canonical $(s,p)$-Souza's atom on $Q$ and $b_J=1_J$ be  the canonical $(1/p,p)$-Souza's atom on $J$. Given $\epsilon > 0$, let 
$$f = \sum_k \sum_{Q\in \mathcal{P}^k} c_Q a_Q$$
be  a $\mathcal{B}^s_{p,q}$-representation of $f$ such that 
$$ \Big(  \sum_k \big(\sum_{Q\in \mathcal{P}^k} |c_Q|^p\big)^{q/p} \Big)^{1/q} \leq (1+\epsilon)|f|_{\mathcal{B}^s_{p,q}}$$
and
$$g =\sum_k  \sum_{J\in \mathcal{P}^k} e_J b_J$$
be  a $\mathcal{B}^{1/p}_{p,\infty}$-representation of $g$ given by Corollary \ref{fou} (in the case of Remark \ref{pos3} we can consider an optimal $\mathcal{B}^{1/p}_{p,\infty}$-positive representation of $g$). We claim that 
$$u_1 =\sum_j  \sum_{J \in \mathcal{P}^j} \big(\sum_{J\subset Q, Q\neq J,Q\in \mathcal{P}} \big( \frac{|J|}{|Q|}\big)^{1/p-s}  c_Q e_J  \big)  a_J, \ and $$
$$u_2 = \sum_k \sum_{Q \in \mathcal{P}^k} \big( \sum_{Q\subset J,  J\in \mathcal{P} } c_Q e_J \big)  a_Q,$$
are $\mathcal{B}^s_{p,q}$-representations of functions $u_i \in \mathcal{B}^s_{p,q}$. Firstly note that the inner sums are finite. Moreover  if $J\in \mathcal{P}^j$ we denote by 
$Q_k(J)$ the unique element  of $\mathcal{P}^k$, with $k\leq j$ that satisfies  $J \subset Q_k(J)$ then 
\begin{align}& \ \big( \sum_{J \in \mathcal{P}^j } \big|\sum_{k \leq j}  \Big( \frac{|J|}{|Q_k(J)|}\Big)^{1/p-s}  c_{Q_k(J)} e_J\big|^p \big)^{1/p} \nonumber  \\
&\leq  \sum_{k \leq j} \Crr{maior}^{(j-k)(1/p-s)}  \big( \sum_{J \in \mathcal{P}^j }| c_{Q_k(J)} e_J|^p \big)^{1/p} \nonumber  \\
&\leq \big( \sum_{J \in \mathcal{P}^j }|e_J|^p \big)^{1/p}  \sum_{k \leq j} \Crr{maior}^{(j-k)(1/p-s)}  \max_{Q\in \mathcal{P}^k}  |c_Q|\nonumber  \\
&\leq \Big( \max_j \big( \sum_{J \in \mathcal{P}^j }|e_J|^p \big)^{1/p} \Big) \sum_{k \leq j} \Crr{maior}^{(j-k)(1/p-s)}  (\sum_{Q\in \mathcal{P}^k}  |c_Q|^p)^{1/p} \nonumber  \\
&\leq \Crr{e}\Crr{no}  |g|_{\mathcal{B}^{1/p}_{p,\infty}} \sum_{k \leq j} \Crr{maior}^{(j-k)(1/p-s)}  (\sum_{Q\in \mathcal{P}^k}  |c_Q|^p)^{1/p} \nonumber  
\end{align} 
The right hand side is a convolution, so we can easily get 
$$ |u_1|_{\mathcal{B}^s_{p,q}} \leq (1+\epsilon) \Crr{e}\Crr{no} \frac{|g|_{\mathcal{B}^{1/p}_{p,\infty}} }{1-\Crr{maior}^{1/p-s}}     |f|_{\mathcal{B}^s_{p,q}}.$$
Moreover by Proposition \ref{boup}.B, with $s=1/p$, we obtain
\begin{align}& \Big( \sum_{Q \in \mathcal{P}^k} \big| \sum_{J\in \mathcal{P}, Q\subset J} c_Q e_J \big|^p  \Big)^{1/p} \nonumber  \\
&\leq  \Big( \sum_{Q \in \mathcal{P}^k} |c_Q|^p \big| \sum_{J\in \mathcal{P}, Q\subset J} e_J \big|^p  \Big)^{1/p} \nonumber  \\
&\leq  \Big(  \sum_{Q \in \mathcal{P}^k} |c_Q|^p \Big)^{1/p} |g|_{\infty}. \nonumber  
\end{align} 
So
$$ |u_1|_{\mathcal{B}^s_{p,q}} \leq  (1+\epsilon) |f|_{\mathcal{B}^s_{p,q}} |g|_\infty .$$
We claim  that $gf= u_1+u_2$. Indeed  let
$$f_{k_0} = \sum_{k<k_0} \sum_{Q\in \mathcal{P}^k} c_Q a_Q$$
and
$$g_{k_0} = \sum_{k<k_0} \sum_{J\in \mathcal{P}^k} e_J b_J$$
By Proposition \ref{lp} we have 
$$\lim_{k_0} |g_{k_0}-g|_{p'}=0$$
and
 $$\lim_{k_0} |f_{k_0}-f|_{p}=0.$$
So
$$\lim_{k_0}|f_{k_0}g_{k_0}-fg|_1=0.$$
Note that 
\begin{itemize}
\item[A.] If $Q\subset J$ then  $a_Q b_J= a_Q$,
\item[B.] If $J \subset Q$ then 
$$a_Q b_J =  \big( \frac{|J|}{|Q|}\big)^{1/p-s} a_J.$$
\end{itemize}
So
\begin{eqnarray*}
f_{k_0}g_{k_0}&=& \sum_{k<k_0} \sum_{Q\in \mathcal{P}^k}  \sum_{i<k_0} \sum_{J\in \mathcal{P}^i} e_J c_Q a_Qb_J\\
&=&  \sum_{k<k_0} \sum_{Q\in \mathcal{P}^k}  \big( \sum_{\substack{J\in \mathcal{P}\\ Q\subset J}} e_Jc_Q\big)a_Q +  \sum_{k<k_0} \sum_{Q\in \mathcal{P}^k}  \sum_{i<k_0} \sum_{\substack{J\in \mathcal{P}^i\\J\subset Q\\J\neq Q}}  e_Jc_Q \big( \frac{|J|}{|Q|}\big)^{1/p-s} a_J \\
&=&  \sum_{k<k_0} \sum_{Q\in \mathcal{P}^k}  \big( \sum_{\substack{J\in \mathcal{P}\\ Q\subset J}} e_Jc_Q\big)a_Q +   \sum_{i<k_0} \sum_{J \in \mathcal{P}^i}  \Big(  \sum_{\substack{Q\in \mathcal{P}\\J\subset Q\\J\neq Q}}  e_Jc_Q \big( \frac{|J|}{|Q|}\big)^{1/p-s}\Big) a_J\\
&=&u_{1,k_0}+u_{2,k_0}.
\end{eqnarray*}
Note that  
$$\lim_{k_0} |u_{r,k_0}-u_r|_1=0 \ and \   |u_{r,k_0}|_{B^s_{p,q}}\leq |u_{r}|_{B^s_{p,q}} \ for \ r=1,2.$$
Now we can use Corollary \ref{compa1}.i to conclude the proof.
\end{proof}

\section{ $\mathcal{B}^s_{p,q}\cap L^\infty$ is a quasi-algebra}

Multipliers in $\mathcal{B}^s_{p,q}\cap L^\infty$ are indeed much easier to come by.

\begin{proposition}[Pointwise multipliers III] \label{mult33} Let $g,f  \in \mathcal{B}^{s}_{p,q}\cap L^\infty$. Then $g\cdot f  \in \mathcal{B}^{s}_{p,q}\cap L^\infty$ and 
$$|f\cdot g|_{\mathcal{B}^s_{p,q}}+|f\cdot g|_{\infty}\leq\Crr{e}\Crr{no} (|f|_{\mathcal{B}^{s}_{p,q}}  +|f|_\infty)( |g|_{\mathcal{B}^{s}_{p,q}} +|g|_\infty).$$
So $\mathcal{B}^{s}_{p,q}\cap L^\infty$ is a quasi-Banach \change{we replaced "quasi Banach" by "quasi-Banach"}  algebra. Here  $\Crr{e}=\Crr{e}(s,p,q)$ and $\Crr{no}=\Crr{no}(s,p,q)$ are as in Corollary \ref{fou}.
\end{proposition} 
\begin{proof} Of course $|f\cdot g|_\infty\leq |f|_\infty|g|_\infty$. Let  $a_Q=|Q|^{s-1/p}1_Q$ be  the canonical $(s,p)$-Souza's atom on $Q$. Let 
$$f = \sum_k \sum_{Q\in \mathcal{P}^k} c_Q a_Q$$
and
$$g =\sum_k  \sum_{J\in \mathcal{P}^k} e_J a_J$$
be  $\mathcal{B}^s_{p,q}$-representations of $f$ and $g$  given by Corollary \ref{fou}. We claim that 
$$u_1 =\sum_k  \sum_{Q \in \mathcal{P}^k} \big( \sum_{Q\subset J,  J\in \mathcal{P} } |J|^{s-1/p} c_Q e_J \big)  a_Q,$$
$$u_2 =\sum_k  \sum_{J \in \mathcal{P}^k} \big(\sum_{J\subset Q, Q\neq J,Q\in \mathcal{P}} |Q|^{s-1/p}  c_Q e_J  \big)  a_J.$$
are $\mathcal{B}^s_{p,q}$-representations of functions $u_i \in \mathcal{B}^s_{p,q}$. Moreover by Proposition \ref{boup}.A  we have 
\begin{align}& \Big( \sum_{Q \in \mathcal{P}^k} \big| \sum_{J\in \mathcal{P}, Q\subset J}  |J|^{s-1/p}  c_Q e_J \big|^p  \Big)^{1/p} \nonumber  \\
&\leq  \Big( \sum_{Q \in \mathcal{P}^k} |c_Q|^p \big| \sum_{J\in \mathcal{P}, Q\subset J}  |J|^{s-1/p}  e_J \big|^p  \Big)^{1/p} \nonumber  \\
&\leq  \Big(  \sum_{Q \in \mathcal{P}^k} |c_Q|^p \Big)^{1/p} |g|_{\infty}. \nonumber  
\end{align} 
So 
$$ |u_1|_{\mathcal{B}^s_{p,q}} \leq \Crr{e}\Crr{no}  |g|_\infty |f|_{\mathcal{B}^s_{p,q}},$$
and by an analogous argument
$$ |u_2|_{\mathcal{B}^s_{p,q}} \leq \Crr{e}\Crr{no} |f|_\infty |g|_{\mathcal{B}^s_{p,q}}.$$
Define $f_{k_0}$ and $g_{k_0}$ as in the proof of Proposition \ref{mult}. By Proposition \ref{boup} we have  $|f_{k_0}|\leq |f|_\infty$ and $|g_{k_0}|\leq |g|_\infty$. Since $\lim_{k_0} f_{k_0}=f$ and $\lim_{k_0} g_{k_0}=g$ in $L^p$, we can assume, taking a subsequence if necessary, that $f_{k_0}g_{k_0}$ converges pointwise to $fg$. So by the Theorem of Dominated Convergence we have $\lim_{k_0} f_{k_0}g_{k_0}= fg$ in $L^1$. 
Finally note that if $J \subset Q$ then 
$$a_Q a_J =  |Q|^{s-1/p} a_J,$$
Now we can use the same argument as in the proof  of Proposition \ref{mult}  to conclude that $gf= u_1+u_2$.  This concludes the proof. 
\end{proof}

\subsection{Regular domains} \label{cf} Here we will give sufficient conditions for the characteristic function of a set to define a bounded pointwise multiplier either on $\mathcal{B}^{s}_{p,q}\cap L^\infty$. For every set  $\Omega$, let 
$$k_0(\Omega)=\min \{k\geq 0 \colon \ \exists P \in \mathcal{P}^k \ s.t. \ P \subset \Omega     \} $$

\begin{definition} We say that a countable family of pairwise disjoint measurable sets $\{\Omega_r\}_{r\in \Lambda}$  is  {\bf  $(\alpha, \Cll{domain},\Cll[c]{domainc})$-regular family} if  one can  find  families $\mathcal{F}^k(\Omega_r)\subset \mathcal{P}^k$, $k\geq  k_0(\Omega_r)$,  such that 
\begin{itemize}
\item[A.] We have $\Omega_r = \cup_{k\geq k_0(\Omega_r)} \cup_{Q\in \mathcal{F}^k(\Omega_r)} Q$.
\item[B.] If $P,Q \in \cup_{k\geq k_0(\Omega_r)} \mathcal{F}^k(\Omega_r)$ and $P\neq Q$ then $P\cap Q=\emptyset$. 
\item[C.] We have
\begin{equation}\label{dom} \sum_{r\in \Lambda} \sum_{Q\in \mathcal{F}^k(\Omega_r)} |Q|^{\alpha}\leq \Crr{domain} \Crr{domainc}^{k-k_0(\cup_r \Omega_r)} |\cup_r \Omega_r|^{\alpha}.\end{equation} 
\end{itemize}
We say that a measurable set $\Omega$ is a $(\alpha, \Crr{domain},\Crr{domainc})$-regular domain if $\{\Omega\}$  is a $(\alpha, \Crr{domain},\Crr{domainc})$-regular family.
\end{definition}

\begin{proposition} Let $\beta > s$. Every $(1-\beta p, C,0)$-strongly regular domain is a  $(1-sp,C',\Crr{maior}^{(\beta-s)p})$-regular domain,
for some $C'$. 
\end{proposition}

\begin{proof} Consider a  $(1-\beta p,\Crr{rp2},0)$-strongly regular domain $\Omega$. There are at most $\Crr{menor}^{-k_0(\Omega)}$ elements in $\mathcal{P}^{k_0(\Omega)}$ and
$$\Big(\frac{\Crr{maior}}{\Crr{menor}}\Big)^{-k_0(\Omega)}  \leq \frac{|Q|}{|W|}\leq \Big(\frac{\Crr{maior}}{\Crr{menor}}\Big)^{k_0(\Omega)}$$
for every $Q,W\in \mathcal{P}^{k_0(\Omega)}$. Consequently
$$\Big(\frac{\Crr{maior}}{\Crr{menor}}\Big)^{-k_0(\Omega)}  \leq \frac{|\Omega|}{|Q|}\leq \Crr{menor}^{-k_0(\Omega)} \Big(\frac{\Crr{maior}}{\Crr{menor}}\Big)^{k_0(\Omega)}$$
for each $Q \in \mathcal{P}^{k_0(\Omega)}$.
For every $Q\in \mathcal{P}^{k_0(\Omega)}$  there is a family $\mathcal{F}^k(Q\cap \Omega)$ such that 
$$\sum_k \sum_{P\in \mathcal{F}^k(Q\cap \Omega)} P= Q\cap \Omega$$
and
$$ \sum_{P\in \mathcal{F}^k(Q\cap \Omega)}|P|^{1-\beta p}\leq C  |Q|^{1-\beta p}.$$
Let 
$$\mathcal{F}^k(\Omega)=\cup_{Q\in \mathcal{P}^{k_0(\Omega)} } \mathcal{F}^k(Q\cap \Omega).$$
We have
\begin{align*}  \sum_{Q\in  \mathcal{P}^{k_0(\Omega)}}\sum_{P\in \mathcal{F}^k(Q\cap \Omega)}|P|^{1-s p}&=\sum_{P\in \mathcal{F}^k(Q\cap \Omega)}|P|^{1-\beta p} |P|^{(\beta-s)p}\\
&\leq  \sum_{Q\in  \mathcal{P}^{k_0(\Omega)}} \big(\max_{P\in \mathcal{F}^k(Q\cap \Omega)}|P|^{(\beta-s)p} \big)  \sum_{P\in \mathcal{F}^k(Q\cap \Omega)}|P|^{1-\beta p}\\
&\leq  \Crr{maior}^{(k-k_0(\Omega))(\beta-s)p}  \sum_{Q\in  \mathcal{P}^{k_0(\Omega)}} |Q|^{(\beta-s)p}  \sum_{P\in \mathcal{F}^k(Q\cap \Omega)}|P|^{1-\beta p}\\
&\leq  C \Crr{maior}^{(k-k_0(\Omega))(\beta-s)p}  \sum_{Q\in  \mathcal{P}^{k_0(\Omega)}} |Q|^{(\beta-s)p} |Q|^{1-\beta p} \\
& \leq C \Crr{maior}^{(k-k_0(\Omega))(\beta-s)p} \sum_{Q\in  \mathcal{P}^{k_0(\Omega)}} |Q|^{1-s p}\\
&\leq   C  \Crr{menor}^{-k_0(\Omega)}   \Big(\frac{\Crr{maior}}{\Crr{menor}}\Big)^{k_0(\Omega)(1-s p)}   \Crr{maior}^{(k-k_0(\Omega))(\beta-s)p} |\Omega|^{1-s p}   
\end{align*} 
This concludes the proof. \end{proof}
\begin{remark} Suppose that there is $\Cll{compara}$ such that for every $k$ and every $Q, W\in \mathcal{P}^k$ we have
$$\frac{1}{\Crr{compara}}\leq    \frac{|Q|}{|W|}\leq \Crr{compara},$$
and 
$$\#\{  P\in \mathcal{P}^{k_0(\Omega)}\colon \ P\cap \Omega\neq \emptyset \}\leq \Cll{numero},$$
Then it is easy to see that one can choose $C'= \Crr{numero}\Crr{compara}C$.

\end{remark}

The following result is similar to results for Sobolev spaces  by Faraco and Rogers \cite{faraco}. See also Sickel \cite{sickel}.

\begin{corollary} If $\{ \Omega_r\}_{r\in \Lambda}$ is a $(1-ps, \Crr{domain},\Crr{domainc})$-regular family  then there is $\Cll{rel}$ such that  for every  $g\in \mathcal{B}^{s}_{p,q}\cap L^\infty$ and $r\in\Lambda$ we can find a  $\mathcal{B}^{s}_{p,q}$-representation
\begin{equation}\label{pdd} g\cdot 1_{\Omega_r}= \sum_k \sum_{Q\in \mathcal{P}^k, Q \subset \Omega_r }  d_Q^r   a_Q,\end{equation}
such that 
\begin{equation}\label{hiip1}  \Big( \sum_j \big(  \sum_r    \sum_{\substack{ Q\in \mathcal{P}^j \\  Q \subset \Omega_r }}  |d_Q^r|^p )^{q/p}  \Big)^{1/q} \leq \Crr{rel}  |g|_{\mathcal{B}^{s}_{p,q}}.\end{equation}
Note that 
$$\Omega=\cup_r\Omega_r$$
is  a $(1-ps, \Crr{domain},\Crr{domainc})$-regular domain and 
 $F(g)=g1_{\Omega}$ is a bounded operator in $\mathcal{B}^{s}_{p,q}\cap L^\infty$ satisfying 
\begin{equation}\label{G} |F|_{\mathcal{B}^{s}_{p,q}\cap L^\infty}\leq  \Crr{e}\Crr{no} \Big( 1+  \frac{\Crr{domain}^{1/p}}{(1-\Crr{domainc}^{q/p})^{1/q}} |\Omega|^{1/p-s}\Big).\end{equation}
Moreover
\begin{equation}\label{estG} |1_\Omega|_{\mathcal{B}^s_{p,q}}\leq  \frac{\Crr{domain}^{1/p}}{(1-\Crr{domainc}^{q/p})^{1/q}} |\Omega|^{1/p-s}.\end{equation}
\end{corollary} 
\begin{proof} Notice that
$$f=1_{\cup_r \Omega_r}= \sum_k  \sum_{Q\in \mathcal{P}^k} c_Q a_Q,$$
where $c_Q=|Q|^{1/p-s}$ for every $Q\in \cup_k \cup_r \mathcal{F}^k(\Omega_r)$ and $c_Q=0$ otherwise.
Let 
$$g =\sum_k  \sum_{J\in \mathcal{P}^k} e_J a_J$$
be  $\mathcal{B}^s_{p,q}$-representations  $g$  given by Corollary \ref{fou}. Consider $u_1, u_2$ as in the proof of Proposition \ref{mult33}. By Proposition \ref{boup} we can get exactly the same estimate as in the proof of Proposition \ref{mult33}.

Note that those  $Q\in \mathcal{P}^k$ for which the corresponding atom $a_Q$  has  a non-vanishing \change{we replaced "no vanishing" by "non-vanishing"}  coefficient in the definition of $u_1$ belongs to $\cup_r \cup_j \mathcal{F}^j(\Omega_r)$, and moreover every $J\in \mathcal{P}^k$ for which the corresponding atom \change{we replaced "atoms" by "atom"}  $a_J$  has  non-vanishing \change{we replaced "no vanishing" by "non-vanishing"}  coefficients in the definition of $u_2$ is contained in some  $Q\in  \mathcal{F}^j(\Omega_r)$, for some $j$ and $r$. In particular $J\subset \Omega_r$.  So (\ref{pdd}) holds, with
$$d_Q^r= \big( \sum_{Q\subset J,  J\in \mathcal{P} } |J|^{s-1/p} c_Q e_J \big)  + \big(\sum_{Q\subset J, J\neq Q,J\in \mathcal{P}} |J|^{s-1/p}  c_J e_Q  \big) $$
for every $Q\subset \Omega_r$.  

Note also that
\begin{eqnarray*}
&&\Big( \sum_k \big(\sum_r \sum_{Q\in \mathcal{F}^k(\Omega_r)} |Q|^{1-sp}\big)^{q/p} \Big)^{1/q}\\
&\leq&  \Crr{domain}^{1/p} \Big( \sum_{k\geq k_0(\cup_r \Omega_r)}   \Crr{domainc}^{(k-k_0(\Omega))q/p} \Big)^{1/q} |\Omega|^{1/p-s}\\
&\leq&  \frac{\Crr{domain}^{1/p}}{(1-\Crr{domainc}^{q/p})^{1/q}} |\Omega|^{1/p-s}.
\end{eqnarray*}
so (\ref{estG})
and consequently (\ref{G}) hold. 
\end{proof}

\begin{remark} Using the methods in Faraco and Rogers \cite{faraco} one can show that  quasiballs in $[0,1]^n$ (and in particular  quasidisks in $[0,1]^2$, that is, domains delimited by quasicircles) give  examples of regular domains in $[0,1]^n$ endowed with the  good grid of dyadic $n$-cubes and the Lebesgue measure $m$.
\end{remark}

\section{A remarkable description of $\mathcal{B}^s_{1,1}$.}

When $p=q=1$ (and $s >0$  small), something curious happens. We can skip the good grid and characterise the Besov space  $\mathcal{B}^s_{1,1}$ of a homogeneous space using regular domains.   Fix $\Crr{domain}\geq 1$ and $\Crr{domainc} \in (0,1)$. Let $\mathcal{W}$ be the family of all $(1-s, \Crr{domain},\Crr{domainc})$-regular domains. Of course $\mathcal{P}\subset \mathcal{W}$. Let $\hat{\mathcal{W}}$ be a family of sets satisfying 
$$\mathcal{P}\subset \hat{\mathcal{W}} \subset \mathcal{W}$$
Define $B^{1-s}$  as the set of all functions $f \in L^{1/(1-s)}$ that can be written as
\begin{equation}\label{ree} f=\sum_{i=0}^{\infty}   c_i \frac{1_{A_i}}{|A_i|^{1-s}},\end{equation}
where $A_i\in  \hat{\mathcal{W}}$ for every $i\in \mathbb{N}$  and 
$$\sum_i |c_i| < \infty.$$
It is easy to see that 
$$|f|_{1/(1-s)}\leq \sum_i |c_i|.$$
Define
$$|f|_{B^{1-s}}=\inf \sum_i |c_i|,$$
where the infimum runs over all possible representations (\ref{ree}). One can see that $(B^{1-s},|\cdot|_{B^{1-s}})$ is a normed vector space.

\begin{proposition} \label{rema}We have that $B^{1-s}=\mathcal{B}^{s}_{1,1}(\mathcal{P})$ and the corresponding norms are equivalent.\end{proposition}
\begin{proof} Note that (\ref{estG})  says that there is $C$ such that if  $A\in  \mathcal{W}$ then   $1_{A} \in \mathcal{B}^s_{1,1}(\mathcal{P})$  and $$|1_A|_{\mathcal{B}^s_{1,1}(\mathcal{P})}\leq C |A|^{1-s}.$$
In particular, \change{we added a comma}  if $f$ has a representation (\ref{ree}) we conclude that
$$|f|_{\mathcal{B}^{s}_{1,1}(\mathcal{P})}\leq  C  |f|_{B^{1-s}}.$$
 In particular $B^{1-s}\subset \mathcal{B}^{s}_{1,1}(\mathcal{P}).$ On the other hand if $g\in \mathcal{B}^{s}_{1,1}(\mathcal{P}).$ then we can  write
$$g = \sum_{k=0}^{\infty} \sum_{Q \in \mathcal{P}^k}    s_Q \frac{1_Q}{|Q|^{1-s}}$$
$$\sum_{P\in \mathcal{P}} |s_Q|= \sum_{k=0}^{\infty}  \sum_{Q \in \mathcal{P}^k}    |s_Q|  < \infty.$$
and $|g|_{\mathcal{B}^{s}_{1,1}(\mathcal{P})}$ is the infimum of $\sum_{P\in \mathcal{P}} |s_Q|$ over all possible representations. In particular  $g\in B^{1-s}$
and
$$|g|_{B^{1-s}}\leq |g|_{\mathcal{B}^{s}_{1,1}(\mathcal{P})}.$$
\end{proof}

\begin{remark} Let $I=[0,1]$ with the dyadic grid $\mathcal{D}$ and the Lebesgue measure $m$. We prove in Part IV that $\mathcal{B}^s_{1,1}(\mathcal{D})$, with $0< s<1$, is the  Besov space $B^s_{1,1}([0,1])$, and its norms are equivalent. Note that every interval $[a,b]\subset [0,1]$ is a $(1-s, 2,2^{s-1})$-regular domain. So we can apply Proposition \ref{rema} with  $\hat{\mathcal{W}}=\{[a,b], \ 0\leq a < b\leq 1\}.$  That is, $f$ belongs to $B^s_{1,1}([0,1])$ if and only if it can be written as in (\ref{ree}), where every $A_i$ is an interval and $\sum_i |c_i| < \infty$, and the norm in $B^s_{1,1}([0,1])$ is equivalent to the infimum  of $\sum_i |c_i|$ over all possible such representations. This characterisation of the Besov space $B^s_{1,1}([0,1])$  was first obtained   by Souza \cite{souzao1}.

\end{remark}

\section{Left compositions.} The following \change{we replaced "folllowing" by "following"}  result generalizes a well-known result  on left composition operators acting  on Besov spaces of $\mathbb{R}^n$.  See  Bourdaud and Kateb \cite{k1}  \cite{k0}\cite{kateb1} for recent  \change{we replaced "recents" by "recent"}  developments on the study of left compositions on  Besov spaces of $\mathbb{R}^n$. 

\begin{proposition} \label{expo} Let 
$$g\colon I \rightarrow \mathbb{C}$$ be a Lipchitz function such that $g(0)=0$. Then the left composition
$$L_g\colon \mathcal{B}^s_{p,q} \rightarrow \mathcal{B}^s_{p,q}$$
defined by $L_g(f)=g\circ f$ is well-defined \change{we replaced "well defined"by "well-defined"} and
$$|g\circ f|_p+ osc^s_{p,q}(g\circ f)\leq K (|f|_p +osc^s_{p,q}(f)),$$
where $K$ is the  Lipchitz constant of $g$. Consequently there exists $C$ such that 
$$|L_g(f)|_{\mathcal{B}^s_{p,q}}\leq C|f|_{\mathcal{B}^s_{p,q}}$$
for every $f\in \mathcal{B}^s_{p,q}$.
\end{proposition}
\begin{proof} Note that 
\begin{align}
osc_p(g\circ f, Q)&=\inf_{a\in \mathbb{C}}\Big(  \int _Q |g(f(x))-a|^p \ dm(x) \Big)^{1/p} \nonumber \\
& \leq \inf_{a\in \mathbb{C}}\Big(  \int _Q |g(f(x))-g(a)|^p \ dm(x) \Big)^{1/p}  \nonumber \\ 
& \leq K  \inf_{a\in \mathbb{C}}\Big(  \int _Q |f(x)-a|^p \ dm(x) \Big)^{1/p} = K osc_p(f, Q).
\end{align}
So it easily follows that $osc^s_{p,q}(g\circ f)\leq K osc^s_{p,q}(f)$. Of course $|g\circ f|_p\leq K|f|_p$.  In particular $g\circ f \in \mathcal{B}^s_{p,q}.$
\end{proof}

\changei{Reference Smania, D. "Classic and Exotic Besov spaces defined by good grids" was published in The Journal of Geometric Analysis. We updated the reference.}
\bibliographystyle{abbrv}
\bibliography{bibliografiab}

\end{document}